\documentclass[11pt, reqno]{amsart}

\usepackage{fullpage}
\usepackage{amsmath}
\usepackage{amsthm}
\usepackage{amsfonts}
\usepackage{amssymb}

\newtheorem{lemma}{Lemma}[section]
\newtheorem{proposition}[lemma]{Proposition}
\newtheorem{theorem}[lemma]{Theorem}
\newtheorem{corollary}[lemma]{Corollary}
\theoremstyle{remark}
\newtheorem{remark}[lemma]{Remark}
\newtheorem{conjecture}[lemma]{Conjecture}
\theoremstyle{definition}
\newtheorem{definition}[lemma]{Definition}
\numberwithin{equation}{section}

\renewcommand{\Re}{\mathrm{Re}\,}
\renewcommand{\Im}{\mathrm{Im}\,}
\newcommand{\mc}{\mathcal}
\newcommand{\R}{\mathbb{R}}
\newcommand{\C}{\mathbb{C}}
\newcommand{\N}{\mathbb{N}}

\newcommand{\X}{{X^{p,\alpha}_\delta}}
\newcommand{\Y}{{Y^{p,\alpha}}}

\title{Nonscattering solutions and blowup at infinity for the critical wave equation}

\author{Roland Donninger}
\address{\'Ecole Polytechnique F\'ed\'erale de Lausanne, 
Department of Mathematics, Station 8, CH-1015 Lausanne, Switzerland}
\email{roland.donninger@epfl.ch}
\thanks{The authors would like to thank T.~Duyckaerts for suggesting this problem.}

\author{Joachim Krieger}
\address{\'Ecole Polytechnique F\'ed\'erale de Lausanne, 
Department of Mathematics, Station 8, CH-1015 Lausanne, Switzerland}
\email{joachim.krieger@epfl.ch}

\begin{document}

\begin{abstract}
We consider the critical focusing wave equation $(-\partial_t^2+\Delta)u+u^5=0$ in $\R^{1+3}$ and
prove the existence of energy class solutions which are of the form
\[ u(t,x)=t^\frac{\mu}{2}W(t^\mu x)+\eta(t,x) \]
in the forward lightcone $\{(t,x)\in\R\times \R^3: |x|\leq t, t\gg 1\}$ 
where $W(x)=(1+\frac13 |x|^2)^{-\frac12}$ is the ground state soliton, 
$\mu$ is an arbitrary prescribed real number (positive or negative) with $|\mu|\ll 1$, and the
error $\eta$ satisfies
\[ \|\partial_t \eta(t,\cdot)\|_{L^2(B_t)}
+\|\nabla \eta(t,\cdot)\|_{L^2(B_t)}\ll 1,\quad B_t:=\{x\in\R^3: |x|<t\} \] for
all $t\gg 1$.
Furthermore, the kinetic energy of $u$ outside the cone is small.
Consequently, depending on the sign of $\mu$, 
we obtain two new types of solutions which either concentrate as $t\to\infty$ (with a continuum of rates)
or stay bounded but do not scatter.
In particular, these solutions contradict a strong version of the soliton resolution conjecture.
\end{abstract}

\maketitle

\section{Introduction}
\label{sec:intro}

In this paper we study the critical focusing wave equation
\begin{equation}
\label{eq:mainintro}
(-\partial_t^2 +\Delta)u(t,x)+u(t,x)^5=0 
\end{equation}
for $u: I\times \R^3\to\R$, $I\subset \R$ an interval.
It is well-known that the Cauchy problem for Eq.~\eqref{eq:mainintro} 
is well-posed for data in the energy space $\dot{H}^1\times L^2(\R^3)$, see 
e.g.~\cite{LS95}, \cite{Sog08}.
Furthermore, Eq.~\eqref{eq:mainintro} admits a static solution $W$, the \emph{ground state
soliton} given by $W(x)=(1+\frac13 |x|^2)^{-\frac12}$, which indicates the presence of interesting
dynamics.
Our main result is the following.

\begin{theorem}
\label{thm:main}
There exists an $\varepsilon_0>0$ such that for any $\delta>0$ and $\mu \in \R$ with $|\mu|\leq \varepsilon_0$ there exists
a $t_0\geq 1$ and an energy class solution $u: [t_0,\infty)\times \R^3\to \R$ of Eq.~\eqref{eq:mainintro}
of the form
\[ u(t,x)=t^{\frac{\mu}{2}}W(t^\mu x)+\eta(t,x),\quad |x|\leq t,\: t\geq t_0 \]
and
\begin{align*} 
\|\partial_t u(t,\cdot)\|_{L^2(\R^3\backslash B_t)}+\|\nabla u(t,\cdot)\|_{L^2(\R^3\backslash B_t)}&\leq \delta,
\\
\|\partial_t \eta(t,\cdot)\|_{L^2(B_t)}+\|\nabla \eta(t,\cdot)\|_{L^2(B_t)}&\leq \delta 
\end{align*}
for all $t\geq t_0$ where $B_t:=\{x\in\R^3: |x|<t\}$.
\end{theorem}

The Cauchy problem for Eq.~\eqref{eq:mainintro} has attracted a lot of interest in the recent past
and we briefly review the most important contributions.
Eq.~\eqref{eq:mainintro} is invariant under the scaling transformation
\begin{equation}
\label{eq:scaling}
u(t,x)\mapsto u^\lambda(t,x):=\lambda^\frac12 u(\lambda t, \lambda x),\quad \lambda>0 
\end{equation}
and its conserved energy 
\[ E(u(t,\cdot), u_t(t,\cdot))=\tfrac12 \|(u(t,\cdot),u_t(t,\cdot))\|_{\dot{H}^1\times L^2(\R^3)}^2
-\tfrac16 \|u(t,\cdot)\|_{L^6(\R^3)}^6 \]
satisfies $E(u^\lambda(t/\lambda,\cdot), u_t^\lambda(t/\lambda, \cdot))=E(u(t,\cdot),u_t(t,\cdot))$
which is why Eq.~\eqref{eq:mainintro} is called \emph{energy critical}.
Historically, the investigation of the global Cauchy problem for 
energy critical wave equations started with the \emph{defocusing case},
\[ (-\partial_t^2+\Delta)u(t,x)-u(t,x)^5=0, \]
where the sign of the nonlinearity is reversed compared to Eq.~\eqref{eq:mainintro}.
After the pioneering works \cite{Rau81} and \cite{Pec84}, the development culminated in a proof of global existence
and scattering for arbitrary data \cite{Str88}, \cite{Gri90}, \cite{GSV92}, \cite{SS93}, \cite{SS94},
\cite{Kap94}, \cite{BS98}, see also \cite{Tao06}.
However, the dynamics in the focusing case are much more complicated.
For instance, it is well-known that there exist solutions with compactly supported smooth initial data which 
blow up in finite time. 
This is most easily seen by observing that
\[ u(t,x)=(\tfrac34)^\frac14 (1-t)^{-\frac12} \]
is an explicit solution which, by finite speed of propagation, can be used to construct a blowup
solution of the aforementioned type.
This kind of breakdown is referred to as \emph{ODE blowup} and it is conjectured to comprise the 
``generic'' blowup
scenario \cite{BCT04}.
We remark in passing that (parts of) this conjecture have been proved for the \emph{subcritical case}
\[ (-\partial_t^2+\Delta ) u+u|u|^{p-1}=0,\quad p\in (1,3], \]
see \cite{MZ03}, \cite{MZ05}, \cite{DS12},
but for $p>3$ the problem is largely open (see, however, \cite{DonSch12a}).
Another (less explicit but classical) argument to obtain finite time blowup for focusing wave equations 
is due to Levine \cite{Lev74}.
Recently, Krieger, Schlag and Tataru \cite{KST09} constructed in the energy critical case $p = 5$ more ``exotic'' blowup solutions 
of the form \footnote{The existence of ground state solitons, i.e., positive static solutions with finite
energy such as $W$ requires $p=5$ in spatial dimension 3, cf.~\cite{GS81}, \cite{JPY93}.}
\[ u(t,x)=(1-t)^{-\frac12(1+\nu)} W((1-t)^{-1-\nu}x)+\eta(t,x),\quad |x|\leq 1-t \]
where $\nu>\frac12$ can be prescribed arbitrarily and 
$\eta$ is small in a suitable sense.
As a matter of fact, the proof of Theorem \ref{thm:main} makes extensive use of the techniques
developed in \cite{KST09}, see also \cite{KST08} and \cite{KST09a}
for analogous results in the case of critical wave maps and Yang-Mills equations.
In this respect we also mention another construction of blowup solutions for the critical wave equation by 
Hillairet and Rapha\"el \cite{HR10}, albeit
for the higher dimensional case $\R^{1+4}$.
Furthermore, Duyckaerts, Kenig and Merle \cite{DKM11}, \cite{DKM12} 
showed that any type II blowup solution \footnote{
A type II blowup solution stays bounded in the energy space. The solutions constructed in 
\cite{KST09} are of this type.}
which satisfies a suitable smallness condition decomposes
into a rescaled ground state soliton plus a small remainder.

Apart from the construction of blowup solutions, it is of interest to obtain conditions on the initial
data under which the solution exists globally.
As a consequence of Strichartz estimates it is relatively easy to establish global existence
and scattering for data with small energy, see \cite{Pec84}, \cite{Sog08}.
However, for energies close to the ground state the situation becomes much more involved.
Krieger and Schlag \cite{KS07} proved the existence of a small codimension one manifold in the 
space of initial data, containing $(W,0)$, which leads to solutions of the form
\[ u(t,x)=\lambda(t)^\frac12 W(\lambda(t)x)+\eta(t,x) \]
where $\lambda(t)\to a>0$ as $t\to \infty$ and $\eta$ scatters like a free wave.
In other words, the solutions arising from data on this manifold exist globally
and scatter to a rescaling of the ground state soliton, see also \cite{Spa01} for numerical
work in this direction.
A different line of investigation was pursued by Kenig and Merle \cite{KM08} who 
established the following celebrated dichotomy.

\begin{theorem}[Kenig-Merle \cite{KM08}]
\label{thm:KM}
Let $u$ be an energy class solution to Eq.~\eqref{eq:mainintro} with 
\[ E(u(0,\cdot),u_t(0,\cdot))<E(W,0). \]
\begin{itemize}
\item
If $\|u(0,\cdot)\|_{\dot{H}^1(\R^3)}<\|W\|_{\dot{H}^1(\R^3)}$ then the solution $u(t,x)$ exists
for all $t\in\R$ and scatters like a free wave as $t\to \pm \infty$.
\item If
$\|u(0,\cdot)\|_{\dot{H}^1(\R^3)}>\|W\|_{\dot{H}^1(\R^3)}$ then the solution $u(t,x)$ blows up
in finite time in both temporal directions.
\end{itemize}
\end{theorem}
Theorem \ref{thm:KM} was extended by Duyckaerts and Merle \cite{DM07} to include the case
$E(u(0,\cdot),u_t(0,\cdot))=E(W,0)$ which, in addition to the possibilities of Theorem \ref{thm:KM}, entails
solutions which scatter towards (a rescaling of) $W$.
We also refer the reader to the recent works by Krieger, Nakanishi and Schlag \cite{KNS10}, \cite{KNS11} where they 
consider
data with energies slightly above the ground state.
Based on the results in \cite{KST09}, \cite{KS07}, \cite{KM08} and \cite{DM07} it seemed plausible
to expect a strong version of the \emph{soliton resolution conjecture} to hold.
Roughly speaking, this conjecture states that the long time evolution 
splits into a finite sum of solitons plus radiation, see \cite{Sof06}.

\begin{conjecture}[Strong soliton resolution at energies close to the ground state]
\label{conj:solres}
Any radial energy class solution of Eq.~\eqref{eq:mainintro} with energy close to $E(W,0)$
either blows up in finite time or scatters to zero like a free wave or scatters towards a rescaling of $W$. 
\end{conjecture}

Our Theorem \ref{thm:main}, however, shows that Conjecture \ref{conj:solres} is wrong.
In addition to the already known dynamics
\begin{itemize}
\item $\lambda(t)\to \infty$ as $t\to 1-$ (exotic blowup \cite{KST09})
\item $\lambda(t)\to a>0$ as $t\to \infty$ (scattering towards (a rescaling of) $W$ \cite{DM07}, 
\cite{KS07})
\end{itemize}
for solutions of the form
$u(t,x)=\lambda(t)^\frac12 W(\lambda(t)x)+\eta(t,x)$ with $\eta$ small, 
our result adds the two new possibilities 
\begin{itemize}
\item $\lambda(t)\to 0$ as $t\to \infty$ (``vanishing'')
\item $\lambda(t)\to \infty$ as $t\to\infty$ (``blowup at infinity'')
\end{itemize}
and either of which contradicts Conjecture \ref{conj:solres}.
Furthermore, there exists a continuum of rates at which the blowup (or vanishing) occurs.
It has to be remarked that the blowup does not take place in the energy norm (which stays bounded)
but in a type II fashion (i.e., only higher order norms blow up).
We also note that, although Conjecture \ref{conj:solres} in this sharp form does not hold,
a weaker version in the radial context was proved after the submission of the present paper \cite{DuyKenMer12b}.
The result in \cite{DuyKenMer12b} shows in particular that the solutions we 
construct are in a sense the only global solutions which do not scatter.
We also remark in passing that we expect the solutions of Theorem \ref{thm:main} to be smooth and therefore, 
unlike in the case of the exotic blowup in \cite{KST09}, there is no conjectured ``quantization'' 
of blowup rates as one passes to smooth solutions.
The intuitive reason for this is that we do not encounter a singularity at the lightcone as in \cite{KST09}
since we cut off our approximate solutions in such a way that they are supported 
inside the smaller cone $r\leq t-c$ for some constant $c$.
At the moment, however, we
cannot rigorously prove the smoothness of the full solutions since parts of our construction rely on a ``soft'' argument which only
yields energy class regularity.
Furthermore, it is worth mentioning that the technique used to prove Theorem~\ref{thm:main} appears to have much 
wider applicability for similar critical nonlinear problems, as did the construction in \cite{KST09}.
We also note that the restrictions on $\mu$ in our Theorem~\ref{thm:main} seem to be in part 
technical and nonessential. Thus, it appears to be natural to expect at least the range 
$-1<\mu<\varepsilon_0$ to be allowable, for suitable $\varepsilon_0>0$. 

While this paper was being written up, T.~Duyckaerts informed the authors of  the following result, obtained jointly with C. Kenig and F. Merle, 
which nicely combines with our Theorem \ref{thm:main} in a similar way as 
\cite{DKM11} and \cite{DKM12} are related to \cite{KST09}.

\begin{theorem}[Duyckaerts-Kenig-Merle \cite{DKM12b}, private communication by T. Duyckaerts]
If $u: [0,\infty)\times \R^3\to\R$ is a radial energy class solution of Eq.~\eqref{eq:mainintro} with
\[ \limsup_{t\to\infty}\left (\|\partial_t u(t,\cdot)\|_{L^2(\R^3)}+\|\nabla u(t,\cdot)\|_{L^2(\R^3)} \right )<
2 \|\nabla W\|_{L^2(\R^3)} \]
then (up to a change of sign)
\begin{align*} 
u(t,x)&=\lambda(t)^\frac12 W(\lambda(t)x)+v(t,x)+o_{\dot{H}^1(\R^3)}(1) \\
\partial_t u(t,x)&=\partial_t v(t,x)+o_{L^2(\R^3)}(1) 
\end{align*}
with $v$ a free wave and $t\lambda(t)\to \infty$ as $t \to \infty$.
\end{theorem}

Theorem \ref{thm:main} should also be contrasted to previous works on other dispersive systems such as the
nonlinear Schr\"odinger equation.
We cannot do justice to the vast literature on this subject but
as an example we mention Tao's result \cite{Tao04} on the cubic focusing Schr\"odinger equation
in $\R^{1+3}$
which states that radial solutions which exist globally decouple into a smooth function localized near the origin, a 
radiative term, and an error that goes to zero as $t \to \infty$.
Unlike Theorem \ref{thm:main}, this result is in concordance with the soliton resolution conjecture. 
We refer the reader to
\cite{Sof06} and the references therein for more positive results in this direction.
Furthermore, the only system (to our knowledge) of ``wave type'' (i.e., 
either nonlinear wave or Schr\"odinger equation)  for which nonscattering solutions similar to ours
are known is the $L^2$-critical nonlinear Schr\"odinger equation in $\R^{1+1}$. 
For this system nonscattering solutions can be constructed by combining
the ``log log'' blowup of \cite{Per01} with the pseudo-conformal symmetry (the ``log log'' blowup
exists in other dimensions as well, see \cite{MerRap04}, \cite{MerRap05}, \cite{MerRap06}).
However, it is evident that the mechanism which furnishes these solutions is completely different and not related
to the situation here.

\subsection{A roadmap of the proof}
We give a brief overview of the proof of Theorem \ref{thm:main} without going into technical details.
As already mentioned, the construction is in parts based on the techniques developed in  \cite{KST08}, \cite{KST09}, 
\cite{KST09a}.
However, in order to deal with the present situation, the method has to be modified and extended considerably.
In the following we restrict ourselves to radial functions and the symbols $r$ and $|x|$ are used
interchangeably.
By a slight abuse of notation we also write $u(t,r)$ instead of $u(t,x)$ meaning that $u(t,\cdot)$
is a radial function.
Furthermore, throughout this paper we set
\footnote{We use this 
convention for ``historical'' reasons, cf.~\cite{KST09}.}
$\lambda(t):=t^{-(1-\nu)}$ where $\nu$ 
is assumed to be close to $1$.
Roughly speaking, the proof splits into three main parts which we now describe in more detail.

\begin{enumerate}
\item Construction of ``elliptic profile modifiers''. 
An obvious idea is to
insert the naive ansatz $u(t,r)=\lambda(t)^\frac12 W(\lambda(t)r)+\eta(t,r)$ into Eq.~\eqref{eq:mainintro}
and to derive an equation for $\eta$.
By doing so, however, one produces an error $\partial_t^2 [\lambda(t)^\frac12 W(\lambda(t)r)]$ which 
decays roughly like
$t^{-2}$ and this turns out to be insufficient.
Consequently, we first modify the profile $\lambda(t)^\frac12 W(\lambda(t)r)$ by a nonperturbative
procedure.
This is done in two steps where we solve suitable (linear) approximations to 
Eq.~\eqref{eq:mainintro} and thereby improve the error at the center and, in the second step, 
near the lightcone $r \approx t$.
This does not yield an actual solution but a function $u_2$ which solves Eq.~\eqref{eq:mainintro}
only up to an error.
However, this error now decays approximately like $t^{-4}$ and thus, we have gained two powers which
is sufficient to proceed. 
This is in contrast to the analogous procedure in \cite{KST09} where a very large number of 
modifications are added to the ground state. In our situation, it turns out that additional 
modifications do not improve the error further. 
The improvement in decay comes at the expense of differentiability at the lightcone which has to be accounted
for by using suitable cut-offs. Thus, in this first stage of the construction, we only obtain an 
approximate solution in a smaller forward light cone (i.e., of the form $|x|\leq t-c$ and hence
strictly contained inside the standard light cone $|x|\leq t$), in contrast to the procedure in \cite{KST09}. 

\item In a second step we insert the ansatz $u=u_2+\varepsilon$ into Eq.~\eqref{eq:mainintro}
and derive an equation for $\varepsilon$ which is of the form
\begin{equation}
\label{eq:epsintro}
(-\partial_t^2+\Delta) \varepsilon+5W_{\lambda(t)}^4\varepsilon=\mbox{nonlinear terms+error}
\end{equation}
where $W_\lambda(x):=\lambda^\frac12 W(\lambda x)$.
Due to the aforementioned lack of smoothness, the right-hand side of the equation is only defined in a forward lightcone
and for the moment we restrict ourselves to this region.
In order to obtain a time-independent potential on the left-hand side we use
$R:=\lambda(t)r$ as a spatial coordinate and, with an appropriate new time coordinate $\tau$,
Eq.~\eqref{eq:epsintro} transforms into
\begin{equation}
\label{eq:eps2intro}
\mc D^2 \varepsilon+c\tau^{-1}\mc{D}\varepsilon+(-\Delta+V)\varepsilon=\mbox{nonlinear terms+error}
\end{equation}
where $\mc{D}$ is a first order transport-type operator and $V=-5W^4$.
In order to solve Eq.~\eqref{eq:eps2intro} we apply the ``distorted Fourier transform'' relative
to the self-adjoint operator $-\Delta+V$.
This requires a careful spectral analysis which is only feasible since we are in the radial case.
The existence of a zero energy resonance plays a prominent role here.
As a result we obtain a transport-type equation for the Fourier coefficients which is then solved
by the method of characteristics combined with a fixed point argument.
The treatment of the error terms on the right-hand side of the transport equation is delicate
and requires a good amount of harmonic analysis. In particular, the functional framework employed differs from that in \cite{KST09}. 

\item The last step consists of a partially ``soft'' argument which is used to extend the solution
to the whole space and, second, to extract suitable initial data that lead to the desired solution 
(recall that we have solved the equation in a forward lightcone where the Cauchy problem
is not well-posed).
For this we rely on a concentration-compactness approach based on the celebrated 
Bahouri-G\'erard decomposition \cite{BG99}. 
\end{enumerate}

\subsection{Notation}
We write $\N$ for the natural numbers $\{1,2,3,\dots\}$ and set $\N_0:=\{0\}\cup \N$.
We use standard Lebesgue and (fractional) Sobolev spaces denoted by 
$L^p(\Omega)$, $W^{s,p}(\Omega)$ and $H^s:=W^{s,2}$ with $\Omega \subset \R^d$.
Our sign convention for the wave operator is $\Box:=-\partial_t^2+\Delta$.
Unless otherwise stated, the letter $C$ (possibly with indices) denotes a positive constant which may change from line to line.
As usual, we write $a\lesssim b$ if $a\leq Cb$ and if the constant $C$ has to be sufficiently large,
we indicate this by $a \ll b$. Similarly, we use $a \gtrsim b$ and $a\simeq b$ means $a\lesssim b$ and $b\lesssim a$.
Furthermore, we reiterate that $\lambda(t):=t^{-(1-\nu)}$ with $\nu$ a real number close to $1$.
Throughout the paper, $\nu$ is supposed to be fixed and sufficiently close to $1$.
Note also that Theorem \ref{thm:main} is trivial if $\nu=1$ since in this case $u(t,x)=W(x)$.
Thus, whenever convenient we exclude the case $\nu=1$ without further notice.
For $x\in\R^d$ we set $\langle x \rangle:=\sqrt{1+|x|^2}$ and write $O(f(x))$ to denote a generic 
\emph{real-valued} function which satisfies $|O(f(x))|\lesssim |f(x)|$ in a domain of $x$ that
is either specified explicitly or follows from the context. If the function attains
complex values as well we indicate this by a subscript, e.g.~$O_\C(x)$.
An $O$-term $O(x^\gamma)$, where $x,\gamma \in \R$, is said to behave like a symbol if 
$|\partial_x^k O(x^\gamma)|\leq C_k |x|^{\gamma-k}$ for all $k\in\N$.
A similar definition applies to symbol behavior of $O(\langle x \rangle^\gamma)$ with $|\cdot|$ 
substituted by $\langle \cdot \rangle$.

\section{Construction of an approximate solution}
\label{sec:approx}

Our intention is to construct a solution $u$ of the form $u(t,r)=\lambda(t)^\frac12
W(\lambda(t)r)+\eta(t,r)$ with $\lambda(t)=t^{-(1-\nu)}$ where $\nu$ is sufficiently close 
to $1$ and $\eta$ is small in a suitable sense.
We first improve the approximate solution $W_{\lambda(t)}(r):=\lambda(t)^\frac12 W(\lambda(t)r)$ by successively adding
two correction terms $v_0$ and $v_1$.
These corrections are obtained by approximately solving the equation in a way we describe in
the following.

\subsection{Improvement at the center}

We set $u_0(t,r):=W_{\lambda(t)}(r)$ and define the first error $e_0$ by
$$ e_0:=\Box u_0+u_0^5=-\partial_t^2 u_0 $$
with $\Box=-\partial_t^2+\Delta$.

\begin{lemma}
\label{lem:e0}
The error $e_0$ is of the form
$$ t^2 e_0(t,r)=c_\nu \lambda(t)^\frac12 \langle \lambda(t)r\rangle^{-1}+t^2 e_0^*(t,r) $$
where $c_\nu$ is a real constant and $e_0^*$ satisfies the bounds
$$ |\partial_t^\ell \partial_r^k 
t^2 e_0^*(t,r)|\leq C_{k,\ell} \lambda(t)^{\frac12+k+\ell} [\lambda(t)t]^{-\ell}
\langle \lambda(t)r\rangle^{-3-k} $$
for all $t \geq t_0>0$, $r \geq 0$ and $k,\ell \in \mathbb{N}_0$.
In addition, we have 
$$\partial_r^{2k+1}e_0(t,r)|_{r=0}=\partial_r^{2k+1}e_0^*(t,r)|_{r=0}=0. $$
\end{lemma}

\begin{proof}
Note first that $W(R)=\sqrt{3}\langle R \rangle^{-1}+W^*(R)$ where
$|\partial_R^k W^*(R)|\leq C_k \langle R \rangle^{-3-k}$
for all $R\geq 0$ and $k \in \mathbb{N}_0$.
Consequently, the claim follows from 
$$e_0(t,r)=-\partial_t^2 u_0(t,r)=-\partial_t^2 [\lambda(t)^\frac12 W(\lambda(t)r)]$$
by the chain rule.
\end{proof}

Note that the decay of $e_0(t,r)$ near the lightcone $r=t$ is better than at the
center.
Consequently, we first attempt to improve the approximation near $r=0$.
Ideally, we would like to add a correction $v_0$ such that $u_1:=u_0+v_0$ becomes an exact solution, 
i.e., 
$$ 0=\Box u_1+u_1^5=\Box v_0+5 u_0^4 v_0+N(u_0,v_0)+e_0 $$
where 
$$ N(u_0,v_0):=10 u_0^3 v_0^2+10 u_0^2 v_0^3 + 5 u_0 v_0^4 +v_0^5. $$
Near the center $r=0$ we expect the time derivative to be less important and therefore we neglect
it altogether and also drop the nonlinearity to obtain the approximate equation
\begin{equation}
\label{def:v0}
\Delta v_0+5 u_0^4 v_0=-e_0
\end{equation}
and the next error $e_1$ is defined as
\begin{equation}
\label{def:e1}
e_1:=\Box u_1+u_1^5.
\end{equation}
We solve Eq.~\eqref{def:v0} for $v_0$ and subsequently show that $e_1$ decays faster than $e_0$.

\begin{lemma}
\label{lem:e1}
There exists a function $v_0$ satisfying Eq.~\eqref{def:v0} such that
$$ v_0(t,r)=c_\nu \lambda(t)^{\frac12}[\lambda(t)t]^{-2}\lambda(t)r+d_\nu \lambda(t)^\frac12 [\lambda(t)t]^{-2}+v_0^*(t,r) $$
where $c_\nu, d_\nu$ are real constants and
$$ |\partial_t^\ell \partial_r^k v_0^*(t,r)|\leq C_{k,\ell} 
\lambda(t)^{\frac12+k+\ell}
[\lambda(t)t]^{-2-\ell}\langle \lambda(t) r \rangle^{-1-k}\left \langle \log \langle \lambda(t)r \rangle \right \rangle$$
for all $t\geq t_0>0$, $r \geq 0$ and $k,\ell \in \mathbb{N}_0$.
As a consequence, the error $e_1$ defined by Eq.~\eqref{def:e1} is of the form
$$ t^2 e_1(t,r)=c_\nu \lambda(t)^{\frac12}[\lambda(t)t]^{-2}\lambda(t)r+d_\nu \lambda(t)^\frac12 [\lambda(t)t]^{-2}+t^2 e_1^*(t,r) $$
with (different) real constants $c_\nu, d_\nu$ and $e_1^*$ satisfies the bounds
$$ |\partial_t^\ell \partial_r^k t^2 e_1^*(t,r)|\leq C_{k,\ell}\lambda(t)^{\frac12+k+\ell}
[\lambda(t)t]^{-2+\epsilon-\ell}
\langle \lambda(t)r\rangle^{-1-k} $$
for all $t\geq t_0>0$, $0\leq r\lesssim t$, any (fixed) $\epsilon>0$, and $k,\ell \in \mathbb{N}_0$.
In addition, we have $\partial_r^{2k+1}w(t,r)|_{r=0}=0$ where $w \in \{v_0,v_0^*,e_1,e_1^*\}$.
\end{lemma}

\begin{proof}
Setting $\tilde{v}_0(t,R):=Rv_0(t,\lambda(t)^{-1}R)$ and $R:=\lambda(t)r$, Eq.~\eqref{def:v0} reads
\begin{equation}
\label{eq:tildev0}
\partial_R^2 \tilde{v}_0(t,R)+5 W(R)^4 \tilde{v}_0(t,R)=-\lambda(t)^{-2}Re_0(t,\lambda(t)^{-1}R) 
\end{equation}
which is an inhomogeneous ODE in $R$ and $t$ can be treated as a parameter.
Explicitly, the potential reads
$$ 5W(R)^4=\frac{5}{(1+\frac{R^2}{3} )^2} $$
and the homogeneous equation has the fundamental system $\{\phi_0,\theta_0\}$ given by
\begin{align*}
\phi_0(R)&=R(1-\tfrac{R^2}{3})(1+\tfrac{R^2}{3})^{-\frac32}=-\sqrt{3}+\phi_0^*(R) \\
\theta_0(R)&=(1+\tfrac{R^2}{3})^{-\frac32}(1-2 R^2+\tfrac{R^4}{9})=\tfrac{1}{\sqrt{3}}R+\theta_0^*(R) 
\end{align*}
where
$$ |\partial_R^k \phi_0^*(R)|\leq C_k \langle R \rangle^{-2-k},\quad
|\partial_R^k \theta_0^*(R)|\leq C_k \langle R \rangle^{-1-k} $$
for all $R\geq 0$ and $k\in \mathbb{N}_0$.
Furthermore, for the Wronskian we 
obtain $W(\theta_0,\phi_0)=\theta_0 \phi_0'-\theta_0' \phi_0=1$.
According to the variation of constants formula, a solution $\tilde{v}_0$ of Eq.~\eqref{eq:tildev0}
is therefore given by
\begin{align*}
\tilde{v}_0(t,R)=&-\lambda(t)^{-2}\phi_0(R)\int_0^R \theta_0(R')R'e_0(t,\lambda(t)^{-1}R')dR' \\
&+\lambda(t)^{-2}\theta_0(R)\int_0^R \phi_0(R')R'e_0(t,\lambda(t)^{-1}R')dR' 
\end{align*}
and by using Lemma \ref{lem:e0} 
we obtain the claim concerning $v_0$.

The assertion for $e_1$ now follows from $e_1=-\partial_t^2 v_0+N(u_0,v_0)$.
We have
$$ e_1(t,r)=c_\nu \lambda(t)^\frac12 [\lambda(t)t]^{-2}t^{-2}\lambda(t)r+d_\nu \lambda(t)^\frac12[\lambda(t)t]^{-2}t^{-2}
-\partial_t^2 v_0^*(t,r)
+N(u_0,v_0) $$
and $|\partial_t^2 v_0^*(t,r)|\lesssim \lambda(t)^{\frac12}[\lambda(t)t]^{-2+\epsilon}t^{-2}
\langle \lambda(t)r\rangle^{-1}$ where the $\epsilon$-loss comes from the logarithm.
The nonlinear contributions are of higher order and belong to $e_1^*$ since
$$ |N(u_0,v_0)(t,r)|\lesssim \lambda(t)^\frac52 [\lambda(t)t]^{-4}\langle \lambda(t)r\rangle^{-1} $$
where we use $[\lambda(t)t]^{-1}\lesssim \langle \lambda(t)r\rangle^{-1}$ for $0\leq r\lesssim t$.
The derivative bounds follow from the corresponding bounds on $v_0^*$ by the Leibniz rule.
\end{proof}

\subsection{Improvement near the lightcone}
We have to go one step further and continue improving our approximate solution.
Thus, we add another correction $v_1$ to $u_1$ and set $u_2:=u_1+v_1=u_0+v_0+v_1$.
This yields
\begin{align*} 
\Box u_2+u_2^5&=\Box v_1+5 u_1^4 v_1+N(u_1,v_1)+\Box u_1+u_1^5 \\
&=\Box v_1+5u_1^4 v_1+N(u_1,v_1)+e_1. 
\end{align*}
This time we intend to improve the approximate solution near the lightcone $r=t$ since the decay of the error
$e_1$ near the center is already good enough.
Of course, near the lightcone we cannot ignore the temporal derivative but thanks to
the decay of
$u_1(t,t)$ it turns out that we may safely neglect the 
potential and the
nonlinearity.
Furthermore, we also ignore the higher order error $e_1^*$ which already decays fast enough. 
Thus, we arrive at
the approximate equation
\begin{equation}
\label{eq:defv1}
\Box v_1(t,r)=-c_\nu \lambda(t)^{-\frac32}t^{-4} \lambda(t)r-d_\nu \lambda(t)^{-\frac32}t^{-4}
\end{equation}
and the next error $e_2$ is given by
\begin{equation}
\label{eq:defe2}
e_2:=\Box u_2+u_2^5.
\end{equation}
Note carefully in Lemma \ref{lem:v1} below that in order to gain decay in $t$ we sacrifice differentiability at
the lightcone.

\begin{lemma}
\label{lem:v1}
There exists a solution $v_1=v_{11}+v_{12}$ of Eq.~\eqref{eq:defv1} and a decomposition 
$v_{1j}=v_{1j}^g+v_{1j}^b$, $j=1,2$,  such that
\begin{align*}
|v_{11}(t,r)|&\lesssim \lambda(t)^{-\frac12}t^{-4}r^3, &
|v_{12}(t,r)|&\lesssim \lambda(t)^{-\frac32}t^{-4}r^2 &  (r&\leq \tfrac12 t) \\
|v_{11}^g(t,r)|&\lesssim \lambda(t)^{-\frac12}t^{-1}, &
|v_{12}^g(t,r)|&\lesssim \lambda(t)^{-\frac32}t^{-2} & (r&\leq 2t) \\
|v_{11}^b(t,r)|&\lesssim \lambda(t)^{-\frac12}t^{-1}(1-\tfrac{r}{t})^{\frac12(1-\nu)}, &
|v_{12}^b(t,r)|&\lesssim \lambda(t)^{-\frac32}t^{-2}(1-\tfrac{r}{t})^{\frac12(1-3\nu)} &(r&<t)
\end{align*}
for all $t\geq t_0>0$ and estimates for the derivatives follow by formally differentiating these bounds.
As a consequence, the error $e_2$ as defined by Eq.~\eqref{eq:defe2}, satisfies the bound
\begin{align*} 
|t^2 e_2(t,r)|&\lesssim \lambda(t)^\frac12
[\lambda(t)t]^{-2+\epsilon} t^{\frac52|1-\nu|}\langle \lambda(t)r\rangle^{-1} 
\end{align*}
for $0<r<t-c$, all $t\geq t_0>c_0\geq c$ (where $c_0$ is a fixed constant), and for any (fixed) $\epsilon>0$.
Finally, we have $\partial_r^{2k}v_{11}(t,r)|_{r=0}=\partial_r^{2k+1}v_{12}(t,r)|_{r=0}=0$, $k\in \N_0$.
\end{lemma}

\begin{proof}
Instead of solving Eq.~\eqref{eq:defv1} directly, we set $v_1=v_{11}+v_{12}$ and consider
\[ \Box v_{11}=-c_\nu \lambda(t)^{-\frac32}t^{-4} \lambda(t)r,\quad
\Box v_{12}=-d_\nu \lambda(t)^{-\frac32}t^{-4} \]
separately.
In order to reduce these equations to ODEs, we use the self-similar coordinate $a=\frac{r}{t}$.
We start with
$$ t^2 \Box v_{11}(t,r)=-c_\nu \lambda(t)^{-\frac12} t^{-1}\tfrac{r}{t} $$
and make the self-similar ansatz
$$ v_{11}(t,r)=\lambda(t)^{-\frac12} t^{-1}\tilde{v}_{11}(\tfrac{r}{t}) $$
which yields
\begin{equation}
\label{eq:odea}
(1-a^2)[\tilde{v}_{11}''(a)+\tfrac{2}{a}\tilde{v}_{11}'(a)]-(1+\nu)a\tilde{v}_{11}'(a)
-[\tfrac12(1-\nu)-1][\tfrac12(1-\nu)-2]\tilde{v}_{11}(a)=-c_\nu a. 
\end{equation}
The homogeneous equation has the fundamental system $\{\theta_\pm\}$ given by
$$ \theta_\pm(a)=\tfrac{1}{a}(1\pm a)^{\frac12(1-\nu)} $$
with the Wronskian
$$ W(\theta_+,\theta_-)(a)=\frac{\nu-1}{a^2(1-a^2)^{\frac12(1+\nu)}}. $$
Furthermore, $\psi:=\theta_--\theta_+$ is another solution of the homogeneous equation which is
smooth at $a=0$ and clearly, $W(\theta_+,\psi)=W(\theta_+,\theta_-)$.
Consequently, a solution $\tilde{v}_{11}$ 
of Eq.~\eqref{eq:odea} is given by
\begin{equation} 
\label{eq:tildev1}
\tilde{v}_{11}(a)=\frac{c_\nu \theta_+(a)}{\nu-1}\int_0^a b^3(1-b^2)^{\frac12(\nu-1)} 
\psi(b)db 
-\frac{c_\nu \psi(a)}{\nu-1}\int_0^a b^3(1-b^2)^{\frac12(\nu-1)} 
\theta_+(b)db
\end{equation}
and it follows that 
$\tilde{v}_{11}(a)=O(a^3)$ as $a\to 0$.
Thus, we obtain $v_{11}(t,r)=v_{11}^g(t,r)+v_{12}^b(t,r)$ with
the stated bounds.
Next, we turn to the equation $t^2\Box v_{12}=-d_\nu\lambda(t)^{-\frac32}t^{-2}$.
Here, we make the ansatz $v_{12}(t,r)=\lambda(t)^{-\frac32}t^{-2}\tilde v_{12}(\frac{r}{t})$
which yields Eq.~\eqref{eq:odea} but with $\nu$ replaced by $3\nu$.
Thus, we obtain $\tilde v_{12}(a)=O(a^2)$ as $a\to 0$ and $v_{12}=v_{12}^g+v_{12}^b$ with the claimed bounds.

By construction, the error $e_2$ is given by
$$ e_2=5u_1^4 v_1+N(u_1,v_1)+e_1^* $$
and from Lemma \ref{lem:e1} we recall the bounds
\begin{align*}
|u_1(t,r)|&\leq |u_0(t,r)|+|v_0(t,r)|\lesssim \lambda(t)^{\frac12}\langle \lambda(t)r\rangle^{-1}
+\lambda(t)^{\frac12}
[\lambda(t)t]^{-2}\lambda(t)r \\
&\lesssim \lambda(t)^\frac12 \langle \lambda(t)r\rangle ^{-1}\\
|e_1^*(t,r)|&\lesssim \lambda(t)^{\frac12}
[\lambda(t)t]^{-2+\epsilon}t^{-2}\langle \lambda(t)r \rangle^{-1}.
\end{align*}
From above we have the bounds
\[ |v_1(t,r)|\lesssim \lambda(t)^\frac12 [\lambda(t)t]^{-4}\langle \lambda(t)r\rangle^3 \]
for $0\leq r <\frac12 t$ and 
\[ |v_{11}^b(t,r)|\lesssim \lambda(t)^\frac12 [\lambda(t)t]^{-1}(1-\tfrac{r}{t})^{\frac12(1-\nu)}
\lesssim \lambda(t)^\frac12 [\lambda(t)t]^{-1}t^{\frac12|1-\nu|} \]
for $\frac12 t\leq r\leq t-c$.
Furthermore,
\[ |v_{12}^b(t,r)|\lesssim \lambda(t)^\frac12 [\lambda(t)t]^{-2}(1-\tfrac{r}{t})^{\frac12(1-3\nu)}\lesssim 
\lambda(t)^\frac12 [\lambda(t)t]^{-1}t^{\frac12|1-\nu|} \]
for $\frac12 t\leq r\leq t-c$ and the stated bounds for $e_2$ follow.
\end{proof}

\begin{remark}
There is a slight nuisance associated with the function $v_{11}$ constructed in Lemma \ref{lem:v1}:
its odd derivatives with respect to $r$ do not vanish at the origin, i.e., $v_{11}$ is not smooth
at the center when viewed as a radial function on $\mathbb{R}^3$.
This inconvenient fact, however, is easily remedied if we replace $v_{11}$ by 
$\theta v_{11}$ where $\theta$ is a smooth function with $\theta(r)=r$ for, say, $r\in [0,\frac12]$ and
$\theta(r)=1$ for $r\geq 1$.
This modification does not affect the bounds given in Lemma \ref{lem:v1} (provided that $t_0>1$)
and it yields 
the desired behavior near the center.
Furthermore, since
$\Box (\theta v_{11})(t,r)=\Box v_{11}(t,r)$ for $r \geq 1$, the stated estimates for the corresponding
error $e_2$ are not altered either.
Consequently, we may equally well assume from the onset that 
$v_{11}$ and $e_2$ are smooth at the center (as functions on $\mathbb{R}^3$).
This remark will be useful later on.
\end{remark}

\section{The transport equation}
\label{sec:exact}

For the sake of clarity we outline the main results of this section.
\begin{enumerate}
\item We make the ansatz $u=u_2+\varepsilon$ and perform the change of variables $(t,r)\mapsto (\tau,R)$
where $\tau=\frac{1}{\nu}\lambda(t)t$, $R=\lambda(t)r$. From the requirement $\Box u+u^5=0$ we derive
an equation for $v(\tau,R):=R\varepsilon(\nu \tilde\lambda(\tau)^{-1}\tau,\tilde\lambda(\tau)^{-1}R)$ of the form
\begin{equation}
\label{eq:voutl} \mc D^2 v+\beta_\nu(\tau)\mc D v+\mc L v=\tilde \lambda(\tau)^{-2}[\mbox{r.h.s.}]  
\end{equation}
where $\mc D=\partial_\tau+\beta_\nu(\tau)(R\partial_R-1)$, $\beta_\nu(\tau)=(1-\frac{1}{\nu})\tau^{-1}$, and
$\mc L=-\partial_R^2-5W(R)^4$.

\item Next, we discuss the spectral theory of the Schr\"odinger operator $\mc L$. We derive the 
asymptotics of the spectral measure $\mu$ associated to $\mc L$ by applying Weyl-Titchmarsh theory. 
As a consequence, we obtain a precise description of the spectral transformation (the ``distorted Fourier transform'')
$\mc U: L^2(0,\infty)\to L^2(\sigma(\mc L),d\mu)$, the unitary map satisfying $\mc U \mc L f(\xi)=\xi \mc U f(\xi)$.

\item We apply this map to Eq.~\eqref{eq:voutl} in order to ``transform away'' the potential $-5W(R)^4$.
This yields an equation for \footnote{In fact, we have to work with a vector-valued function $x$ since
$\mc L$ has a negative eigenvalue. This is not essential for the argument but complicates the notation.
Thus, for the moment we ignore this issue.}
$x(\tau,\xi):=[\mc U v(\tau,\cdot)](\xi)$ of the form
\[ [\hat{\mc D}_c^2 +\beta_\nu(\tau) \hat{\mc D}_c+\xi]x(\tau,\xi)=\tilde \lambda^{-2}(\tau)[\mbox{r.h.s.}]
+\mbox{``$\mc K$-terms''} \]
where $\hat{\mc D}_c=\partial_\tau-2\beta_\nu \xi \partial_\xi+O(\tau^{-1})$.
The expression ``$\mc K$-terms'' stands for nonlocal error terms which arise from the application
of $\mc U$ to $R\partial_R$ (in $\mc D$).

\item Then we study the inhomogeneous equation
\[ [\hat {\mc D}_c^2+\beta_\nu(\tau) \hat{\mc D}_c+\xi]x(\tau,\xi)=b(\tau,\xi) \]
and solve it by the method of characteristics.
We obtain suitable pointwise bounds for the kernel of the solution operator (the ``parametrix'').

\item Finally, we introduce the basic solution spaces we work with and prove bounds for
the parametrix in these spaces.
\end{enumerate}

\subsection{Change of variables}
The function $u_2$ constructed in Section \ref{sec:approx} satisfies the critical wave equation
only up to an error $e_2$.
In this section we aim at constructing an \emph{exact} solution to
$$ \Box u+u^5=0 $$
of the form $u=u_2+\varepsilon$ where 
$\varepsilon(t,r)$
decays sufficiently fast as $t \to \infty$.
Recall that $u_2$ is nonsmooth at the lightcone and thus, we restrict our construction to a truncated
forward lightcone 
$$ K_{t_0,c}^\infty:=\{(t,x) \in \R\times \R^3: t\geq t_0, |x| \leq t-c\} $$
for $t_0>c>0$.
Consequently, we have to solve the equation
$$ \Box \varepsilon+5u_2^4 \varepsilon+N(u_2,\varepsilon)+e_2=0 $$
which we rewrite as
\begin{equation}
\label{eq:eps}
\Box \varepsilon+5u_0^4 \varepsilon=5(u_0^4-u_2^4)\varepsilon-N(u_2,\varepsilon)-e_2.
\end{equation}
As before, in order to obtain a time-independent potential, we use the new space variable
$R(t,r)=\lambda(t)r$.
Furthermore, it is convenient to introduce the new time variable $\tau(t)=\frac{1}{\nu}t^\nu$.
We write $\lambda(t)=\tilde{\lambda}(\tau(t))$ and note that
$$ \tau'(t)=\lambda(t),\quad \tilde{\lambda}'(\tau(t))=\frac{\lambda'(t)}{\lambda(t)}. $$
Consequently, the derivatives transform according to
$$ \partial_t=\tilde{\lambda}(\tau)\left (\partial_\tau+\tilde{\lambda}'(\tau)
\tilde{\lambda}(\tau)^{-1}R\partial_R \right ),\quad
\partial_r=\tilde{\lambda}(\tau)\partial_R
$$
and by setting $\tilde{v}(\tau(t),R(t,r))=\varepsilon(t,r)$ we obtain from
Eq.~\eqref{eq:eps} the problem
\begin{align}
\label{eq:tildev}
[\partial_\tau&+\beta_\nu(\tau)R\partial_R ]^2 \tilde{v}
+\beta_\nu(\tau)
[\partial_\tau+\beta_\nu(\tau)R\partial_R]\tilde{v} 
-\left [\partial_R^2+\tfrac{2}{R}\partial_R+5W(R)^4 \right ]\tilde{v} \\
&=\tilde{\lambda}(\tau)^{-2}\left [5(u_2^4-u_0^4)\tilde{v}+N(u_2,\tilde{v})+e_2 \right ]
\nonumber
\end{align}
with 
\begin{equation}
\label{eq:defbeta}
\beta_\nu(\tau)=\tilde{\lambda}'(\tau)\tilde{\lambda}(\tau)^{-1}=-(\tfrac{1}{\nu}-1) \tau^{-1}
\end{equation}
where it is understood, of
course, that the functions $u_0$, $u_2$ and $e_2$ be evaluated accordingly.
Finally, the standard substitution $\tilde{v}(\tau,R)=R^{-1}v(\tau,R)$ transforms the radial
$3d$ Laplacian into the radial $1d$ Laplacian and, by noting that
\[ [\partial_\tau+\beta_\nu(\tau)R\partial_R]\tfrac{v(\tau,R)}{R}=
\tfrac{1}{R}[\partial_\tau+\beta_\nu(\tau)R\partial_R-\beta_\nu(\tau)]v(\tau,R), \]
we end up with the main equation
\begin{align}
\label{eq:main}
\mc{D}^2 v+\beta_\nu(\tau)\mc{D} v+\mc{L}v
=\tilde{\lambda}(\tau)^{-2}\left [5(u_2^4-u_0^4)v+R N(u_2,R^{-1}v)+Re_2\right ]
\end{align}
where $\mc{D}=\partial_\tau+\beta_\nu(\tau)(R\partial_R-1)$ and $\mc{L}=-\partial_R^2-5W(R)^4$.
Our goal is to solve Eq.~\eqref{eq:main} backwards in time
with zero Cauchy data at $\tau=\infty$.
Roughly speaking, the idea is to perform a distorted Fourier transform with respect to the self-adjoint
operator $\mc{L}$ and to solve the remaining transport-type equation on the Fourier side
by the method of characteristics.

\subsection{Spectral theory of $\mc{L}$ and the distorted Fourier transform}

In the following we recall some standard facts about the spectral theory of $\mc{L}$, see
e.g.~ \cite{GZ06}, \cite{W03II}, \cite{T09}, \cite{DS63II}.
We write $V:=-5W^4$ and emphasize that $V(R)$ depends smoothly on $R$ and decays 
like $R^{-4}$ as $R \to \infty$.
The Schr\"odinger operator $\mc{L}=-\partial_R^2+V$ is self-adjoint in $L^2(0,\infty)$ with domain
$$ \mathrm{dom}(\mc{L})=\{f \in L^2(0,\infty): f,f' \in AC[0,R]\:\forall R>0, f(0)=0, \mc{L}f \in L^2(0,\infty)\} $$
since the endpoint $0$ is regular whereas $\infty$ is in the limit-point case.
Furthermore, there exists a zero energy resonance which is induced by the scaling symmetry
of the wave equation.
More precisely, the function 
\begin{equation}
\label{eq:res}
\phi_0(R)=2R \partial_\lambda|_{\lambda=1}\lambda^\frac12 W(\lambda R)=\frac{R(1-\frac{1}{3}R^2)}
{(1+\frac{1}{3}R^2)^{3/2}} 
\end{equation}
belongs to $L^\infty(0,\infty)$ and (formally) satisfies $\mc{L}\phi_0=0$.
Since $\phi_0$ has precisely one zero on $(0,\infty)$, it follows by Sturm oscillation theory
(see e.g.~ \cite{DS63II})
that $\mc{L}$ has exactly one simple negative eigenvalue $\xi_d<0$.
The corresponding eigenfunction $\phi_d$ is smooth and positive on $(0,\infty)$ and decays
exponentially towards $\infty$.
Thus, the spectrum of $\mc{L}$ is given by $\sigma(\mc{L})=\{\xi_d\}\cup [0,\infty)$ and 
thanks to the decay of $V$, the continuous part is in fact absolutely continuous. 

We denote by $\{\phi(\cdot,z),
\theta(\cdot,z)\}$, $z\in \mathbb{C}$, the standard fundamental system  of 
\begin{equation}
\label{eq:spec}
\mc{L}f=zf
\end{equation}
satisfying
$$ \phi(0,z)=\theta'(0,z)=0,\quad \phi'(0,z)=\theta(0,z)=1. $$
In particular we have $W(\theta(\cdot,z), \phi(\cdot,z))=1$ for all $z \in \mathbb{C}$.
Furthermore, $\psi(\cdot,z)$ for $z\in\mathbb{C}\backslash \mathbb{R}$ 
denotes the Weyl-Titchmarsh solution of Eq.~\eqref{eq:spec}, i.e., 
the unique solution of Eq.~\eqref{eq:spec} which belongs to $L^2(0,\infty)$ and 
satisfies
$\psi(0,z)=1$.
Consequently, for each $z \in \mathbb{C}\backslash \mathbb{R}$ there exists a number $m(z)$ such
that
$$ \psi(\cdot,z)=\theta(\cdot,z)+m(z)\phi(\cdot,z) $$
and we obtain
$m(z)=W(\theta(\cdot,z), \psi(\cdot,z))$.
The Weyl-Titchmarsh $m$-function is of crucial importance since it determines the spectral 
measure.

\begin{proposition}
\label{prop:Fourier}
\begin{enumerate}
\item For any $\xi>0$ the limit
$$ \rho(\xi):=\tfrac{1}{\pi}\lim_{\varepsilon \to 0+}\Im m(\xi+i\varepsilon) $$
exists but $\rho(\xi) \to \infty$ as $\xi \to 0+$.
\item 
Let $\mu$ be the Borel measure defined by
$$ d\mu(\xi)=\frac{d\Theta_{\xi_d}(\xi)}{\|\phi(\cdot,\xi_d)\|_{L^2(0,\infty)}^2}+\rho(\xi)d\xi $$
where $d\Theta_{\xi_d}$ denotes the Dirac measure at $\xi_d$.
Then there exists a unitary operator $\mc{U}: L^2(0,\infty) \to L^2(\sigma(\mc{L}),d\mu)$, 
the ``distorted Fourier transform'', which diagonalizes $\mc{L}$, i.e., 
$$ \mc{U}\mc{L}=M_{\mathrm{id}}\mc{U} $$
where $M_\mathrm{id}$ is the (maximally defined) operator of multiplication 
by the identity function. \footnote{In other words, $M_\mathrm{id}f(\xi)=\xi f(\xi)$.}

\item The distorted Fourier transform is explicitly given by
$$ \mc{U}f(\xi)=\lim_{b\to \infty}\int_0^b \phi(R,\xi)f(R)dR, \quad \xi\in \sigma(\mc{L}) $$
where the limit is understood with respect to $\|\cdot\|_{L^2(\sigma(\mc{L}),d\mu)}$.

\item The inverse transform $\mc{U}^{-1}$ reads
$$ \mc{U}^{-1}\hat{f}(R)=\frac{\phi(R,\xi_d)}{\|\phi(\cdot,\xi_d)\|_{L^2(0,\infty)}^2}
\hat{f}(\xi_d)+\lim_{b\to \infty}\int_0^b 
\phi(R,\xi)\hat{f}(\xi)\rho(\xi)d\xi $$
where the limit is understood with respect to $\|\cdot\|_{L^2(0,\infty)}$.
\end{enumerate}
\end{proposition}

\begin{proof}
Let $\Im z>0$ (and $\Im \sqrt{z}>0$). 
The Weyl-Titchmarsh solution is given by $\psi(\cdot,z)=c_0(z)f_+(\cdot,z)$ where 
the Jost function $f_+(\cdot,z)$ is defined by
$\mc{L}f_+(\cdot,z)=zf_+(\cdot,z)$ and $f_+(R,z)\sim e^{i\sqrt{z}R}$ as $R \to \infty$.
The coefficient $c_0(z)$ is chosen such that $\psi(0,z)=1$, i.e.,
\begin{equation}
\label{eq:defc0}
c_0(z)=\frac{1}{W(f_+(\cdot,z),\phi(\cdot,z))}. 
\end{equation}
The Jost function
satisfies the Volterra equation
\begin{equation} 
\label{eq:Volterraf+}
f_+(R,z)=e^{i\sqrt{z}R}+\frac{1}{\sqrt{z}}\int_R^\infty \sin(\sqrt{z}(R'-R))
V(R')f_+(R',z)dR' 
\end{equation}
and from this representation it follows immediately that 
$$ f_+(R,\xi)=\lim_{\varepsilon\to 0+}f_+(R,\xi+i\varepsilon)$$
exists provided that $\xi>0$ and moreover, $f_+(R,\xi)$ satisfies Eq.~\eqref{eq:Volterraf+} with $z=\xi$, cf.~\cite{DT79}.
From Eq.~\eqref{eq:Volterraf+} we also have $W(f_+(\cdot,z),\overline{f_+(\cdot,z)})=-2i\sqrt{z}$
and thus, by expanding
$$ f_+(\cdot,\xi)=a(\xi)\phi(\cdot,\xi)+b(\xi)\theta(\cdot,\xi) $$
we obtain (recall that $\phi(\cdot,\xi)$ and $\theta(\cdot,\xi)$ are real-valued)
$$ -2i\sqrt{\xi}=W(f_+(\cdot,\xi),\overline{f_+(\cdot,\xi)})=2i\Im(\overline{a(\xi)}b(\xi)) $$
which in particular implies $b(\xi)\not=0$ if $\xi>0$.
Consequently, we infer
$$ W(f_+(\cdot,\xi),\phi(\cdot,\xi))=b(\xi)\not=0 $$
and this shows that $c_0(\xi)$ (and therefore $\psi(R,\xi)$) is well-defined and finite 
provided that $\xi>0$.
However, due to the zero energy resonance we clearly have $|c_0(\xi)| \to \infty$ as $\xi \to 0+$.
The connection between the Weyl-Titchmarsh $m$-function and the spectral measure $\mu$ 
is provided by the classical formula
$$ \tilde{\mu}(\xi)=\tfrac{1}{\pi}\lim_{\delta \to 0+}\lim_{\varepsilon \to 0+}
\int_\delta^{\xi+\delta}\Im m(t+i\varepsilon)dt $$
where the distribution function $\tilde{\mu}$ determines $\mu$ in the sense of Lebesgue-Stieltjes.
The statements about the distorted Fourier transform are well-known and classical, see e.g.~ 
\cite{W03II},
\cite{T09}, \cite{DS63II}, \cite{GZ06}.
\end{proof}

\subsection{Asymptotics of the spectral measure for small $\xi$}
In order to be able to apply the distorted Fourier transform, we require more detailed information
on the behavior of the spectral measure. We start with the asymptotics as $\xi \to 0+$ where
the spectral measure blows up due to the existence of the zero energy resonance.
We also obtain estimates for the fundamental system $\{\phi(\cdot,\xi),\theta(\cdot,\xi)\}$ which
will be relevant later on.
As before, $\phi_0(R)=\phi(R,0)$ is the resonance function given in Eq.~\eqref{eq:res} and we write
$\theta_0(R):=\theta(R,0)$. Explicitly, we have
$$ \theta_0(R)=\frac{1-2R^2+\frac19 R^4}{(1+\frac13 R^2)^{3/2}}. $$

\begin{lemma}
\label{lem:Phi}
There exists a (complex-valued) function $\Phi(\cdot,\xi)$ satisfying 
$\mc{L}\Phi(\cdot,\xi)=\xi \Phi(\cdot,\xi)$ such that
$$ \Phi(R,\xi)=[\phi_0(R)+i\theta_0(R)][1+a(R,\xi)] $$ where 
$a(0,\xi)=a'(0,\xi)=0$ and 
\begin{align*} 
|\partial_\xi^\ell \partial_R^k [\Re a(R,\xi)]|&\leq C_{k,\ell}\langle R \rangle^{2-k}\xi^{1-\ell} \\
|\partial_\xi^\ell \partial_R^k [\Im a(R,\xi)]|&\leq C_{k,\ell}\left [\langle R \rangle^{1-k}\xi^{1-\ell} 
+\langle R\rangle^{4-k}\xi^{2-\ell} \right ]
\end{align*}
for all $R \in [0,\xi^{-\frac12}]$, $0<\xi\lesssim  1$ and $k,\ell \in \mathbb{N}_0$.
In particular, we have $\phi=\Re \Phi$ and $\theta=\Im \Phi$.
\end{lemma}

\begin{proof}
We write $\Phi_0:=\phi_0+i\theta_0$ and note that $\Phi_0$ does not vanish anywhere on $[0,\infty)$.
Inserting the ansatz $\Phi(\cdot,\xi)=\Phi_0[1+a(\cdot,\xi)]$ into $\mc{L}\Phi(\cdot,\xi)=\xi \Phi(\cdot,\xi)$ 
yields the Volterra equation
\begin{equation}
\label{eq:volterraa}
a(R,\xi)=-\xi \int_0^R \int_{R'}^R \Phi_0(R'')^{-2}dR''\;\Phi_0(R')^2 [1+a(R',\xi)]dR' 
\end{equation}
which is of the form
$$ a(R,\xi)=\int_0^R K(R,R',\xi)[1+a(R',\xi)]dR' $$ 
with a kernel satisfying $|K(R,R',\xi)|\lesssim \langle R'\rangle \xi$ for all $0\leq R' \leq R$ and $\xi>0$.
Consequently, we have
$$ \int_0^{\xi^{-\frac12}}\sup_{R \in (R',\xi^{-\frac12})}|K(R,R',\xi)|dR' \lesssim 1 $$
and a standard Volterra iteration yields the existence of $a(\cdot,\xi)$ with
$|a(R,\xi)|\lesssim \langle R\rangle^2 \xi$ for all $R \in [0,\xi^{-\frac12}]$ and $0<\xi\lesssim 1$.
Obviously, we have $a(0,\xi)=a'(0,\xi)=0$ and this immediately yields $\phi=\Re \Phi$ and 
$\theta=\Im \Phi$.
Now observe that, for $0\leq R' \leq R$,
$$ \int_{R'}^R \Phi_0(R'')^{-2}dR''=\int_{R'}^R \left [O(\langle R'' \rangle^{-2})
+iO(\langle R'' \rangle^{-3})\right ]dR''=O(\langle R'\rangle^{-1})+iO(\langle R'\rangle^{-2}) $$
which implies
$$ \int_{R'}^R \Phi_0(R'')^{-2}dR''\;\Phi_0(R')^2=O(\langle R'\rangle)
+iO(1). $$
Consequently, with $|a(R,\xi)|\lesssim \langle R \rangle^2 \xi$ from above and 
Eq.~\eqref{eq:volterraa} we infer
$$ \Im a(R,\xi)=O(\langle R \rangle \xi)+\xi \Im \int_0^R O_\mathbb{C}(\langle R'\rangle )a(R',\xi)dR'
=O(\langle R \rangle \xi)+O(\langle R \rangle^4 \xi^2). $$
The derivative bounds follow inductively from Eq.~\eqref{eq:volterraa}
by symbol calculus.
\end{proof}

Next, we consider the Jost function $f_+(\cdot,\xi)$.

\begin{lemma}
\label{lem:Jost}
The Jost function $f_+(\cdot,\xi)$ of the operator $\mc{L}$ is of the form
$$ f_+(R,\xi)=e^{i\sqrt{\xi} R}[1+b(R,\xi)] $$
where $b(\cdot,\xi)$ satisfies the bounds
$$ |\partial_\xi^\ell \partial_R^k b(R,\xi)|\leq C_{k,\ell}\langle R \rangle^{-3-k}\xi^{-\frac12-\ell} $$
for all $R\geq \xi^{-\frac16}$, $0<\xi\lesssim 1$ and $k,\ell \in \mathbb{N}_0$.
\end{lemma}

\begin{proof}
The function $b(\cdot,\xi)$ satisfies the Volterra equation
\begin{align}
\label{eq:volterrab}
b(R,\xi)&=\frac{1}{2i\sqrt{\xi}}\int_R^\infty \left [e^{2i\sqrt{\xi}(R'-R)}-1 \right ]V(R')
[1+b(R',\xi)]dR' \\
&=\int_R^\infty K(R,R',\xi)[1+b(R',\xi)]dR'. \nonumber 
\end{align}
Thanks to the strong decay of $V$ we have 
$|K(R,R',\xi)|\lesssim \langle R' \rangle^{-4}\xi^{-\frac12}$ for $R\leq R'$ and thus,
$$ \int_{\xi^{-\frac16}}^\infty \sup_{R \in (\xi^{-\frac16},R')}|K(R,R',\xi)|dR' \lesssim 1. $$
Consequently, a standard Volterra iteration yields the claim. The derivative bounds follow 
inductively by symbol calculus.
\end{proof}

The information provided by Lemmas \ref{lem:Phi} and \ref{lem:Jost} already suffices to obtain 
the asymptotics of the spectral measure for $\xi \to 0+$.

\begin{lemma}
\label{lem:smxi}
For the functions $c_0$ and $\rho$ given in Eq.~\eqref{eq:defc0} and Proposition \ref{prop:Fourier},
respectively, we have
$$
c_0(\xi)=-\tfrac{i}{\sqrt{3}}\xi^{-\frac12}[1+O_\mathbb{C}(\xi^\frac15)],\quad
\rho(\xi)=\tfrac{1}{3\pi}\xi^{-\frac12}[1+O(\xi^\frac15)]
$$
for $0<\xi\ll 1$ where the $O$-terms behave like symbols under differentiation.
\end{lemma}

\begin{proof}
For the Weyl-Titchmarsh solution $\psi(\cdot,\xi)$ we have $\psi(\cdot,\xi)=c_0(\xi)f_+(\cdot,\xi)$
where the coefficient $c_0(\xi)$ is given by
$$ c_0(\xi)=\frac{1}{W(f_+(\cdot,\xi),\phi(\cdot,\xi))}, $$
cf.~the proof of Proposition \ref{prop:Fourier}.
From Lemma \ref{lem:Phi} we have 
\begin{align*} 
\phi(R,\xi)&=\Re \Phi(R,\xi)=\phi_0(R)[1+\Re a(R,\xi)]-\theta_0(R)\Im a(R,\xi) \\
&=\phi_0(R)[1+O(\langle R\rangle^2 \xi)+O(\langle R\rangle^5\xi^2)],\quad R\geq 1 
\end{align*}
and, by noting that $|\phi_0'(R)|\simeq \langle R\rangle^{-3}$ for $R\geq 1$,
$$\phi'(R,\xi)=O(\langle R\rangle^{-3})[1+O(\langle R \rangle^4 \xi)+O(\langle R\rangle^7 \xi^2)],
\quad R\geq 1.$$ 
We evaluate the Wronskian at $R=\xi^{-\frac{3}{10}}$.
Thus, we use
\begin{align*} 
\phi(\xi^{-\frac{3}{10}},\xi)&=\phi_0(\xi^{-\frac{3}{10}})[1+O(\xi^\frac25)]=-\sqrt{3}[1+O(\xi^\frac25)] \\
\phi'(\xi^{-\frac{3}{10}},\xi)&=O(\xi^{\frac{9}{10}})[1+O(\xi^{-\frac{1}{5}})]=O(\xi^{\frac{7}{10}})
\end{align*}
and, from Lemma \ref{lem:Jost},
\begin{align*}
f_+(\xi^{-\frac{3}{10}},\xi)&=e^{i\xi^{1/5}}[1+O_\mathbb{C}(\xi^\frac25)]=1+O_\mathbb{C}(\xi^\frac15) \\
f_+'(\xi^{-\frac{3}{10}},\xi)&=i\xi^\frac12 e^{i\xi^{1/5}}[1+O_\mathbb{C}(\xi^\frac15)]
=i\xi^\frac12 [1+O_\mathbb{C}(\xi^\frac15)]
\end{align*}
and all $O$-terms behave like symbols under differentiation.
Consequently, we obtain
$$ W(f_+(\cdot,\xi),\phi(\cdot,\xi))=\sqrt{3}i\xi^\frac12[1+O_\mathbb{C}(\xi^{\frac{1}{5}})] $$
and thus, 
$c_0(\xi)=-\frac{i}{\sqrt{3}}\xi^{-\frac12}[1+O_\mathbb{C}(\xi^{\frac15})]$ where the $O$-term
behaves like a symbol.

The Weyl-Titchmarsh $m$-function is given by 
$$m(\xi)=W(\theta(\cdot,\xi),\psi(\cdot,\xi))
=c_0(\xi)W(\theta(\cdot,\xi),f_+(\cdot,\xi)).$$
From Lemma \ref{lem:Phi} we have
$$ \theta(\xi^{-\frac{3}{10}},\xi)=O(\xi^{-\frac{3}{10}}),\quad \theta'(\xi^{-\frac{3}{10}},\xi)
=\tfrac{1}{\sqrt{3}}+O(\xi^\frac15) $$
and thus, $W(\theta(\cdot,\xi),f_+(\cdot,\xi))=-\frac{1}{\sqrt{3}}+O_\mathbb{C}(\xi^\frac15)$.
This yields $\rho(\xi)=\tfrac{1}{\pi}\Im m(\xi)=\frac{1}{3\pi} \xi^{-\frac12}[1+O(\xi^\frac15)]$ with an $O$-term that 
behaves like a symbol.
\end{proof}

\subsection{Asymptotics of the spectral measure for large $\xi$}
In this section we study the behavior of $\rho(\xi)$ as $\xi \to \infty$.
This is considerably easier than the limit $\xi \to 0+$.
In order to get a small factor in front of the potential, it is convenient to rescale the equation
$\mc{L}f=\xi f$ by setting $f(R)=\tilde{f}(\xi^\frac12 R)$ which yields
\begin{equation}
\label{eq:specrescaled}
\tilde{f}''(y)+\tilde{f}(y)=\xi^{-1}V(\xi^{-\frac12}y)\tilde{f}(y)
\end{equation}
for $y \geq 0$, $\xi \gtrsim 1$ and this form already suggests to treat the right-hand side
perturbatively.

\begin{lemma}
\label{lem:Jostlgxi}
The Jost function $f_+(\cdot,\xi)$ of $\mc{L}$ is of the form
$$ f_+(R,\xi)=e^{i\sqrt{\xi}R}[1+b(R,\xi)] $$
where
$$ |\partial_\xi^\ell \partial_R^k b(R,\xi)|\leq C_{k,\ell}\langle R\rangle^{-3-k}
\xi^{-\frac12-\ell} $$
for all $R \geq 0$, $\xi\gtrsim 1$ and $k,\ell \in \mathbb{N}_0$.
\end{lemma}

\begin{proof}
We start by constructing a solution $\tilde{f}_+(\cdot,\xi)$ to Eq.~\eqref{eq:specrescaled}
of the form $\tilde{f}_+(y,\xi)=e^{iy}[1+\tilde{b}(y,\xi)]$.
Inserting this ansatz into Eq.~\eqref{eq:specrescaled} yields the Volterra equation
\begin{align}
\label{eq:volterraf+}
\tilde{b}(y,\xi)&=\tfrac{1}{2i}\xi^{-1}\int_y^\infty [e^{2i(y'-y)}-1]V(\xi^{-\frac12}y')
[1+\tilde{b}(y',\xi)]dy' \\
&=\int_y^\infty K(y,y',\xi)[1+\tilde{b}(y',\xi)]dy' \nonumber
\end{align}
where $|K(y,y',\xi)|\lesssim \xi^{-1}\langle \xi^{-\frac12}y'\rangle^{-4}$ for all $0\leq y\leq y'$
and $\xi \gtrsim 1$.
Consequently, we have 
$$ \int_0^\infty \sup_{y \in (0,y')}|K(y,y',\xi)|dy'\lesssim \xi^{-\frac12}\lesssim 1 $$
and a Volterra iteration yields $|\tilde{b}(y,\xi)|\lesssim \xi^{-\frac12}\langle \xi^{-\frac12}
y\rangle^{-3}$.
Furthermore, by introducing the new variable $u=y'-y$, we rewrite Eq.~\eqref{eq:volterraf+} as
$$ \tilde{b}(y,\xi)=\tfrac{1}{2i}\xi^{-1}\int_0^\infty [e^{2iu}-1]V(\xi^{-\frac12}(u+y))
[1+\tilde{b}(u+y,\xi)]du $$
and with 
$$ |\partial_\xi^\ell \partial_y^k V(\xi^{-\frac12}(u+y))|\leq C_{k,\ell}
\xi^{-\frac12 k-\ell}\langle \xi^{-\frac12}(u+y)\rangle^{-4-k}, \quad k,\ell \in \mathbb{N}_0 $$
for all $y, u\geq 0$, $\xi \gtrsim 1$,
we obtain inductively the bounds
$$ |\partial_\xi^\ell \partial_y^k \tilde{b}(y,\xi)|\leq C_{k,\ell}\xi^{-\frac12-\frac12 k-\ell}
\langle \xi^{-\frac12}y \rangle^{-3-k} $$
for all $y \geq 0$, $\xi \gtrsim 1$ and $k,\ell \in \mathbb{N}_0$.
We have $\tilde{f_+}(\xi^\frac12 R)\sim e^{i\sqrt{\xi}R}$ as $R \to \infty$ and thus, 
as already suggested by the notation,
the Jost solution is given by $f_+(R,\xi)=\tilde{f}_+(\xi^\frac12 R,\xi)$ and by setting
$b(R,\xi)=\tilde{b}(\xi^\frac12 R,\xi)$ we obtain the stated form of $f_+(\cdot,\xi)$ and
the bounds for $b$ follow from the ones for $\tilde{b}$ by the chain rule.
\end{proof}

\begin{lemma}
\label{lem:lgxi}
The functions $c_0$ and $\rho$ are of the form
$$ c_0(\xi)=1+O_\mathbb{C}(\xi^{-\frac12}),\quad \rho(\xi)=\tfrac{1}{\pi}\xi^\frac12 [1+O(\xi^{-\frac12})] $$
for $\xi \gtrsim 1$ where the $O$-terms behave like symbols under differentiation.
\end{lemma}

\begin{proof}
By evaluation at $R=0$ we obtain from Lemma \ref{lem:Jostlgxi}
$$ W(f_+(\cdot,\xi),\phi(\cdot,\xi))=f_+(0,\xi)=1+b(0,\xi)=1+O_\mathbb{C}(\xi^{-\frac12}) $$
which, by Eq.~\eqref{eq:defc0}, implies $c_0(\xi)=1+O_\mathbb{C}(\xi^{-\frac12})$ 
where the $O$-term 
behaves like a symbol.
Furthermore, we have
$$ W(\theta(\cdot,\xi),f_+(\cdot,\xi))=f_+'(0,\xi)=i\xi^\frac12[1+b(0,\xi)]
+b'(0,\xi)=i\xi^\frac12[1+O_\mathbb{C}(\xi^{-\frac12})] $$
and we infer 
$$ \rho(\xi)=\tfrac{1}{\pi}\Im\left [c_0(\xi)W(\theta(\cdot,\xi),f_+(\cdot,\xi))\right ]
=\tfrac{1}{\pi}\xi^{\frac12}[1+O(\xi^{-\frac12})] $$
with an $O$-term that behaves like a symbol.
\end{proof}

It is now a simple matter to obtain a convenient representation of $\phi(\cdot,\xi)$ in terms
of the Jost function $f_+(\cdot,\xi)$.

\begin{corollary}
\label{cor:phi}
The function $\phi(\cdot,\xi)$ has the representation
$$ \phi(R,\xi)=a(\xi)f_+(R,\xi)+\overline{a(\xi)f_+(R,\xi)} $$
where
\begin{align*}
a(\xi)&=\tfrac{\sqrt{3}}{2}+O_\mathbb{C}(\xi^\frac15),\quad 0 < \xi \ll 1 \\
a(\xi)&=\tfrac{1}{2i}\xi^{-\frac12}[1+O_\mathbb{C}(\xi^{-\frac12})],\quad \xi \gtrsim 1
\end{align*}
and the $O$-terms behave like symbols.
\end{corollary}

\begin{proof}
Since $W(f_+(\cdot,\xi),\overline{f_+(\cdot,\xi)})=-2i\sqrt{\xi}$ it is clear
that there exist coefficients $a(\xi)$, $b(\xi)$ such that $\phi(\cdot,\xi)=a(\xi)f_+(\cdot,\xi)
+b(\xi)\overline{f_+(\cdot,\xi)}$ provided that $\xi>0$.
From the fact that $\phi(\cdot,\xi)$ is real-valued it follows that $b(\xi)=\overline{a(\xi)}$.
Consequently, we obtain
$$ \tfrac{1}{c_0(\xi)}=W(f_+(\cdot,\xi),\phi(\cdot,\xi))=-2i\xi^\frac12 \overline{a(\xi)} $$
and Lemmas \ref{lem:smxi}, \ref{lem:lgxi} yield the claim.
\end{proof}

\subsection{The transference identity}
Unfortunately, 
it is not straightforward to apply the distorted Fourier transform to Eq.~\eqref{eq:main} 
due to the presence
of the derivative $R\partial_R$.
The idea is to substitute the term $R\partial_R$ by a suitable derivative on the Fourier
side. 
This is not possible without making an error.
For the following it is convenient to distinguish between the continuous and the discrete part 
of the spectrum of $\mc{L}$.
This is most effectively done by introducing vector notation.
Consequently, we interpret the distorted Fourier transform, now denoted by $\mc{F}$, as a 
vector-valued map
$\mc{F}: L^2(0,\infty) \to \mathbb{C} \times L^2((0,\infty),\rho(\xi)d\xi)$ given by
$$ \mc{F}f:=\left ( \begin{array}{c}\mc{U}f(\xi_d) \\ \mc{U}f|_{[0,\infty)} \end{array}
\right ) $$
with $\mc{U}$ from Proposition \ref{prop:Fourier}.
By Proposition \ref{prop:Fourier} the inverse map 
\[ \mc{F}^{-1}: \C\times L^2((0,\infty),\rho(\xi)d\xi) \to L^2(0,\infty) \] reads
\begin{equation}
\label{eq:defF-1}
\mc{F}^{-1}\left (\begin{array}{c} a \\ f \end{array} \right )=
a\frac{\phi(\cdot,\xi_d)}{\|\phi(\cdot,\xi_d)\|_{L^2(0,\infty)}^2}+\lim_{b\to\infty} \int_0^b
\phi(\cdot,\xi)f(\xi)\rho(\xi)d\xi. 
\end{equation}
We define the error operator $\mc{K}$ by 
\begin{equation}
\label{eq:K}
\mc{F}((\cdot)f'-f)=\mc{A}\mc{F}f +\mc{K}\mc{F}f 
\end{equation}
for $f\in C^\infty_c(0,\infty)$ where
$$ \mc{A}:=\left (\begin{array}{cc}
0 & 0 \\ 0 & \mc A_c\end{array} \right )
,\quad \mc A_c g(\xi):=-2 \xi g'(\xi)-(\tfrac52+\tfrac{\xi\rho'(\xi)}{\rho(\xi)})g(\xi) $$
and we write
$$ \mc{K}=\left (\begin{array}{cc}\mc{K}_{dd} & \mc{K}_{dc} \\
\mc{K}_{cd} & \mc{K}_{cc} \end{array} \right )$$
for the matrix components of $\mc{K}$.
We call Eq.~\eqref{eq:K}
the ``transference identity'' since it allows us to transfer derivatives with respect
to $R$ to derivatives with respect to the Fourier variable $\xi$.
In order to motivate the definitions of $\mc{A}$ and $\mc{K}$, let us take 
the free case as a model problem, i.e., assume for the moment that $V=0$ and
$\mc{L}=-\partial_R^2$.
Note, however, that the free case with a Dirichlet condition at zero is not a good model for 
our problem since it is not resonant.
Consequently, we assume a Neumann condition instead.
In the free case there is no discrete spectrum and the corresponding $\phi(\cdot,\xi)$ can be given explicitly
and reads $\phi(R,\xi)=-\cos(\xi^\frac12 R)$.
Furthermore, the spectral measure is $\rho(\xi)=\frac{1}{\pi}\xi^{-\frac12}$.
For the transference identity we obtain
\begin{align*}
\mc{U}((\cdot)f'-f)(\xi)&=\int_0^\infty \phi(R,\xi)Rf'(R)dR-\mc{U}f(\xi) \\
&=-\int_0^\infty R\xi^{\frac12}\sin(\xi^\frac12 R)f(R)dR-2\int_0^\infty \phi(R,\xi)f(R)dR \\
&=-2\xi\partial_\xi \int_0^\infty\phi(R,\xi)f(R)dR-2\mc{U}f(\xi)
=-2 \xi (\mc{U}f)'(\xi)-(\tfrac52+\tfrac{\xi\rho'(\xi)}{\rho(\xi)})\mc{U}f(\xi)
\end{align*}
for $f \in C_c^\infty(0,\infty)$ and we recover the operator $\mc{A}_c$ whereas the corresponding 
error operator is identically zero.
Due to the strong decay of $V(R)$ as $R \to \infty$ it is reasonable
to expect that the transference identity Eq.~\eqref{eq:K} is well approximated by the above model case, 
at least to leading order.
Therefore, $\mc{K}$ should be ``small'' in a suitable sense. 
We will make this idea rigorous in Section \ref{sec:K} where we prove appropriate mapping
properties of $\mc{K}$ which exhibit a certain smoothing effect that turns out to be crucial
for the whole construction.

\subsection{Application of the distorted Fourier transform}
Now we intend to apply the distorted Fourier transform to Eq.~\eqref{eq:main}.
In order to be able to do so, however, we have to deal with the fact that the functions on the
right-hand side of Eq.~\eqref{eq:main} are only defined in a forward lightcone and we have to
extend them smoothly to all $r\geq 0$.
To this end we use a smooth cut-off $\chi$ satisfying $\chi(x)=0$ for $x \leq \frac12$ and
$\chi(x)=1$ for $x\geq 1$.
Then $\chi(\frac{t-r}{c})$ is identically $1$ in the truncated cone $r\leq t-c$ and identically
$0$ if $r\geq t-\frac{c}{2}$. Of course, we assume here that $t \geq c$.
Furthermore, in terms of the new variables $\tau=\frac{1}{\nu}t^\nu=\frac{1}{\nu}\lambda(t)t$ and
$R=\lambda(t)r$, the cut-off reads
$$ \tilde{\chi}(\tau,R):=\chi\left (\tilde{\lambda}(\tau)^{-1}\tfrac{\nu \tau-R}{c}\right). $$
Consequently, the equation we really want to solve is given by
\begin{equation}
\label{eq:main2}
\mc{D}^2 v+\beta_\nu(\tau)\mc{D} v+\mc{L}v
=\tilde{\lambda}(\tau)^{-2}\tilde{\chi}\left [5(u_2^4-u_0^4)v+R N(u_2,R^{-1}v)+Re_2\right ],
\end{equation}
cf.~Eq.~\eqref{eq:main}, and we
recall that $\mc{D}=\partial_\tau+\beta_\nu(\tau)(R\partial_R-1)$.
Thus, we have
$$ \mc{F}\mc{D}=\partial_\tau \mc{F}+\beta_\nu(\mc{A}+\mc{K})\mc{F}=:\hat{\mc{D}}\mc{F} $$
and this yields
$
\mc{F}\mc{D}^2=\hat{\mc{D}}^2 \mc{F}
$
where
\begin{align*}
\hat{\mc{D}}^2&=\partial_\tau^2 +2\beta_\nu(\mc{A}+\mc{K})\partial_\tau+\beta_\nu^2 
(\mc{A}^2+\mc{K}\mc{A}+\mc{A}\mc{K}+\mc{K}^2)+\beta_\nu'(\mc{A}+\mc{K}) \\
&=(\partial_\tau+\beta_\nu \mc{A})^2+2\beta_\nu \mc{K}\partial_\tau
+\beta_\nu^2 (2\mc{KA}+[\mc{A,K}]+\mc{K}^2)+\beta_\nu' \mc{K}.
\end{align*}
We conclude that
$$ \hat{\mc{D}}^2+\beta_\nu \hat{\mc{D}}=(\partial_\tau+\beta_\nu \mc{A})^2+\beta_\nu(2\mc{K}+1)
(\partial_\tau+\beta_\nu\mc{A})+\beta_\nu^2(\mc{K}^2+[\mc{A},\mc{K}]+\mc{K}+\tfrac{\beta_\nu'}{\beta_\nu^2}\mc{K}). $$
In the following we write $(x_d(\tau),x(\tau,\xi))=\mc{F}v(\tau,\cdot)(\xi)$.
Consequently, by applying $\mc{F}$ to Eq.~\eqref{eq:main2}, we end up with the system
\begin{align}
\label{eq:sysFourier}
&\left ( \begin{array}{cc}\partial_\tau^2+\xi_d & 0 \\
0 & (\partial_\tau+\beta_\nu(\tau)\mc A_c)^2+\xi \end{array} \right )
\left (\begin{array}{c}x_d(\tau) \\ x(\tau,\xi) \end{array} \right )
=\sum_{j=1}^5 \mc{N}_j\left ( \begin{array}{c}x_d \\ x \end{array} \right)
(\tau,\xi) \\
&\quad \quad -\beta_\nu(\tau)(2\mc{K}+1)(\partial_\tau+\beta_\nu(\tau)\mc{A})\left (\begin{array}{c}x_d(\tau)\\
x(\tau,\cdot) \end{array} \right )(\xi) \nonumber \\
&\quad \quad -\beta_\nu(\tau)^2 \left ( \mc{K}^2+[\mc{A},\mc{K}]+\mc{K}+\tfrac{\beta_\nu'(\tau)}{\beta_\nu(\tau)^2}
\mc{K} \right )\left (\begin{array}{c}x_d(\tau) \\ x(\tau,\cdot) \end{array} \right )(\xi) \nonumber \\
&\quad \quad +\left ( \begin{array}{c}\hat{e}_2(\tau,\xi_d) \\ \hat{e}_2(\tau,\xi) \end{array}
\right )\nonumber
\end{align}
where the operators $\mc{N}_j$, $j \in \{1,2,3,4,5\}$, are given by
\begin{align}
\label{eq:defR}
\mc{N}_j \left (\begin{array}{c}x_d \\ x \end{array} \right )(\tau,\xi)&:=\mc{F}
\left (|\cdot|\varphi_j(\tau,\cdot) 
\left [|\cdot|^{-1}\mc{F}^{-1}
\left (\begin{array}{c}x_d(\tau) \\ x(\tau,\cdot) \end{array} \right ) \right]^j \right )(\xi)
\end{align}
with
\begin{align*}
\varphi_1(\tau,R)&=5\tilde{\lambda}(\tau)^{-2}\tilde{\chi}(\tau,R)[
u_2(\nu \tilde{\lambda}(\tau)^{-1}\tau, \tilde{\lambda}(\tau)^{-1}R)^4
-u_0(\nu \tilde{\lambda}(\tau)^{-1}\tau, \tilde{\lambda}(\tau)^{-1}R)^4] \\
\varphi_2(\tau,R)&=10 \tilde{\lambda}(\tau)^{-2}\tilde{\chi}(\tau,R) 
u_2(\nu \tilde{\lambda}(\tau)^{-1}\tau, \tilde{\lambda}(\tau)^{-1}R)^3 \\
\varphi_3(\tau,R)&=10 \tilde{\lambda}(\tau)^{-2}\tilde{\chi}(\tau,R) 
u_2(\nu \tilde{\lambda}(\tau)^{-1}\tau, \tilde{\lambda}(\tau)^{-1}R)^2 \\
\varphi_4(\tau,R)&=5\tilde{\lambda}(\tau)^{-2}\tilde{\chi}(\tau,R)
u_2(\nu \tilde{\lambda}(\tau)^{-1}\tau, \tilde{\lambda}(\tau)^{-1}R) \\ 
\varphi_5(\tau,R)&=\tilde{\lambda}(\tau)^{-2}\tilde{\chi}(\tau,R)
\end{align*}
and
\begin{equation}
\label{eq:defhate2}
\hat{e}_2(\tau,\xi)=\tilde{\lambda}(\tau)^{-2}\int_0^\infty \phi(R,\xi)
\chi\left (\tilde{\lambda}(\tau)^{-1}\tfrac{\nu \tau-R}{c}\right)
R e_2\left (\nu \tilde{\lambda}(\tau)^{-1}\tau,
\tilde{\lambda}(\tau)^{-1}R\right )dR.
\end{equation} 

\subsection{Solution of the transport equation}
Our goal is to treat the entire right-hand side of Eq.~\eqref{eq:sysFourier} perturbatively. 
To this end it is necessary to be able to solve the two decoupled equations
\begin{align}
\label{eq:xd}
x_d''(\tau)+\xi_d x_d(\tau)&=b_d(\tau) \\
\label{eq:x}
\left [\partial_\tau-\beta_\nu(\tau)\left (2\xi\partial_\xi+\tfrac52+\tfrac{\xi\rho'(\xi)}{\rho(\xi)}
\right )
\right ]^2 
x(\tau,\xi)+\xi x(\tau,\xi)&=b(\tau,\xi)
\end{align}
for some given functions $b_d$ and $b$.
Recall that we are interested in decaying solutions as $\tau \to \infty$ and by variation of constants
it is readily seen
that
$$ x_d(\tau)=\int_{\tau_0}^\infty H_d(\tau,\tau')b_d(\tau')d\tau',\quad H_d(\tau,\tau'):=-\tfrac12 
|\xi_d|^{-\frac12} e^{-|\xi_d|^{1/2}|\tau-\tau'|} $$
for some constant $\tau_0$ is a solution to Eq.~\eqref{eq:xd} which behaves well at infinity.
For future reference we denote by $\mc{H}_d$ the solution operator, i.e.,
\begin{equation}
\label{eq:defHd}
\mc{H}_d f(\tau):=\int_{\tau_0}^\infty H_d(\tau,\tau')f(\tau')d\tau'. 
\end{equation}

In order to solve Eq.~\eqref{eq:x} we first define a new variable $y(\tau,\xi)$ by
$$ x(\tau,\xi)=\tilde{\lambda}(\tau)^\frac52 \rho(\xi)^{-\frac12}y(\tau,\xi). $$
Then, by recalling that 
$\beta_\nu(\tau)=\tilde{\lambda}'(\tau)\tilde{\lambda}(\tau)^{-1}$, we observe that
\[ \left [\partial_\tau-\tfrac52\beta_\nu(\tau)-\beta_\nu(\tau) \left
(2\xi\partial_\xi+\tfrac{\xi\rho'(\xi)}{\rho(\xi)}
\right ) \right]x(\tau,\xi)=\tilde{\lambda}(\tau)^\frac52 \rho(\xi)^{-\frac12}
[\partial_\tau-2\beta_\nu(\tau)\xi\partial_\xi]y(\tau,\xi)
\]
and thus, Eq.~\eqref{eq:x} is equivalent to
\begin{equation}
\label{eq:y}
[\partial_\tau-2\beta_\nu(\tau)\xi\partial_\xi]^2 y(\tau,\xi)+\xi y(\tau,\xi)=\tilde{\lambda}(\tau)^{-\frac52}
\rho(\xi)^\frac12 b(\tau,\xi).
\end{equation}
Now we solve Eq.~\eqref{eq:y} by the method of characteristics, i.e., we compute
$$ \tfrac{d}{d\tau}y(\tau,\xi(\tau))=\partial_\tau y(\tau,\xi(\tau))+\xi'(\tau)\partial_\xi
y(\tau,\xi(\tau)) $$
and by comparison with the differential operator in Eq.~\eqref{eq:y} we 
obtain the characteristic equation 
$\xi'(\tau)=-2\beta_\nu(\tau)\xi(\tau)$
which, by recalling that $\beta_\nu(\tau)=-(\frac{1}{\nu}-1)\tau^{-1}$ from Eq.~\eqref{eq:defbeta}, 
is readily solved as $\xi(\tau)=\gamma \tau^{2(\frac{1}{\nu}-1)}$ for some constant $\gamma$.
Thus, along the characteristic $\tau \mapsto (\tau,\gamma\tau^{2(\frac{1}{\nu}-1)})$, Eq.~\eqref{eq:y}
takes the form
\begin{equation}
\label{eq:tildex}
\tilde{y}''(\tau;\gamma)+\gamma \tau^{2(\frac{1}{\nu}-1)}\tilde{y}(\tau;\gamma)=\tilde{b}(\tau;\gamma) 
\end{equation}
where $\tilde{y}(\tau;\gamma)=y(\tau,\gamma\tau^{2(\frac{1}{\nu}-1)})$ and $\tilde{b}(\tau;\gamma)
=\tilde{\lambda}(\tau)^{-\frac52}\rho(\gamma\tau^{2(\frac{1}{\nu}-1)})^\frac12 b(\tau,\gamma\tau^{2(\frac{1}{\nu}-1)})$.
By setting $\tilde{y}(\tau;\gamma)=\tau^{-\frac12(\frac{1}{\nu}-1)}w(\nu \gamma^\frac12
\tau^{\frac{1}{\nu}})$ we infer that the homogeneous
version of Eq.~\eqref{eq:tildex} is equivalent to
$$ w''(z)+\left (1-\frac{(\frac{\nu}{2})^2-\frac14}{z^2}\right )w(z)=0 $$
where $z=\nu \gamma^\frac12 \tau^{\frac{1}{\nu}}$ and this identifies Eq.~\eqref{eq:tildex} as a Bessel equation.
Consequently, a fundamental system $\{\phi_j(\cdot;\gamma): j=0,1\}$ for the homogeneous version of 
Eq.~\eqref{eq:tildex} is given
by 
\begin{equation}
\label{eq:phi01}
\begin{aligned} 
\phi_0(\tau;\gamma)&=a_{\nu}\tau^\frac12 J_{\nu/2}(\nu \gamma^\frac12 \tau^{\frac{1}{\nu}}) \\
\phi_1(\tau;\gamma)&=b_{\nu}\tau^\frac12 Y_{\nu/2}(\nu \gamma^\frac12 \tau^{\frac{1}{\nu}})
\end{aligned}
\end{equation}
 where $J_{\nu/2}$, $Y_{\nu/2}$ are
the standard Bessel functions,
see e.g.~\cite{Olv74}, \cite{DLMF}, and $a_\nu$, $b_\nu$ are chosen such that
\begin{equation}
\label{eq:Besselasym}
a_\nu J_{\nu/2}(z)=\nu^{-\frac{\nu}{2}}z^{\frac{\nu}{2}}[1+O(z^2)],\quad 
b_\nu Y_{\nu/2}(z)=\nu^{\frac{\nu}{2}}z^{-\frac{\nu}{2}}[1+O(z^\nu)] 
\end{equation}
as $z \to 0+$.
This yields the asymptotics
\begin{equation}
\label{eq:phiasymsm}
\begin{aligned}
\phi_0(\tau;\gamma)&=\gamma^{\frac{\nu}{4}}\tau[1+O(\gamma \tau^{\frac{2}{\nu}})] \\
\phi_1(\tau;\gamma)&=\gamma^{-\frac{\nu}{4}}[1+O(\gamma^\frac{\nu}{2}\tau)] 
\end{aligned}
\end{equation}
for, say, $0<\gamma^\frac12 \tau^\frac{1}{\nu}\leq 1$ and the $O$-terms behave like symbols.
By evaluation at $\tau=0$ we also obtain the Wronskian $W(\phi_0(\cdot;\gamma),\phi_1(\cdot;\gamma))=-1$.
Furthermore, from the Hankel asymptotics we have $|H_{\nu/2}^{(j)}(z)|\lesssim z^{-\frac12}$
for $z\geq 1$, $j=1,2$ (see \cite{Olv74}, \cite{DLMF}) and thus, the relations
$J_{\nu/2}=\frac12 (H_{\nu/2}^{(1)}+H_{\nu/2}^{(2)})$ as well as $Y_{\nu/2}=\frac{1}{2i}
(H_{\nu/2}^{(1)}-H_{\nu/2}^{(2)})$ immediately yield the bound
\begin{equation}
\label{eq:phiasymlg}
|\phi_j(\tau;\gamma)|\lesssim \gamma^{-\frac14}\tau^{-\frac12(\frac{1}{\nu}-1)},\quad j=0,1 
\end{equation}
for $\gamma^\frac12 \tau^\frac{1}{\nu}\geq 1$.
Consequently, assuming sufficient decay of $\tilde{b}(\cdot;\gamma)$, a decaying solution to Eq.~\eqref{eq:tildex} is given by
$$ \tilde{y}(\tau;\gamma)=\int_\tau^\infty [\phi_1(\tau;\gamma)\phi_0(\sigma;\gamma)
-\phi_0(\tau;\gamma)\phi_1(\sigma;\gamma)
]\tilde{b}(\sigma;\gamma)d\sigma. $$
In order to obtain an expression for $x(\tau,\xi)$, we set $\gamma=\xi \tau^{-2(\frac{1}{\nu}-1)}$ and
this yields
\begin{equation}
\label{eq:H}
x(\tau,\xi)=\int_\tau^\infty H_c(\tau,\sigma,\xi)b(\sigma,(\tfrac{\sigma}{\tau})^{2(\frac{1}{\nu}-1)}
\xi)d\sigma=:\mc{H}_c b(\tau,\xi)
\end{equation}
with 
\begin{align}
\label{eq:defH}
H_c(\tau,\sigma,\xi)=&\tilde{\lambda}(\tau)^\frac52 \rho(\xi)^{-\frac12}
\Big[\phi_1\left (\tau;\xi \tau^{-2(\frac{1}{\nu}-1)} \right)
\phi_0\left (\sigma;\xi \tau^{-2(\frac{1}{\nu}-1)} \right)\\
&-\phi_0\left (\tau;\xi \tau^{-2(\frac{1}{\nu}-1)}\right)
\phi_1\left (\sigma;\xi \tau^{-2(\frac{1}{\nu}-1)}\right) \Big]\tilde{\lambda}(\sigma)^{-\frac52}
\rho \left ((\tfrac{\sigma}{\tau})^{2(\frac{1}{\nu}-1)}\xi \right)^{\frac12}. \nonumber
\end{align}
Now we establish bounds for $H_c(\tau,\sigma,\xi)$.

\begin{lemma}
\label{lem:boundsH}
The function $H_c$ defined by Eq.~\eqref{eq:defH} satisfies the bounds
$$ |H_c(\tau,\sigma,\xi)|\lesssim (\tfrac{\sigma}{\tau})^{\frac72|\frac{1}{\nu}-1|} \left \{ \begin{array}{ll}
\xi^{-\frac12} & \tau \xi^\frac12\geq 1,\: \sigma \xi^\frac12 \geq 1 \\
\tau^{\frac12(1-\nu)}\xi^{-\frac14(1+\nu)} & 0<\tau \xi^\frac12 \leq 1,\: \sigma \xi^\frac12 \geq 1 \\
\sigma & 0<\tau \xi^\frac12 \leq 1,\: 0<\sigma \xi^\frac12 \leq 1
\end{array} \right.
$$
for all $1\leq \tau \leq \sigma$ and $\xi > 0$.
\end{lemma}

\begin{proof}
Recall that $\tilde{\lambda}(\tau)=(\nu\tau)^{-(\frac{1}{\nu}-1)}$ and thus,
$\tilde{\lambda}(\tau)^\frac52 \tilde{\lambda}(\sigma)^{-\frac52}=(\tfrac{\sigma}{\tau})^{\frac52(\frac{1}{\nu}-1)}$.
Furthermore, if $\xi\geq 1$ we have $|\rho(\xi)|^{-\frac12}\lesssim \xi^{-\frac14}$ by Lemma
\ref{lem:lgxi} and, if $(\tfrac{\sigma}{\tau})^{2(\frac{1}{\nu}-1)}\xi\geq 1$, we infer
\[ \left |\rho(\xi)^{-\frac12}\rho\left ((\tfrac{\sigma}{\tau})^{2(\frac{1}{\nu}-1)}\xi \right )^\frac12
\right |\lesssim \xi^{-\frac14}(\tfrac{\sigma}{\tau})^{\frac12(\frac{1}{\nu}-1)}\xi^\frac14 \lesssim 
 (\tfrac{\sigma}{\tau})^{\frac12|\frac{1}{\nu}-1|}\]
again by Lemma \ref{lem:lgxi}.
Note that we always have $(\tfrac{\sigma}{\tau})^{\frac12(\frac{1}{\nu}-1)}\lesssim 
(\tfrac{\sigma}{\tau})^{\frac12|\frac{1}{\nu}-1|}$ regardless of the sign of $\frac{1}{\nu}-1$
since $1\leq \tau\leq \sigma$ is assumed throughout this proof.
If, on the other hand, $0<(\tfrac{\sigma}{\tau})^{2(\frac{1}{\nu}-1)} \xi\leq 1$ we obtain
\[ \left |\rho(\xi)^{-\frac12}\rho\left ((\tfrac{\sigma}{\tau})^{2(\frac{1}{\nu}-1)}\xi \right )^\frac12
\right |\lesssim \xi^{-\frac14}(\tfrac{\sigma}{\tau})^{-\frac12(\frac{1}{\nu}-1)}\xi^{-\frac14}
\lesssim (\tfrac{\sigma}{\tau})^{\frac12|\frac{1}{\nu}-1|}\]
since $|\rho(\xi)|^\frac12 \lesssim \xi^{-\frac14}$ for $0<\xi\leq 1$ by Lemma \ref{lem:smxi}.

In the case $0<\xi\leq 1$ we have $|\rho(\xi)|^{-\frac12}\lesssim \xi^\frac14$ and thus,
either
\[  \left |\rho(\xi)^{-\frac12}\rho\left ((\tfrac{\sigma}{\tau})^{2(\frac{1}{\nu}-1)}\xi \right )^\frac12
\right |\lesssim\xi^\frac14
(\tfrac{\sigma}{\tau})^{-\frac12(\frac{1}{\nu}-1)}\xi^{-\frac14} \lesssim 
(\tfrac{\sigma}{\tau})^{\frac12|\frac{1}{\nu}-1|} \]
or
\[ \left |\rho(\xi)^{-\frac12}\rho\left ((\tfrac{\sigma}{\tau})^{2(\frac{1}{\nu}-1)}\xi \right )^\frac12
\right |\lesssim \xi^\frac14 
(\tfrac{\sigma}{\tau})^{\frac12(\frac{1}{\nu}-1)}\xi^\frac14 \lesssim 
(\tfrac{\sigma}{\tau})^{\frac12|\frac{1}{\nu}-1|} \]
by Lemma \ref{lem:lgxi} depending on whether $0<(\tfrac{\sigma}{\tau})^{\frac12(\frac{1}{\nu}-1)}\xi\leq 1$
or $(\tfrac{\sigma}{\tau})^{\frac12(\frac{1}{\nu}-1)}\xi \geq 1$.
We conclude that
\begin{equation}
\label{eq:factor}
\left |\tilde{\lambda}(\tau)^\frac52 \tilde{\lambda}(\sigma)^{-\frac52}
\rho(\xi)^{-\frac12}\rho\left ((\tfrac{\sigma}{\tau})^{2(\frac{1}{\nu}-1)}\xi \right )^\frac12
\right | \lesssim (\tfrac{\sigma}{\tau})^{3|\frac{1}{\nu}-1|} 
\end{equation}
for all $1\leq \tau\leq \sigma$ and $\xi>0$.

It remains to estimate the terms involving $\phi_j$.
We have different asymptotic descriptions of $\phi_j(\tau;\gamma)$ depending on whether 
$0<\gamma^\frac12 \tau^{\frac{1}{\nu}}\leq 1$ or $\gamma^\frac12 \tau^{\frac{1}{\nu}}\geq 1$.
For $\gamma=\xi \tau^{-2(\frac{1}{\nu}-1)}$ this distinction reads
$0<\tau\xi^\frac12\leq 1$ or $\tau \xi^\frac12 \geq 1$.
Thus, in principle we have to deal with the four cases
\begin{enumerate}
\item $\tau \xi^\frac12 \geq 1$ and $\sigma \xi^\frac12\geq 1$
\item $0<\tau \xi^\frac12 \leq 1$ and $\sigma \xi^\frac12 \geq 1$
\item $0<\tau \xi^\frac12 \leq 1$ and $0<\sigma \xi^\frac12 \leq 1$
\item $\tau \xi^\frac12 \geq 1$ and $0<\sigma \xi^\frac12 \leq 1$.
\end{enumerate}
However, since we are only interested in $\tau \leq \sigma$, case $(4)$ is void.
\begin{enumerate}
\item We use the bound from the Hankel asymptotics stated in Eq.~\eqref{eq:phiasymlg} to obtain
$$ |\phi_0(\tau;\gamma)\phi_1(\sigma;\gamma)|\lesssim \gamma^{-\frac12}\tau^{-\frac12(\frac{1}{\nu}-1)}
\sigma^{-\frac12(\frac{1}{\nu}-1)}\lesssim (\tfrac{\sigma}{\tau})^{\frac12|\frac{1}{\nu}-1|}\xi^{-\frac12} $$
by evaluation at $\gamma=\xi 
\tau^{-2(\frac{1}{\nu}-1)}$.
This bound is symmetric in $\tau$ and $\sigma$ and thus, by Eq.~\eqref{eq:factor}, we infer
$$ |H_c(\tau,\sigma,\xi)|\lesssim (\tfrac{\sigma}{\tau})^{\frac72|\frac{1}{\nu}-1|} \xi^{-\frac12}. $$
\item From Eqs.~\eqref{eq:phiasymsm} and \eqref{eq:phiasymlg} we have
\begin{align*} 
|\phi_1(\tau;\gamma)\phi_0(\sigma;\gamma)|&\lesssim \gamma^{-\frac14(1+\nu)}\sigma^{-\frac12(\frac{1}{\nu}-1)}\\
|\phi_0(\tau;\gamma)\phi_1(\sigma;\gamma)|&\lesssim \gamma^{-\frac14(1-\nu)}\tau \sigma^{-\frac12(\frac{1}{\nu}-1)} 
\lesssim \gamma^{-\frac14(1+\nu)}\sigma^{-\frac12(\frac{1}{\nu}-1)}
\end{align*}
and thus,
\begin{align*}|H_c(\tau,\sigma,\xi)|&\lesssim (\tfrac{\sigma}{\tau})^{3|\frac{1}{\nu}-1|} \tau^{\frac12(\frac{1}{\nu}-\nu)}
\sigma^{-\frac12(\frac{1}{\nu}-1)}\xi^{-\frac14(1+\nu)} \\
&\lesssim 
(\tfrac{\sigma}{\tau})^{\frac72|\frac{1}{\nu}-1|} \tau^{\frac12(1-\nu)}\xi^{-\frac14(1+\nu)} 
\end{align*}
by Eq.~\eqref{eq:factor}.

\item Eq.~\eqref{eq:phiasymsm} yields
$$ |\phi_0(\tau;\gamma)\phi_1(\sigma;\gamma)|\lesssim \tau $$
and this implies
$|H_c(\tau,\sigma,\xi)|\lesssim (\tfrac{\sigma}{\tau})^{3|\frac{1}{\nu}-1|} (\tau+\sigma) 
\lesssim (\tfrac{\sigma}{\tau})^{3|\frac{1}{\nu}-1|} \sigma$.
\end{enumerate}
\end{proof}

We shall also require bounds for the differentiated kernel 
\begin{equation}
\label{eq:defdiffH}
\hat{H}_c(\tau,\sigma,\xi):=
[\partial_\tau-\beta_\nu(\tau)(2\xi \partial_\xi+\tfrac52+\tfrac{\xi\rho'(\xi)}{\rho(\xi)})]
H_c(\tau,\sigma,\xi) 
\end{equation}
where, as always, $\beta_\nu(\tau)
=-(\frac{1}{\nu}-1)\tau^{-1}$.
These are established next.

\begin{lemma}
\label{lem:diffH}
The differentiated kernel satisfies the bounds
\[ |\hat{H}_c(\tau,\sigma,\xi)|\lesssim (\tfrac{\sigma}{\tau})^{\frac72|\frac{1}{\nu}-1|}
\left \{ \begin{array}{ll}
1 & \tau \xi^\frac12\geq 1,\: \sigma \xi^\frac12 \geq 1 \\
\tau^{-\frac12(1-\nu)}\xi^{-\frac14(1-\nu)} & 0<\tau \xi^\frac12 \leq 1,\: \sigma \xi^\frac12 \geq 1 \\
1 & 0<\tau \xi^\frac12 \leq 1,\: 0<\sigma \xi^\frac12 \leq 1
\end{array} \right.
\]
for all $1\leq \tau \leq \sigma$ and $\xi > 0$.
\end{lemma}

\begin{proof}
Note that for any differentiable function $f$ of two variables we have
\begin{align*}
&[\partial_\tau-\beta_\nu(\tau)(2\xi \partial_\xi+\tfrac52+\tfrac{\xi\rho'(\xi)}{\rho(\xi)})]
\left (\tilde{\lambda}(\tau)^\frac52 \rho(\xi)^{-\frac12}f(\tau,\xi)
\right ) \\
&\quad =\tilde{\lambda}(\tau)^\frac52 \rho(\xi)^{-\frac12}[\partial_\tau-2\beta_\nu(\tau)\xi\partial_\xi]
f(\tau,\xi).
\end{align*}
By Eq.~\eqref{eq:defH} and
\begin{align*} [\partial_\tau-2\beta_\nu(\tau)\xi\partial_\xi]\phi_j\left (\tau; \xi\tau^{-2(\frac{1}{\nu}-1)}
\right )&=\partial_1 \phi_j\left (\tau; \xi\tau^{-2(\frac{1}{\nu}-1)}
\right ) \\
[\partial_\tau-2\beta_\nu(\tau)\xi\partial_\xi]\phi_j\left (\sigma; \xi\tau^{-2(\frac{1}{\nu}-1)}
\right )&=0
\end{align*}
for $j=0,1$ we therefore obtain
\begin{align*}
\hat{H}_c(\tau,\sigma,\xi)=&\tilde{\lambda}(\tau)^\frac52 \rho(\xi)^{-\frac12}\partial_1 \phi_1
\left (\tau;\xi \tau^{-2(\frac{1}{\nu}-1)} \right)
\phi_0\left (\sigma;\xi \tau^{-2(\frac{1}{\nu}-1)} \right)\\
&-\partial_1 \phi_0\left (\tau;\xi \tau^{-2(\frac{1}{\nu}-1)}\right)
\phi_1\left (\sigma;\xi \tau^{-2(\frac{1}{\nu}-1)}\right)
\tilde{\lambda}(\sigma)^{-\frac52}
\rho \left ((\tfrac{\sigma}{\tau})^{2(\frac{1}{\nu}-1)}\xi \right)^{\frac12}. \nonumber
\end{align*}
From Eq.~\eqref{eq:phiasymsm} we immediately infer
\begin{equation}
\label{eq:pphism}
\begin{aligned}
\partial_1 \phi_0 (\tau;\gamma)&=\gamma^{\frac{\nu}{4}}[1+O(\gamma \tau^{\frac{2}{\nu}})] \\
\partial_1 \phi_1(\tau;\gamma)&= O(\gamma^{\frac{\nu}{4}})
\end{aligned}
\end{equation}
for $0<\gamma^\frac12 \tau^{\frac{1}{\nu}} \leq 1$.
Furthermore, since
\begin{align*}
\partial_1 \phi_0(\tau;\gamma)&=\tfrac12 \tau^{-\frac12}a_\nu J_{\nu/2}(\nu \gamma^\frac12
\tau^{\frac{1}{\nu}})+\gamma^\frac12 \tau^{\frac{1}{\nu}-\frac12}a_\nu J'_{\nu/2}
(\nu \gamma^\frac12 \tau^{\frac{1}{\nu}}) \\
\partial_1 \phi_1(\tau;\gamma)&=\tfrac12 \tau^{-\frac12}b_\nu Y_{\nu/2}(\nu \gamma^\frac12
\tau^{\frac{1}{\nu}})+\gamma^\frac12 \tau^{\frac{1}{\nu}-\frac12}b_\nu Y'_{\nu/2}
(\nu \gamma^\frac12 \tau^{\frac{1}{\nu}}),
\end{align*}
see Eq.~\eqref{eq:phi01},
the identity $C'_{\nu/2}=\frac12 (C_{\nu/2-1}-C_{\nu/2+1})$, $C \in \{J,Y\}$ \cite{DLMF},
and the asymptotics of the Hankel functions yield $|C'_{\nu/2}(z)|\lesssim z^{-\frac12}$ for 
$z \gtrsim 1$ and thus, 
\begin{equation}
\label{eq:pphilg}
|\partial_1\phi_j(\tau;\gamma)|\lesssim \gamma^{-\frac14}\tau^{-\frac12(1+\frac{1}{\nu})}
+\gamma^{\frac14}\tau^{\frac12(\frac{1}{\nu}-1)}\lesssim \gamma^{\frac14}\tau^{\frac12(\frac{1}{\nu}-1)} 
\end{equation}
for $\gamma^\frac12 \tau^{\frac{1}{\nu}}\geq 1$ and $j=0,1$.

As in the proof of Lemma \ref{lem:boundsH} we now distinguish three cases and we always assume
$1\leq \tau\leq \sigma$.
\begin{enumerate}
\item If $\tau \xi^\frac12\geq 1$ and $\sigma \xi^\frac12 \geq 1$ we use 
Eqs.~\eqref{eq:phiasymlg} and \eqref{eq:pphilg} to conclude
\[ |\partial_1 \phi_1(\tau; \gamma)\phi_0(\sigma; \gamma)|+
|\partial_1 \phi_0(\tau; \gamma)\phi_1(\sigma; \gamma)|\lesssim (\tfrac{\sigma}{\tau})^{\frac12|\frac{1}{\nu}-1|} \]
which, by Eq.~\eqref{eq:factor}, implies
$|\hat{H}_c(\tau,\sigma,\xi)|\lesssim (\tfrac{\sigma}{\tau})^{\frac72 |\frac{1}{\nu}-1|}$.

\item If $0<\tau \xi^\frac12\leq 1$ and $\sigma \xi^\frac12\geq 1$ we obtain
\[ |\partial_1 \phi_1(\tau;\gamma)\phi_0(\sigma;\gamma)|+|\partial_1 \phi_0(\tau;\gamma)
\phi_1(\sigma;\gamma)|\lesssim \gamma^{-\frac14(1-\nu)}\sigma^{-\frac12(\frac{1}{\nu}-1)} \]
by Eqs.~\eqref{eq:pphism} and \eqref{eq:phiasymlg}. Hence, 
upon setting $\gamma=\xi \tau^{-2(\frac{1}{\nu}-1)}$, we conclude
\[ |\hat{H}_c(\tau,\sigma,\xi) |\lesssim (\tfrac{\sigma}{\tau})^{\frac72|\frac{1}{\nu}-1|}
\tau^{-\frac12(1-\nu)}\xi^{-\frac14(1-\nu)}. \] 

\item In the case $0<\tau \xi^\frac12\leq 1$ and $0<\sigma \xi^\frac12\leq 1$ we have, by
Eqs.~\eqref{eq:phiasymsm} and \eqref{eq:pphism}, 
\[ |\partial_1 \phi_1(\tau;\gamma)\phi_0(\sigma;\gamma)|+|\partial_1 \phi_0(\tau;\gamma)
\phi_1(\sigma;\gamma)|\lesssim 1+\gamma^\frac{\nu}{2} \sigma\lesssim 1 \]
which yields
$|\hat{H}_c(\tau,\sigma,\xi)|\lesssim (\tfrac{\sigma}{\tau})^{3|\frac{1}{\nu}-1|}$
by Eq.~\eqref{eq:factor}.
\end{enumerate}
\end{proof}

\subsection{Estimates for the solution operator}

In order to set up our main contraction argument for Eq.~\eqref{eq:sysFourier} we have to introduce
appropriate function spaces.
\begin{definition}
\label{def:spaces}
For $\delta,\alpha \in \R$ and $p \in [1,\infty)$ we define norms $\|\cdot\|_\X$ and
$\|\cdot\|_\Y$ by
\begin{align*} \|f\|_\X&:=\left (\int_0^\infty \left |f(\xi)\left (\xi \langle \xi\rangle^{-1}
\right)^{\frac12-\delta}
\right |^p
d\xi \right )^{1/p}+\left (\int_0^\infty |f(\xi)|^2 \xi \langle \xi \rangle^{2\alpha}
\rho(\xi)d\xi \right )^{1/2}, \\
\|f\|_\Y&:=\|f\|_{L^p(0,\infty)}+\left (\int_0^\infty |f(\xi)|^2 \langle \xi \rangle^{2\alpha}
\rho(\xi)d\xi \right )^{1/2}.
\end{align*}
Furthermore, for a function $b$ of two variables and a Banach space $X$ we write
$$ \|b\|_{L^{\infty,\beta}_{\tau_0}X}:=\sup_{\tau \geq \tau_0}\tau^\beta 
\|b(\tau,\cdot)\|_{X} $$
where $\beta\geq 0$ and $\tau_0>0$ (in what follows we always assume
$\tau_0$ to be sufficiently large).
\end{definition}

For the following it is convenient to introduce
the notation 
\[ \mc B_{c,\nu}b(\tau,\xi):=[\partial_\tau-\beta_\nu(\tau)(2\xi \partial_\xi
+\tfrac52+\tfrac{\xi\rho'(\xi)}{\rho(\xi)})]b(\tau,\xi). \]

\begin{proposition}
\label{prop:H}
Fix a $\delta$ with $2|\frac{1}{\nu}-1|<\delta<\frac12$ 
and let $p \in (1,\infty)$ be so large that 
\[ p'(1-\delta+2|\tfrac{1}{\nu}-1|)<1 \] 
where $p'$ is the H\"older
conjugate of $p$, i.e., $\frac{1}{p}+\frac{1}{p'}=1$.
Suppose further that $\tau_0\geq 1$, $\beta\geq \frac52$ and 
$\alpha \in [0,1]$ be fixed.
Then we have the bounds
\begin{align*}
\|\mc{H}_c b\|_{L^{\infty,\beta-1-2\delta}_{\tau_0}\X} &\lesssim \|b\|_{L^{\infty,\beta}_{\tau_0}\Y} \\
\|\mc B_{c,\nu}\mc{H}_c b\|_{L^{\infty,\beta-1}_{\tau_0}\Y} &\lesssim \|b\|_{L^{\infty,\beta}_{\tau_0}\Y}.
\end{align*}
\end{proposition}

\begin{proof}
Let $q\in (1,\infty)$.
By H\"older's inequality we have
\begin{align*} |\mc{H}_cb(\tau,\xi)|&\leq \int_\tau^\infty |H_c(\tau,\sigma,\xi)
b(\sigma,(\tfrac{\sigma}{\tau})^{2(\frac{1}{\nu}-1)}\xi)|d\sigma \\
&\leq A_{\mu,q'}(\tau,\xi)\left (\int_\tau^\infty |\sigma^\mu 
b(\sigma,(\tfrac{\sigma}{\tau})^{2(\frac{1}{\nu}-1)}\xi)|^q d\sigma \right )^{1/q}
\end{align*}
with $\mu \in \R$ and
\[ A_{\mu,q'}(\tau,\xi):=\left ( \int_\tau^\infty |\sigma^{-\mu}H_c(\tau,\sigma,\xi)|^{q'}d\sigma 
\right )^{1/q'}. \]
We claim that 
\begin{equation} 
\label{eq:A}
A_{\mu,q'}(\tau,\xi)\lesssim \left \{ \begin{array}{ll}
\tau^{-\mu+\frac{1}{q'}}\xi^{-\frac12} & \tau\xi^\frac12 \geq 1 \\
\tau^{-\mu+\frac{1}{q'}+1} & 0<\tau\xi^\frac12\leq 1 \end{array} \right . 
\end{equation}
provided that $\mu>1+\frac{1}{q'}+5|\frac{1}{\nu}-1|$.
Indeed, if $\tau\xi^\frac12\geq 1$ we have from Lemma \ref{lem:boundsH} the bound
$|H_c(\tau,\sigma,\xi)|\lesssim (\tfrac{\sigma}{\tau})^{\frac72|\frac{1}{\nu}-1|}\xi^{-\frac12}$ and
this implies the first estimate in Eq.~\eqref{eq:A}.
In order to prove the second bound in Eq.~\eqref{eq:A} we use $|H_c(\tau,\sigma,\xi)|\lesssim 
(\tfrac{\sigma}{\tau})^{5|\frac{1}{\nu}-1|}\sigma$
from Lemma \ref{lem:boundsH} which yields
$A_{\mu,q'}(\tau,\xi)\lesssim \tau^{-\mu+\frac{1}{q'}+1}$ in the case $0<\tau \xi^\frac12\leq 1$
as claimed.

Now note that Eq.~\eqref{eq:A} implies 
\[ A_{\mu,q'}(\tau,\xi)\lesssim \tau^{-\mu+\frac{1}{q'}+2\delta} \xi^{-\frac12+\delta} \lesssim
\tau^{-\mu+\frac{1}{q'}+2\delta}(\xi \langle \xi \rangle^{-1})^{-\frac12+\delta} \]
for all $\xi>0$ and thus, 
by interchanging the order of integration, we obtain
\begin{align*} \int_0^\infty \left |\mc H_c b(\tau,\xi) (\xi \langle \xi \rangle^{-1})^{\frac12-\delta}
\right |^p d\xi&\lesssim \int_\tau^\infty \int_0^\infty 
\left |A_{\mu,p'}(\tau,\xi)(\xi\langle \xi \rangle^{-1})^{\frac12-\delta}\right |^p 
\left |\sigma^\mu 
b(\sigma,(\tfrac{\sigma}{\tau})^{2(\frac{1}{\nu}-1)}\xi)\right |^pd\xi d\sigma \\
&\lesssim \tau^{p(-\mu+\frac{1}{p'}+2\delta+2(\frac{1}{\nu}-1))}\int_\tau^\infty \int_0^\infty
|\sigma^{\mu-2(\frac{1}{\nu}-1)} b(\sigma,\eta)|^p d\eta d\sigma.
\end{align*}
Now we set $\mu=\beta-\delta$ which is admissible since $\beta-\delta>1+\frac{1}{p'}+5|\frac{1}{\nu}-1|$
provided $\nu$ is sufficiently close to $1$ which we may safely assume.
Hence, we infer
\begin{align*} \int_0^\infty \left |\mc H_c b(\tau,\xi) (\xi \langle \xi \rangle^{-1})^{\frac12-\delta}
\right |^p d\xi &\lesssim \tau^{p(-\beta+\frac{1}{p'}+3\delta+2(\frac{1}{\nu}-1))}\|b\|_{L^{\infty,\beta}_{\tau_0}L^p(0,\infty)}^p
\int_\tau^\infty \sigma^{-p(\delta+2(\frac{1}{\nu}-1))}d\sigma \\
&\lesssim \tau^{p(-\beta+1+2\delta)}\|b\|_{L^{\infty,\beta}_{\tau_0}\Y}^p
\end{align*}
where the last step is justified since $1-\delta-2(\frac{1}{\nu}-1)<\frac{1}{p'}=1-\frac{1}{p}$ and this implies 
$-p(\delta+2(\frac{1}{\nu}-1))<-1$.

For the $L^2$ based part in $\|\cdot\|_\X$ we proceed similarly and obtain from Eq.~\eqref{eq:A}
the bound
$A_{\mu,2}(\tau,\xi)\lesssim \tau^{-\mu+\frac12}\xi^{-\frac12}$
for all $\xi>0$.
This shows
\begin{align*}
\int_0^\infty |\mc{H}_cb(\tau,\xi)|^2 \xi \langle \xi \rangle^{2\alpha}\rho(\xi)d\xi
&\lesssim \tau^{2(-\mu+\frac12)}\int_\tau^\infty \int_0^\infty 
\left |\sigma^\mu b(\sigma,\omega(\tau,\sigma)^{-1}\xi) \right |^2
\langle \xi \rangle^{2\alpha}\rho(\xi)d\xi d\sigma \\
&=\tau^{2(-\mu+\frac12)}\int_\tau^\infty \int_0^\infty 
\left |\sigma^\mu b(\sigma,\eta) \right |^2
\langle \omega(\tau,\sigma)\eta \rangle^{2\alpha}\omega(\tau,\sigma)\rho(\omega(\tau,\sigma)\eta)
d\eta d\sigma \\
\end{align*}
where we write $\omega(\tau,\sigma)=(\tfrac{\sigma}{\tau})^{-2(\frac{1}{\nu}-1)}$.
We clearly have
$\langle \omega(\tau,\sigma)\eta \rangle^{2\alpha}\lesssim (\frac{\sigma}{\tau})^{4|\frac{1}{\nu}-1|}
\langle \eta \rangle^{2\alpha}$ for all
$\eta>0$ and also, $\omega(\tau,\rho)\rho(\omega(\tau,\sigma)\eta)\lesssim (\frac{\sigma}{\tau})^{3|\frac{1}{\nu}-1|}
\rho(\eta)$ provided that 
$\omega(\tau,\sigma)\eta\geq 1$, cf.~Lemma \ref{lem:lgxi}.
In the case $0<\omega(\tau,\sigma)\eta\leq 1$, Lemma \ref{lem:smxi} implies
\[ \omega(\tau,\sigma)\rho(\omega(\tau,\sigma)\eta)\lesssim \omega(\tau,\sigma)^{\frac12}
\eta^{-\frac12}\lesssim (\tfrac{\sigma}{\tau})^{|\frac{1}{\nu}-1|}\rho(\eta). \]
Consequently, by choosing $\mu=\beta-\frac58$ we infer
\begin{align*}\int_0^\infty |\mc{H}_cb(\tau,\xi)|^2 \xi \langle \xi \rangle^{2\alpha}\rho(\xi)d\xi
&\lesssim \tau^{2(-\beta+\frac98-\frac72|\frac{1}{\nu}-1|)}\|b\|_{L^{\infty,\beta}_{\tau_0}\Y}
\int_\tau^\infty \sigma^{-\frac54+7|\frac{1}{\nu}-1|}d\sigma \\
&\lesssim \tau^{2(-\beta+1)}\|b\|_{L^{\infty,\beta}_{\tau_0}\Y}
\end{align*}
and this finishes the proof of the first estimate.

For the second bound note that the operator $\mc B_{c,\nu}\mc{H}_c$ has the kernel 
$\hat{H}(\tau,\sigma,\xi)$ from Lemma
\ref{lem:diffH} and based on the bounds given there it is straightforward to prove the claimed
estimate by repeating the above arguments.
\end{proof}

It is also an easy exercise to prove an appropriate bound for the discrete part $\mc{H}_d$.

\begin{lemma}
\label{lem:Hd}
Let $\beta>0$ and suppose $b_d \in L^{\infty,\beta}_{\tau_0}$. Then
$$ \|\mc{H}_db_d\|_{L^{\infty,\beta}_{\tau_0}}\lesssim \|b_d\|_{L^{\infty,\beta}_{\tau_0}},\quad
\|(\mc{H}_db_d)'\|_{L^{\infty,\beta}_{\tau_0}}\lesssim \|b_d\|_{L^{\infty,\beta}_{\tau_0}}. $$
\end{lemma}

\begin{proof}
By definition (Eq.~\eqref{eq:defHd}) we have
\begin{align*} 
\mc{H}_d b_d(\tau)&=-\tfrac12 |\xi_d|^{-\frac12} e^{-|\xi_d|^{1/2}\tau}\int_{\tau_0}^\tau 
e^{|\xi_d|^{1/2}\sigma}
b_d(\sigma)d\sigma \\
&\quad -\tfrac12 |\xi_d|^{-\frac12} e^{|\xi_d|^{1/2}\tau}\int_\tau^\infty e^{-|\xi_d|^{1/2}\sigma}
b_d(\sigma)d\sigma \\ 
&=:I_1(\tau)+I_2(\tau).
\end{align*}
It is evident that $|I_2(\tau)|\lesssim \langle \tau \rangle^{-\beta}$ and in order to estimate
$I_1(\tau)$ 
we note that
\[ |I_1(\tau)|\lesssim \sup_{\sigma>\tau_0}\sigma^\beta |b_d(\sigma)|
e^{-|\xi_d|^{1/2}\tau}\int_{\tau_0}^\tau 
e^{|\xi_d|^{1/2}\sigma}
\sigma^{-\beta}d\sigma  \]
and the first assertion follows by performing one integration by parts.
The proof of the second bound is identical.
\end{proof}

\section{Estimates for the nonlinear and inhomogeneous terms}
\label{sec:nonlinhom}
We provide estimates (in terms of the spaces in Definition \ref{def:spaces}) for the various
contributions on the right-hand side of our main equation \eqref{eq:sysFourier} that do not
involve the operator $\mc{K}$ from the transference identity.
Thus, this section is mainly concerned with the nonlinear contributions.
In order to treat the nonlinearity, we first discuss mapping properties of the distorted Fourier transform $\mc F$.
These allow us to transfer the problem to the physical side where the nonlinearity can be estimated
using standard tools.
The main ingredients are basic inequalities and interpolation theory of Sobolev
spaces as well as the fractional Leibniz rule.
As a consequence, we infer the crucial contraction property of the nonlinearity on our
spaces.

\subsection{The inhomogeneous term}
\label{sec:inhom}
We start with the inhomogeneous term $\hat{e}_2$ as defined in Eq.~\eqref{eq:defhate2}.

\begin{lemma}
\label{lem:Fe2}
For any fixed $\epsilon>0$ we have 
$$ \hat{e}_2 \in L^{\infty,3-\epsilon-3|\frac{1}{\nu}-1|}\Y,\quad 
\hat{e}_2(\cdot,\xi_d)\in L^{\infty,3-\epsilon-3|\frac{1}{\nu}-1|} $$
for all $p > 1$ and $\alpha \in [0,\frac14)$.
\end{lemma}

\begin{proof}
We distinguish between $\xi\lesssim 1$ (including $\xi=\xi_d$) and $\xi \gtrsim 1$.
If $\xi\lesssim 1$ we have from Corollary \ref{cor:phi} the bound $|\phi(R,\xi)|\lesssim 1$ and
from Lemma \ref{lem:v1} we recall 
$$ \left |\tilde{\lambda}(\tau)^{-2}e_2\left (\nu \tilde{\lambda}(\tau)^{-1}\tau,
\tilde{\lambda}(\tau)^{-1}R\right ) \right |\lesssim \tilde{\lambda}(\tau)^{\frac12}\tau^{-4+\epsilon+\frac52|\frac{1}{\nu}-1|}\langle 
R \rangle^{-1} $$
in the truncated cone $r\leq t-c$ which corresponds to $R\leq \nu\tau-\tilde{\lambda}(\tau)c$.
Thus, by Eq.~\eqref{eq:defhate2} we obtain
\begin{align*} |\hat{e}_2(\tau,\xi)|&\lesssim \tilde{\lambda}(\tau)^{\frac12}\tau^{-4+\epsilon+\frac52|\frac{1}{\nu}-1|}
\int_0^\infty \chi\left (\tilde{\lambda}(\tau)^{-1}\tfrac{\nu \tau-R}{c}\right)dR \\ 
&\lesssim \tilde{\lambda}(\tau)^{\frac12}\tau^{-4+\epsilon+\frac52|\frac{1}{\nu}-1|} \int_0^{\tau}dR\lesssim 
\tau^{-3+\epsilon+3 |\frac{1}{\nu}-1|}
\end{align*} 
by recalling that $\tilde{\lambda}(\tau)\simeq \tau^{-(\frac{1}{\nu}-1)}$
since the cut-off localizes to the lightcone $R\leq \nu \tau\lesssim \tau$.

If $\xi\gtrsim 1$ we use 
$$ \phi(R,\xi)=O_\mathbb{C}(\xi^{-\frac12})e^{i\sqrt{\xi}R}[1+O_\mathbb{C}
(\langle R\rangle^{-3}\xi^{-\frac12})]
+O_\mathbb{C}(\xi^{-\frac12})e^{-i\sqrt{\xi}R}[1+O_\mathbb{C}
(\langle R\rangle^{-3}\xi^{-\frac12})]$$
with symbol behavior of all $O$-terms (Lemma \ref{lem:Jostlgxi} and Corollary \ref{cor:phi})
and perform one integration by parts to obtain
\begin{align*} 
|\hat{e}_2(\tau,\xi)|&\lesssim \tilde{\lambda}(\tau)^{\frac12}\tau^{-4+\epsilon+\frac52|\frac{1}{\nu}-1|}
\xi^{-1}\int_0^{\tau}\langle R \rangle^{-1}dR \\
&\quad +\tilde{\lambda}(\tau)^{-\frac12}\tau^{-4+\epsilon+\frac52|\frac{1}{\nu}-1|}\xi^{-1}\int_0^\infty \left |\chi' \left (
\tilde{\lambda}(\tau)^{-1}\tfrac{\nu \tau-R}{c}\right)\right |dR \\
&\lesssim \tilde{\lambda}(\tau)^{\frac12}\tau^{-3+\epsilon+\frac52|\frac{1}{\nu}-1|}\xi^{-1}+
\tilde{\lambda}(\tau)^{-\frac12}\tau^{-4+\epsilon+\frac52|\frac{1}{\nu}-1|}\xi^{-1}
\int_{\nu \tau-\tilde{\lambda}(\tau)c}^{\nu \tau-\frac12 \tilde{\lambda}(\tau)c}dR \\
&\lesssim \tau^{-3+\epsilon+3 |\frac{1}{\nu}-1|}\xi^{-1}.
\end{align*}
\end{proof}

\subsection{Mapping properties of $\mc{F}$}
It is of fundamental importance to understand the action of the distorted Fourier transform 
on the spaces $\X$.
In the following we use the standard Sobolev
spaces $W^{s,p}(\mathbb{R}^3)$ and in order to make sense of $\|f\|_{W^{s,p}(\mathbb{R}^3)}$ for a function 
$f: [0,\infty)\to \mathbb{C}$, we identify $f$ with the radial function $x \mapsto f(|x|)$ on $\mathbb{R}^3$.
In this sense we have, for instance, $\|f\|_{L^2(0,\infty)}\simeq \||\cdot|^{-1}f\|_{L^2(\mathbb{R}^3)}$.
Note that if $f: [0,\infty) \to \mathbb{C}$ is smooth and $f^{(2k-1)}(0)=0$ for
all $k \in \mathbb{N}$ then $x \mapsto f(|x|)$ is a smooth function on $\mathbb{R}^3$.
To begin with, we write
$$ \|f\|_{L^{2,\alpha}_\rho}^2:=\int_0^\infty |f(\xi)|^2 \langle \xi\rangle^{2\alpha}\rho(\xi)d\xi $$ 
and recall a fundamental result from \cite{KST09}.

\begin{lemma}
\label{lem:H2alpha}
Let $\alpha\geq 0$. Then
$$ \|(a,f)\|_{\mathbb{C}\times L^{2,\alpha}_\rho}\simeq \||\cdot|^{-1}\mc{F}^{-1}(a,f)\|_{H^{2\alpha}(\mathbb{R}^3)} $$
for all $(a,f) \in \mathbb{C} \times L^{2,\alpha}_\rho$ or, equivalently,
$$ \|\mc{F}(|\cdot|\,g)\|_{\mathbb{C}\times L^{2,\alpha}_\rho}\simeq \|g\|_{H^{2\alpha}(\mathbb{R}^3)} $$
for all radial $g \in H^{2\alpha}(\mathbb{R}^3)$.
\end{lemma}

\begin{proof}
This is a consequence of the unitarity of $\mc{U}$, see \cite{KST09}.
\end{proof}

\begin{lemma}
\label{lem:F-1}
Fix a small $\delta >0$ and let $p\geq 1$ be so large that $p'(1-\delta)<1$. Furthermore, denote by $\chi$ a smooth 
cut-off function which satisfies $\chi(\xi)=0$ for $\xi \in [0,1]$ and $\chi(\xi)=1$ for $\xi\geq 2$.
Then, for any $\alpha\geq 0$, the following estimates hold:
\begin{enumerate}
\item $\||\cdot|^{-1}\mc{F}^{-1}(a,\chi f)\|_{H^{2\alpha+1}(\mathbb{R}^3)}\lesssim \|(a,f)\|_{\mathbb{C}\times \X}$,
\item $\|\mc{F}^{-1}(a,(1-\chi)f)\|_{L^\infty(0,\infty)}\lesssim \|(a,f)\|_{\mathbb{C}\times \X}$,
\item $\||\cdot|^{-1}\mc{F}^{-1}(a,(1-\chi)f)\|_{L^q(\mathbb{R}^3)}\lesssim \|(a,f)\|_{\mathbb{C}\times \X}$ 
for any $q\in (3,\infty]$,
\item $\||\cdot|^{-1}\mc{F}^{-1}(a,(1-\chi)f)\|_{\dot{W}^{2\theta,\frac{2}{\theta+\frac{2}{q}(1-\theta)}}(\mathbb{R}^3)}\lesssim
\|(a,f)\|_{\mathbb{C}\times \X}$ for any $q \in (3,\infty)$ and $\theta \in [0,1]$.
\end{enumerate}
\end{lemma}

\begin{proof}
Recall that $\mc{F}^{-1}(a,f)$ is given by
$$ \mc{F}^{-1}(a,f)=a\frac{\phi(\cdot,\xi_d)}{\|\phi(\cdot,\xi_d)\|^2_{L^2(0,\infty)}}
+\int_0^\infty \phi(\cdot,\xi)f(\xi)\rho(\xi)d\xi. $$
Since $|\cdot|^{-1}\phi(|\cdot|,\xi_d) \in C^\infty(\R^3)$ with exponential decay towards infinity, 
the estimates for the discrete part follow immediately.
Thus, we focus on the integral term.
\begin{enumerate}
\item By Lemma \ref{lem:H2alpha} it suffices to note that
$$ \|\chi f\|_{L^{2,\alpha+1/2}_\rho}\lesssim \|\chi f\|_\X. $$

\item \label{item:Linf}
By Corollary \ref{cor:phi} we have $|\phi(R,\xi)|\lesssim 1$ for all $R,\xi \geq 0$ and thus, 
H\"older's inequality yields
$$ \left |\int_0^\infty \phi(R,\xi)[1-\chi(\xi)]f(\xi)\rho(\xi)d\xi \right |
\lesssim \left (\int_0^2 |f(\xi)\xi^{\frac12-\delta}|^p d\xi\right )^{1/p}
\left ( \int_0^2 |\xi^{-\frac12+\delta} \rho(\xi)|^{p'}d\xi \right )^{1/p'}$$
and the right-hand side of this inequality is finite and controlled by $\|f\|_\X$ since 
$|\rho(\xi)\xi^{-\frac12+\delta}|\lesssim \xi^{-1+\delta}$ for $\xi \in (0,2)$
(Lemma \ref{lem:smxi}) and $p'(-1+\delta)>-1$ by assumption on $p$.

\item \label{item:Lq}
From Lemma \ref{lem:Phi} and Corollary \ref{cor:phi} we obtain
$$ \sup_{\xi \in (0,2)}\left |\frac{\phi(R,\xi)}{R} \right| \lesssim \langle R \rangle^{-1} $$
for all $R \geq 0$ and by repeating the argument in (\ref{item:Linf}) this implies
$$ |R^{-1}\mc{F}^{-1}(a,(1-\chi)f)(R)|\lesssim \langle R \rangle^{-1}\|(a,(1-\chi)f)\|_{\mathbb{C}\times
\X} $$
which yields
$$ \||\cdot|^{-1}\mc{F}^{-1}(a,(1-\chi)f)\|_{L^q(\mathbb{R}^3)}\lesssim \|(a,(1-\chi)f)\|_{\mathbb{C}\times
\X} $$
for any $q \in (3,\infty]$.

\item
From Lemma \ref{lem:H2alpha} we have
$$ \||\cdot|^{-1}\mc{F}^{-1}(a,(1-\chi)f)\|_{\dot{H}^2(\mathbb{R}^3)}\lesssim \|(a,(1-\chi)f)\|_{\mathbb{C}\times
\X} $$
and the claim follows from this and (\ref{item:Lq}) by complex interpolation since 
$[L^q,\dot{H}^2]_\theta=
\dot{W}^{2\theta,\frac{2}{\theta+\frac{2}{q}(1-\theta)}}$, see \cite{BL1976}.
\end{enumerate}
\end{proof}

\subsection{The operator $\mc{N}_1$}
In order to estimate the operator $\mc{N}_1$, defined in Eq.~\eqref{eq:defR}, we 
first prove the following auxiliary result.

\begin{lemma}
\label{lem:FR}
Let $p,\delta$ be as in Lemma \ref{lem:F-1} and let
$\varphi: [0,\infty) \to \mathbb{R}$ be a function satisfying
$$ |\varphi^{(k)}(R)|\leq C_k \langle R \rangle^{-\gamma-k} $$
for some $\gamma > \frac32$ and all $R\geq 0$, $k\in \mathbb{N}_0$. 
Then $\varphi \in H^2(\mathbb{R}^3)$ and there exists a small $\epsilon>0$ such that
\begin{align*} 
\|\mc{F}(|\cdot|\varphi g)\|_{\mathbb{C}\times Y^{p,\frac18}}
\lesssim &\left [\|\varphi\|_{H^2(\mathbb{R}^3)}+\|\varphi\|_{L^1(0,\infty)}\right ]
\Big [\||\cdot|g_-\|_{L^\infty(0,\infty)}+\|g_-\|_{L^\infty(0,\infty)} \\
&+\|g_-\|_{\dot{W}^{\frac14,16-\epsilon}
(\mathbb{R}^3)}+\|g_+\|_{H^{\frac54}(\mathbb{R}^3)} \Big ]
\end{align*}
for all $g=g_-+g_+$ where $g_-$ and $g_+$ belong to the respective spaces on the right-hand side.
\end{lemma}

\begin{proof}
Note first that $\varphi \in L^2(\mathbb{R}^3)$.
Furthermore, 
we have 
$$\partial_k \partial_j \varphi(|x|)=\varphi''(|x|)\tfrac{x_j x_k}{|x|^2}
+\varphi'(|x|)\left [\tfrac{\delta_{jk}}{|x|}-\tfrac{x_jx_k}{|x|^3}\right ] $$
and thus, $|\partial_k \partial_j \varphi(|x|)|\lesssim |x|^{-1}\langle x\rangle^{-\gamma-1}$
which implies $\varphi \in H^2(\mathbb{R}^3)$.
Now recall that 
$$ \mc{F}(\varphi g)=\left (\begin{array}{c}
\int_0^\infty \phi(R,\xi_d)\varphi(R)g(R)dR \\
\int_0^\infty \phi(R,\cdot)\varphi(R)g(R)dR
\end{array} \right )$$
and $\|\cdot\|_\Y\simeq \|\cdot\|_{L^p(0,\infty)}+\|\cdot\|_{L^{2,\alpha}_\rho}$.
We proceed in four steps, estimating each component separately.

\begin{enumerate}
\item 
We note the estimate
$$ \|uv\|_{H^\frac14(\mathbb{R}^3)}\lesssim \|u\|_{H^{\frac78}(\mathbb{R}^3)}
\|v\|_{H^{\frac78}(\mathbb{R}^3)} $$
which is a consequence of the fractional Leibniz rule (\cite{T00}, 
p.~105, Proposition 1.1)
$$ \|uv\|_{H^\frac14(\mathbb{R}^3)}\lesssim \|u\|_{L^\frac{24}{5}(\mathbb{R}^3)}
\|v\|_{W^{\frac14,\frac{24}{7}}(\mathbb{R}^3)}+
\|v\|_{L^\frac{24}{5}(\mathbb{R}^3)}
\|u\|_{W^{\frac14,\frac{24}{7}}(\mathbb{R}^3)}$$
and the Sobolev embeddings $H^\frac78(\mathbb{R}^3)\hookrightarrow L^\frac{24}{5}(\mathbb{R}^3)$,
$H^\frac78(\mathbb{R}^3) \hookrightarrow W^{\frac14,\frac{24}{7}}(\mathbb{R}^3)$.
With these preparations at hand we immediately conclude from
Lemma \ref{lem:H2alpha} that
\begin{align*}\|\mc{F}(|\cdot|\varphi g_+)\|_{\mathbb{C}\times L^{2,\frac18}_\rho}&\simeq
\|\varphi g_+\|_{H^{\frac14}(\mathbb{R}^3)}\lesssim
\|\varphi\|_{H^\frac78(\mathbb{R}^3)}\|g_+\|_{H^{\frac78}(\mathbb{R}^3)} \\
&\lesssim \|\varphi\|_{H^2(\mathbb{R}^3)}\|g_+\|_{H^\frac54(\mathbb{R}^3)}.
\end{align*}

\item \label{item:Lp}
In order to estimate the $L^p$-part of $g_+$ note first that 
$|\phi(R,\xi)|\lesssim \langle \xi \rangle^{-\frac12}$ for all $R\geq 0$ by Lemmas \ref{lem:Phi},
\ref{lem:Jost}, \ref{lem:Jostlgxi} and Corollary \ref{cor:phi}.
This yields
\begin{align*} |[\mc{F}(|\cdot|\varphi g_+)]_2(\xi)|&\lesssim \langle \xi \rangle^{-\frac12}
 \|\varphi\|_{L^2(0,\infty)}\||\cdot|g_+\|_{L^2(0,\infty)} \\
 &\simeq \langle \xi \rangle^{-\frac12}
 \|\varphi\|_{L^2(0,\infty)}\|g_+\|_{L^2(\mathbb{R}^3)} \\
 &\lesssim \langle \xi \rangle^{-\frac12}\|\varphi\|_{H^2(\mathbb{R}^3)}
 \|g_+\|_{H^{\frac54}(\mathbb{R}^3)}
 \end{align*}
 by noting that $\|\varphi\|_{L^2(0,\infty)}\lesssim \|\varphi\|_{L^\infty(\mathbb{R}^3)}
 +\|\varphi\|_{L^2(\mathbb{R}^3)}\lesssim \|\varphi\|_{H^2(\mathbb{R}^3)}$.
By raising the above inequality to the power $p>2$ we conclude
 $$\|\mc{F}(|\cdot|\varphi g_+)\|_{\mathbb{C}\times L^p(0,\infty)}\lesssim 
 \|\varphi\|_{H^2(\mathbb{R}^3)}\|g_+\|_{H^{\frac54}(\mathbb{R}^3)}. $$
 
 \item 
In order to estimate the $L^p$-norm of the second component we note, as in (\ref{item:Lp}),
$$ |[\mc{F}(|\cdot|\varphi g_-)]_2(\xi)|\lesssim \langle \xi \rangle^{-\frac12}
\|\varphi\|_{L^1(0,\infty)}\||\cdot|g_-\|_{L^\infty(0,\infty)}
 $$
and therefore, 
$$ \|[\mc{F}(|\cdot|\varphi g_-)]_2\|_{L^p(0,\infty)}\lesssim
\|\varphi\|_{L^1(0,\infty)}\||\cdot|g_-\|_{L^\infty(0,\infty)} $$
provided that $p>2$.

\item For the $L^2$-part of $g_-$ we use Lemma \ref{lem:H2alpha} again and obtain
\begin{align*} \|\mc{F}(|\cdot|\varphi g_-)\|_{\mathbb{C} \times L^{2,\frac18}_\rho}
&\simeq \|\varphi g_-\|_{H^{\frac14}(\mathbb{R}^3)} \\
&\simeq \|\varphi g_-\|_{L^2(\mathbb{R}^3)}
+\|\varphi g_-\|_{\dot{H}^{\frac14}(\mathbb{R}^3)}.
\end{align*}
Note that
$$ \|\varphi g_-\|_{L^2(\mathbb{R}^3)}
\lesssim \|\varphi\|_{L^2(\mathbb{R}^3)}\|g_-\|_{L^\infty(0,\infty)}
\lesssim \|\varphi\|_{H^2(\mathbb{R}^3)}\|g_-\|_{L^\infty(0,\infty)}$$
and it remains to bound the $\dot{H}^{\frac14}$-norm.
For this we use the fractional Leibniz rule and obtain
\begin{align*} \|\varphi g_-\|_{\dot{H}^{\frac14}(\mathbb{R}^3)}
\lesssim &\|\varphi\|_{\dot{H}^\frac14(\mathbb{R}^3)}
\|g_-\|_{L^\infty(0,\infty)} \\ 
&+\|\varphi\|_{L^{\frac{16}{7}+}(\mathbb{R}^3)}\|g_-\|_{
\dot{W}^{\frac14,16-}(\mathbb{R}^3)}
\end{align*}
and this yields the claim since
$H^2(\mathbb{R}^3) \hookrightarrow L^{\frac{16}{7}+}(\mathbb{R}^3)$ by Sobolev embedding.
\end{enumerate}
\end{proof}

Now we are ready to establish the crucial estimate for the operator $\mc{N}_1$, 
cf.~Eq.~\eqref{eq:defR}.
\begin{lemma}
Let $p,\delta$ be as in Lemma \ref{lem:F-1} and $\beta_d, \beta_c \geq 0$.
Then we have
$$ \|\mc{N}_1 (x_d,x)\|_{L^{\infty,\beta_d+2}_{\tau_0}\times L^{\infty,\beta_c+2}_{\tau_0}Y^{p,\frac18}}
\lesssim \|(x_d,x)\|_{L^{\infty,\beta_d}_{\tau_0} \times L^{\infty,\beta_c}_{\tau_0}X^{p,\frac18}_\delta}. $$
\end{lemma}

\begin{proof}
Recall from Eq.~\eqref{eq:defR} that
$$ \mc{N}_1(x_d,x)(\tau,\xi)=\mc{F}\left (|\cdot|\varphi_1(\tau,\cdot) |\cdot|^{-1}\mc{F}^{-1}(x_d(\tau),
x(\tau,\cdot))\right )(\xi) $$
with
$$ \varphi_1(\tau,R)=5\tilde{\lambda}(\tau)^{-2}\tilde{\chi}(\tau,R)[
u_2(\nu \tilde{\lambda}(\tau)^{-1}\tau, \tilde{\lambda}(\tau)^{-1}R)^4
-u_0(\nu \tilde{\lambda}(\tau)^{-1}\tau, \tilde{\lambda}(\tau)^{-1}R)^4]. $$
Furthermore, we have $u_2=u_0+v_0+v_1$ with $v_0$ and $v_1$ from Lemmas \ref{lem:e1}
and \ref{lem:v1}, respectively.
Now note that
$$ u_2^4-u_0^4=4u_0^3 v+6u_0^2v^2+4u_0v^3+v^4 $$
where $v:=v_0+v_1$ and from Lemmas \ref{lem:e1} and \ref{lem:v1} as well as the definition of
$u_0$ we have the bounds 
\begin{align*} 
|v(\nu \tilde{\lambda}(\tau)^{-1}\tau,\tilde{\lambda}(\tau)^{-1}R)|&\lesssim 
\tilde{\lambda}(\tau)^\frac12 \tau^{-2}\langle R \rangle \\
|u_0(\nu \tilde{\lambda}(\tau)^{-1}\tau,\tilde{\lambda}(\tau)^{-1}R)|&\lesssim
\tilde{\lambda}(\tau)^{\frac12}\langle R \rangle^{-1}
\end{align*}
for $R\leq \nu\tau-\frac12 \tilde{\lambda}(\tau)c$
which imply 
$$ |\varphi_1(\tau,R)|\lesssim \tau^{-2}\langle R \rangle^{-2}. $$
We also have 
$|\partial_R^k \varphi_1(\tau,R)|\leq C_k \tau^{-2}\langle R \rangle^{-2-k}$ 
for all $k \in \mathbb{N}_0$ and this implies 
$$ \|\varphi_1(\tau,\cdot)\|_{H^2(\mathbb{R}^3)}+\|\varphi_1(\tau,\cdot)\|_{L^1(0,\infty)}
\lesssim \tau^{-2},$$ cf.~Lemma \ref{lem:FR}.

For given $(x_d,x)\in L^{\infty,\beta_d}_{\tau_0}\times L^{\infty,\beta_c}_{\tau_0}X^{p,\frac18}_\delta$ we now set
$$ y(\tau,R):=R^{-1}\mc{F}^{-1}(x_d(\tau),x(\tau,\cdot))(R), $$
i.e., we have $\mc{N}_1(x_d,x)(\tau,\xi)=\mc{F}(|\cdot|\varphi_1(\tau,\cdot) y(\tau,\cdot))(\xi)$
and obviously, our goal is to apply Lemma \ref{lem:FR}.
According to Lemma \ref{lem:F-1} we have a decomposition $y=y_-+y_+$ such that
\begin{align*} 
\|y_+(\tau,\cdot)\|_{H^\frac54(\mathbb{R}^3)}&\lesssim \|(x_d(\tau),x(\tau,\cdot))\|_{\mathbb{C}
\times X^{p,\frac18}_\delta} \\
\||\cdot|y_-(\tau,\cdot)\|_{L^\infty(0,\infty)}&\lesssim \|(x_d(\tau),x(\tau,\cdot))\|_{\mathbb{C}
\times X^{p,\frac18}_\delta} \\
\|y_-(\tau,\cdot)\|_{L^\infty(0,\infty)}&\lesssim \|(x_d(\tau),x(\tau,\cdot))\|_{\mathbb{C}
\times X^{p,\frac18}_\delta} \\
\|y_-(\tau,\cdot)\|_{\dot{W}^{\frac14,16-\epsilon}(\mathbb{R}^3)}
&\lesssim \|(x_d(\tau),x(\tau,\cdot))\|_{\mathbb{C}
\times X^{p,\frac18}_\delta}
\end{align*}
with the $\epsilon$ from Lemma \ref{lem:FR}.
Consequently, Lemma \ref{lem:FR} indeed applies and yields 
$$ \|\mc{N}_1(x_d,x)(\tau,\cdot)\|_{\C\times Y^{p,\frac18}}\lesssim \tau^{-2}
\|(x_d(\tau),x(\tau,\cdot))\|_{\mathbb{C}\times X^{p,\frac18}_\delta} $$
which implies the claim.
\end{proof}

\subsection{The operator $\mc{N}_5$}
Next, we study the operator $\mc{N}_5$. 
As before, we start with an auxiliary estimate that does not take into account the time dependence.

\begin{lemma}
\label{lem:FR5}
Let $p,\delta$ be as in Lemma \ref{lem:F-1}.
Then there exists a small $\epsilon>0$ such that
\begin{align*} \|\mc{F}\left (|\cdot|g^5\right )\|_{\mathbb{C}\times Y^{p,\frac18}}
\lesssim &\|g_-\|_{L^5(\mathbb{R}^3)}^5
+\|g_-\|_{L^\infty(\mathbb{R}^3)}^5 \\
&+\|g_-\|_{\dot{W}^{\frac14,16-\epsilon}(\mathbb{R}^3)}^5
+\|g_+\|_{H^\frac54(\mathbb{R}^3)}^5
\end{align*}
for all $g=g_-+g_+$ such that the right-hand side is finite.
\end{lemma}

\begin{proof}
We have $g^5=(g_-+g_+)^5=g_-^5+\dots+g_+^5$
and study the extreme cases $g_-^5$ and $g_+^5$ first. The intermediate terms are then estimated
by interpolation.
Let us note the estimate 
$$ \left \|\prod_{j=1}^5 f_j \right \|_{H^\frac14(\mathbb{R}^3)}
\lesssim \sum_{j=1}^5 \|f_j\|_{W^{\frac14,6}(\mathbb{R}^3)}\prod_{\ell\not= j}
\|f_\ell\|_{L^{12}(\mathbb{R}^3)} $$
which is a consequence of the fractional Leibniz rule (\cite{T00}, p.~105, Proposition 1.1).
By the Sobolev embeddings
$H^\frac54(\mathbb{R}^3)\hookrightarrow L^{12}(\mathbb{R}^3)$ and
$H^\frac54(\mathbb{R}^3)\hookrightarrow W^{\frac14,6}(\mathbb{R}^3)$ we infer
$$ \left \|\prod_{j=1}^5 f_j \right \|_{H^\frac14(\mathbb{R}^3)}
\lesssim \prod_{j=1}^5 \|f_j\|_{H^{\frac54}(\mathbb{R}^3)} $$
which will be useful in the following.

\begin{enumerate}
\item According to Lemma \ref{lem:H2alpha} we have
$$
\|\mc{F}(|\cdot|g_+^5)\|_{\mathbb{C}\times L^{2,\frac18}_\rho}\simeq 
\|g_+^5\|_{H^\frac14(\mathbb{R}^3)}\lesssim 
\|g_+\|_{H^\frac54(\mathbb{R}^3)}^5. $$

\item In order to estimate the $L^p$-part of $g_+$ recall that $|\phi(R,\xi)|\lesssim 
\langle \xi\rangle^{-\frac12}$ for all $R\geq 0$ and thus,
\begin{align*}
|[\mc{F}(|\cdot|g_+^5)]_2(\xi)|&\lesssim \|\phi(\cdot, \xi)\|_{L^2(0,1)}
\||\cdot|g_+^5\|_{L^2(0,1)} \\
&\quad +\||\cdot|^{-1}\phi(\cdot,\xi)\|_{L^\infty(1,\infty)}
\||\cdot|^2 g_+^5\|_{L^1(1,\infty)} \\
&\lesssim \langle \xi \rangle^{-\frac12}\left (\|g_+\|_{L^{10}(\mathbb{R}^3)}^5
+\|g_+\|_{L^5(\mathbb{R}^3)}^5 \right ).
\end{align*}
Consequently, the Sobolev embeddings $H^\frac54(\mathbb{R}^3)\hookrightarrow L^{10}(\mathbb{R}^3)$
and
$H^\frac54(\mathbb{R}^3)\hookrightarrow L^5(\mathbb{R}^3)$ yield
$$ \|[\mc{F}(|\cdot|g_+^5)]_2\|_{L^p(0,\infty)}\lesssim 
\|g_+\|_{H^\frac54(\mathbb{R}^3)}^5 $$
for $p>2$ as desired.

\item 
For the $L^p$-part of $g_-$ note that
\begin{align*} |[\mc{F}(|\cdot|g_-^5)]_2(\xi)|&\lesssim \langle \xi\rangle^{-\frac12}
\left [\|g_-^5\|_{L^\infty(0,1)}+
\||\cdot|^2g_-^5\|_{L^1(1,\infty)} \right ] \\
&\lesssim \langle \xi \rangle^{-\frac12}\left [ \|g_-\|_{L^\infty(\mathbb{R}^3)}^5
+\|g_-\|_{L^{5}(\mathbb{R}^3)}^5 \right ]
\end{align*}
which yields the desired bound
$$ \|[\mc{F}(|\cdot|g_-^5)]_2\|_{L^p(0,\infty)}\lesssim 
 \|g_-\|_{L^\infty(\mathbb{R}^3)}^5
+\|g_-\|_{L^5(\mathbb{R}^3)}^5 $$
provided that $p>2$.

\item
According to Lemma \ref{lem:F-1} we have
$$ \|\mc{F}(|\cdot|g_-^5)\|_{\mathbb{C}\times L^{2,\frac18}_\rho}
\simeq \|g_-^5\|_{H^\frac14(\mathbb{R}^3)} $$
and since
$\|g_-^5\|_{L^2(\mathbb{R}^3)}=\|g_-\|_{L^{10}(\mathbb{R}^3)}^5 \lesssim
\|g_-\|_{L^5(\mathbb{R}^3)}^5+\|g_-\|_{L^\infty(\mathbb{R}^3)}^5$
it suffices to control the homogeneous Sobolev norm 
$\|g_-^5\|_{\dot{H}^\frac14(\mathbb{R}^3)}$.
For this we use the fractional Leibniz rule to conclude
\begin{align*} \|g_-^5\|_{\dot{H}^\frac14(\mathbb{R}^3)}&\lesssim \|g_-\|_{L^{\frac{64}{7}+}(\mathbb{R}^3)}^4
\|g_-\|_{\dot{W}^{\frac14,16-}(\mathbb{R}^3)} \\
&\lesssim \left [ \|g_-\|_{L^5(\mathbb{R}^3)}^4+\|g_-\|_{L^\infty(\mathbb{R}^3)}^4 \right ]
\|g_-\|_{\dot{W}^{\frac14,16-}(\mathbb{R}^3)}.
\end{align*}
\end{enumerate}
We briefly comment on how to estimate the mixed terms.
First of all, it is trivial to bound the $L^p$ norms since
$$ |[\mc{F}(|\cdot|g_-^{5-k}g_+^k)]_2(\xi)|\lesssim \int_0^\infty |\phi(R,\xi)|R \left [|g_-(R)|^5
+|g_+(R)|^5 \right ]dR $$
for all $k \in \{1,2,3,4\}$ by the pointwise inequality $|g_-^{5-k}g_+^k|\lesssim |g_-|^5+|g_+|^5$
which brings us back to the two extreme cases considered above.
For the $L^2$-parts the only nontrivial contributions come from the $\dot{H}^{\frac14}(\mathbb{R}^3)$
homogeneous Sobolev norms.
In order to control these, one proceeds as before by applying 
the fractional Leibniz rule followed by Sobolev embedding, i.e.,
\begin{align*}
\|g_-^{5-k} g_+^k\|_{\dot{H}^\frac14(\mathbb{R}^3)}&\lesssim 
\|g_-\|_{L^\infty(\mathbb{R}^3)}^{5-k}\|g_+\|_{L^{3(k-1)}(\mathbb{R}^3)}^{k-1} 
\|g_+\|_{\dot{W}^{\frac14,6}(\mathbb{R}^3)} \\
&\quad +\|g_+\|_{L^{k\frac{16}{7}+}(\mathbb{R}^3)}^k
\|g_-\|_{L^\infty(\mathbb{R}^3)}^{4-k}\|g_-\|_{\dot{W}^{\frac14,16-}(\mathbb{R}^3)} \\
&\lesssim \|g_-\|_{L^\infty(\mathbb{R}^3)}^{5-k}\|g_+\|_{H^\frac54(\mathbb{R}^3)}^k
+\|g_-\|_{\dot{W}^{\frac14,16-}(\mathbb{R}^3)}^{5-k}\|g_+\|_{H^\frac54(\mathbb{R}^3)}^k.
\end{align*}
\end{proof}

\begin{lemma}
\label{lem:R5}
Let $p,\delta$ be as in Lemma \ref{lem:F-1} and $\beta_d, \beta_c \geq 0$.
Then we have the estimate
$$ \|\mc{N}_5(x_d,x)\|_{L^{\infty,5\beta_d-\frac14}_{\tau_0}\times L^{\infty,5\beta_c-\frac14}_{\tau_0}Y^{p,\frac18}}
\lesssim \|(x_d,x)\|_{L^{\infty,\beta_d}_{\tau_0}\times L^{\infty,\beta_c}_{\tau_0}X^{p,\frac18}_\delta}^5. $$
\end{lemma}

\begin{proof}
Recall that
$$ \mc{N}_5(x_d,x)(\tau,\xi)=\mc{F}\left (|\cdot|\varphi_5(\tau,\cdot)\left [|\cdot|^{-1}\mc{F}^{-1}
(x_d(\tau),x(\tau,\cdot))\right ]^5\right )
(\xi) $$
and by setting $y(\tau,R):=R^{-1}\mc{F}(x_d(\tau),x(\tau,\cdot))(\xi)$ we obtain from
Lemma \ref{lem:F-1} the existence of a decomposition $y=y_-+y_+$ with the bound
$$ \|y_-(\tau,\cdot)\|_{L^5(\mathbb{R}^3)}+\|y_-(\tau,\cdot)\|_{L^\infty(\mathbb{R}^3)}
+\|y_-(\tau,\cdot)\|_{\dot{W}^{\frac14,16-\epsilon}(\mathbb{R}^3)}\lesssim 
\|(x_d(\tau),x(\tau,\cdot))\|_{\mathbb{C}\times X^{p,\frac18}_\delta} $$
as well as
$$ \|y_+(\tau,\cdot)\|_{H^\frac54(\mathbb{R}^3)}\lesssim 
\|(x_d(\tau),x(\tau,\cdot))\|_{\mathbb{C}\times X^{p,\frac18}_\delta} $$
for any (small) $\epsilon>0$.
Thus, since $|\partial_R^k \varphi_5(\tau,R)|\leq C_k \tilde{\lambda}(\tau)^{-2-k}\lesssim \tau^\frac14$
for bounded $k$,
Lemma \ref{lem:FR5} yields
$$ \|\mc{N}_5(x_d(\tau),x(\tau,\cdot))\|_{\mathbb{C}\times Y^{p,\frac18}}
\lesssim \tau^\frac14\|(x_d(\tau),x(\tau,\cdot))\|_{\mathbb{C}\times X^{p,\frac18}_\delta}^5 $$
which implies the claim.
\end{proof}

\begin{remark}
The loss of $\tau^\frac14$ in Lemma \ref{lem:R5} can be improved to $\tau^\epsilon$ with 
$\epsilon\to 0+$
as $\nu \to 1$ but the stated result suffices for our purposes and it avoids the introduction
of an additional $\epsilon$.
\end{remark}

The remaining operators $\mc{N}_2$, $\mc{N}_3$, $\mc{N}_4$ are in a certain sense 
interpolates between $\mc{N}_1$ and $\mc{N}_5$ and can be treated in the exact same fashion.
Note that we do not gain additional decay in $\tau$ from the $\varphi_j$ factors because they involve
the (rescaled) soliton $u_0$ which does not decay.
In fact, from $\varphi_j$ we may even have a loss of $\tau^\epsilon$ (with $\epsilon \to 0+$ as $\nu \to 1$) depending
on the sign of $1-\nu$.
Consequently, the gain comes exclusively from the nonlinearity, as is the case for
the operator $\mc{N}_5$.
We only state the corresponding result and leave the verification to the reader.

\begin{lemma}
\label{lem:R234}
Let $p,\delta$ be as in Lemma \ref{lem:F-1} and $\beta_d, \beta_c \geq 0$.
Then we have the estimate
$$ \|\mc{N}_j (x_d,x)\|_{L^{\infty,j\beta_d-\frac14}_{\tau_0}\times L^{\infty,j\beta_c-\frac14}_{\tau_0}Y^{p,\frac18}}
\lesssim \|(x_d,x)\|_{L^{\infty,\beta_d}_{\tau_0}\times L^{\infty,\beta_c}_{\tau_0}X^{p,\frac18}_\delta}^j $$
for $j \in \{2,3,4\}$.
\end{lemma}

Furthermore, by basically repeating the above computations and using elementary identities such as
$$ a^n-b^n=(a-b)\sum_{j=0}^{n-1}a^jb^{n-1-j},$$
it is straightforward to see that we have the following estimate
for the operator $\mc{N}:=\sum_{j=1}^5 \mc{N}_j$.

\begin{lemma}
\label{lem:R}
Let $p,\delta$ be as in Lemma \ref{lem:F-1} and $\beta_d, \beta_c \geq \frac32$. Then there exists a continuous
function $\gamma: [0,\infty)\times [0,\infty) \to \mathbb{R}$ such that 
\begin{align*} 
\|\mc{N}(x_d,x)-\mc{N}(y_d,y)\|_{L^{\infty,\beta_d+\frac54}_{\tau_0}\times L^{\infty,\beta_c+\frac54}_{\tau_0}Y^{p,\frac18}}
\leq \gamma(X,Y)\|(x_d,x)-(y_d,y)\|_{L^{\infty,\beta_d}_{\tau_0}\times L^{\infty,\beta_c}_{\tau_0}X^{p,\frac18}_\delta}
\end{align*}
where
\begin{align*} 
X:=\|(x_d,x)\|_{L^{\infty,\beta_d}_{\tau_0}\times L^{\infty,\beta_c}_{\tau_0}X^{p,\frac18}_\delta},\quad 
Y:=\|(y_d,y)\|_{L^{\infty,\beta_d}_{\tau_0}\times L^{\infty,\beta_c}_{\tau_0}X^{p,\frac18}_\delta}.
\end{align*}
\end{lemma}

\section{Estimates for the terms involving $\mc{K}$ and $[\mc{A},\mc{K}]$}

\label{sec:K}
Finally, we consider the terms on the right-hand side of Eq.~\eqref{eq:sysFourier} 
which come from the transference identity. 
The main results in this respect are summarized in the following proposition.

\begin{proposition}[Estimates for the $\mc K$-operators]
\label{prop:mainK}
Let $\delta>0$ be small and assume $p \in (1,\infty)$ to be so large that $p'(1-\delta)<1$.
Then we have the estimates
\begin{align*}
\|\mc K (a,f)\|_{\C \times X^{p,\frac18}_\delta}&\lesssim \|(a,f)\|_{\C \times X^{p,\frac18}_\delta} &
\|[\mc A,\mc K] (a,f)\|_{\C \times X^{p,\frac18}_\delta}&\lesssim \|(a,f)\|_{\C \times X^{p,\frac18}_\delta} \\
\|\mc K (a,f)\|_{\C \times Y^{p,\frac18}}&\lesssim \|(a,f)\|_{\C \times X^{p,\frac18}_\delta} &
\|[\mc A, \mc K] (a,f)\|_{\C \times Y^{p,\frac18}}&\lesssim \|(a,f)\|_{\C \times X^{p,\frac18}_\delta} \\
\|\mc K (a,g)\|_{\C \times Y^{p,\frac18}}&\lesssim \|(a,g)\|_{\C \times Y^{p,\frac18}} &
\|[\mc A, \mc K] (a,g)\|_{\C \times Y^{p,\frac18}}&\lesssim \|(a,g)\|_{\C \times Y^{p,\frac18}}
\end{align*}
for all $a \in \C$, $f \in X^{p,\frac18}_\delta$, and $g\in Y^{p,\frac18}$.
\end{proposition}

Based on these bounds and the nonlinear estimates from Section \ref{sec:nonlinhom} we then prove 
the existence of a solution to Eq.~\eqref{eq:sysFourier} by a fixed point argument.

\subsection{A convenient expression for $\mc K$}
First, we need to obtain an explicit expression for the operator $\mc{K}$.
Definition \eqref{eq:K} can be rewritten as
$$ (\mc{K}+1)\left ( \begin{array}{c}a \\ f\end{array} \right )=
\mc{F}\left (|\cdot|\left [ \mc{F}^{-1}\left ( \begin{array}{c}a \\ f\end{array} \right )\right]'\right )
-\mc{A}\left ( \begin{array}{c} a \\ f \end{array} \right )$$
or, equivalently,
$$ \mc{F}^{-1}(\mc{K}+1)\left ( \begin{array}{c}a \\ f\end{array} \right )
=|\cdot|\left [\mc{F}^{-1}\left ( \begin{array}{c}a \\ f\end{array} \right )\right]'
-\mc{F}^{-1}\mc{A}\left ( \begin{array}{c}a \\ f\end{array} \right ). $$ 
By definition of $\mc{A}$ and Eq.~\eqref{eq:defF-1} we have
\begin{align*}
\mc{F}^{-1}\mc{A}\left ( \begin{array}{c}a \\ f\end{array} \right )&=
\mc{F}^{-1}\left ( \begin{array}{c}0 \\ -2|\cdot|f'-\tfrac52f-\tfrac{|\cdot|\rho'}{\rho}f \end{array}\right ) \\
&=-2\int_0^\infty \phi(\cdot,\eta)\eta f'(\eta)\rho(\eta)d\eta- \int_0^\infty
\phi(\cdot,\eta)\left [\tfrac52 +\tfrac{\eta \rho'(\eta)}{\rho(\eta)}\right ]f(\eta)\rho(\eta)d\eta
\end{align*}
and, assuming that $f \in C_c^\infty(0,\infty)$, an integration by parts yields
\begin{align*}
\mc{F}^{-1}\mc{A}\left ( \begin{array}{c}a \\ f\end{array} \right )&=
2\int_0^\infty \eta \partial_\eta \phi(\cdot,\eta)f(\eta)\rho(\eta)d\eta
+\int_0^\infty \phi(\cdot,\eta)\left [-\tfrac12+\tfrac{\eta \rho'(\eta)}{\rho(\eta)} \right ]f(\eta)
\rho(\eta)d\eta \\
&=2\int_0^\infty \eta \partial_\eta \phi(\cdot,\eta)f(\eta)\rho(\eta)d\eta+\mc{F}^{-1}
\left ( \begin{array}{c}0 \\ (-\tfrac12+\tfrac{|\cdot|\rho'}{\rho})f\end{array} \right ).
\end{align*}
Consequently, we obtain
\begin{align*}
\mc{F}^{-1}\mc{K}\left ( \begin{array}{c}a \\ f\end{array} \right )(R)
=&a\frac{(R\partial_R-1) \phi(R,\xi_d)}{\|\phi(\cdot,\xi_d)\|_{L^2(0,\infty)}^2} 
+\int_0^\infty [R\partial_R \phi(R,\eta)-2\eta \partial_\eta\phi(R,\eta)]f(\eta)\rho(\eta)d\eta  \\
&-\mc{F}^{-1}
\left ( \begin{array}{c}0 \\ (\tfrac12+\tfrac{|\cdot|\rho'}{\rho})f\end{array} \right )(R)
\end{align*}
and thus,
\begin{align*}
\mc{K}_{dd} a&=\frac{a}{\|\phi(\cdot,\xi_d)\|_{L^2(0,\infty)}^2}\int_0^\infty\phi(R,\xi_d)
[R\partial_R-1] \phi(R,\xi_d)dR=-\tfrac32 a \\
\mc{K}_{cd}a(\xi)&=\frac{a}{\|\phi(\cdot,\xi_d)\|_{L^2(0,\infty)}^2} \int_0^\infty
\phi(R,\xi)[R\partial_R-1] \phi(R,\xi_d)dR \\
\mc{K}_{dc}f&=\int_0^\infty \int_0^\infty \phi(R,\xi_d)[R\partial_R\phi(R,\eta)-
2\eta \partial_\eta \phi(R,\eta)]f(\eta)\rho(\eta)d\eta dR \\
\mc{K}_{cc}f(\xi)&=\int_0^\infty \int_0^\infty \phi(R,\xi)[R\partial_R\phi(R,\eta)-2\eta 
\partial_\eta \phi(R,\eta)]f(\eta)\rho(\eta)d\eta dR \\
&\quad -\left (\tfrac12+\tfrac{\xi \rho'(\xi)}{\rho(\xi)}\right )f(\xi)
\end{align*} 
where, as before,
$$ \mc{K}=\left (\begin{array}{cc}\mc{K}_{dd} & \mc{K}_{dc} \\
\mc{K}_{cd} & \mc{K}_{cc} \end{array} \right ). $$

\subsection{$\mc{K}_{cc}$ as a Calder\'on-Zygmund operator}
First, we focus on $\mc{K}_{cc}$ which is the most complicated of the above operators.
The respective estimates for $\mc{K}_{cd}$ and $\mc{K}_{dc}$ will follow easily after we have 
understood $\mc{K}_{cc}$ and $\mc{K}_{dd}$ is trivial anyway.
In order to proceed, we need a more manageable representation of $\mc{K}_{cc}$, more precisely
of the integral part of $\mc{K}_{cc}$ which we denote by $\tilde{\mc{K}}_0$ for the moment, i.e.,
$$ \tilde{\mc{K}}_0 f(\xi):=\int_0^\infty \int_0^\infty \phi(R,\xi)[R\partial_R\phi(R,\eta)-2\eta 
\partial_\eta \phi(R,\eta)]f(\eta)\rho(\eta)d\eta dR. $$
Note first that, for $f\in C_c^\infty(0,\infty)$, the function
$$ R\mapsto \int_0^\infty [R\partial_R \phi(R,\eta)-2\eta \partial_\eta \phi(R,\eta)]f(\eta)
\rho(\eta)d\eta $$
is rapidly decreasing as $R \to \infty$. This can be immediately concluded from the 
representation of $\phi$ given in Corollary \ref{cor:phi} and integration by parts. 
It follows that the operator $\tilde{\mc{K}}_0$ is well-defined as a linear mapping from $C^\infty_c(0,\infty)$
to $C[0,\infty)$. Furthermore, by
dominated convergence, $\tilde{\mc{K}}_0$ is continuous when viewed as a map into the space of distributions. 
Consequently, by the Schwartz kernel theorem there exists a (distributional) kernel 
$\tilde{K}_0$ such that
\begin{equation}
\label{eq:schwartzkernel}
\tilde{\mc{K}}_0f(\xi)=\int_0^\infty \tilde{K}_0(\xi,\eta)f(\eta)d\eta. 
\end{equation}
In fact, the operator $\mc{K}_{cc}$ has already been studied in \cite{KST09} and from there 
we have the following result (note carefully that our $\mc{K}_{cc}$ is called $\mc{K}_0$ in \cite{KST09}).

\begin{theorem}
\label{thm:K0}
For $f \in C^\infty_c(0,\infty)$ the operator $\mc{K}_{cc}$ is given by
$$ \mc{K}_{cc}f(\xi)=
\int_0^\infty K_0(\xi,\eta)f(\eta)d\eta $$
where the kernel $K_0$ is of the form
$$ K_0(\xi,\eta)=\frac{\rho(\eta)}{\xi-\eta}F(\xi,\eta) $$
with a symmetric function $F \in C^2((0,\infty)\times (0,\infty))$.
Furthermore, for any $N\in \mathbb{N}$, $F$ satisfies the bounds
  \[\begin{split}
   | F(\xi,\eta)| &\leq C_N \left\{ \begin{array}{cc} \xi+\eta &
       \xi+\eta \leq 1 \cr (\xi+\eta)^{-1} (1+|\xi^\frac12
       -\eta^\frac12|)^{-N} & \xi+\eta \geq 1
     \end{array} \right.\\
   | \partial_{\xi} F(\xi,\eta)|+| \partial_{\eta} F(\xi,\eta)| &\leq C_N \left\{
     \begin{array}{cc} 1 & \xi+\eta \leq 1 \cr (\xi+\eta)^{-\frac32}
       (1+|\xi^\frac12 -\eta^\frac12|)^{-N} & \xi+\eta \geq 1
     \end{array} \right.\\
  \max_{j+k=2} | \partial^j_{\xi}\partial^k_{\eta} F(\xi,\eta)| &\leq C_N \left\{
     \begin{array}{cc} (\xi+\eta)^{-\frac12} & \xi+\eta \leq 1 \cr
       (\xi+\eta)^{-2} (1+|\xi^\frac12 -\eta^\frac12|)^{-N} &
       \xi+\eta \geq 1
     \end{array} \right. .
     \end{split}
   \]
\end{theorem}

\begin{proof}
This is Theorem 5.1 in \cite{KST09}.
One starts with an integration by parts and 
thereby identifies the function $F$ which can be expressed as an integral over
$\phi$, $\rho$ and some explicitly known function resulting from the potential $V$.
The stated estimates are then obtained by a careful analysis of this expression based on the
asymptotic descriptions of $\phi$ and $\rho$ from Section \ref{sec:exact}.
We refer the reader to \cite{KST09} for the details.
\end{proof}

\subsection{Bounds for $\mc{K}_{cc}$}
In order to obtain estimates in the $\X$ and $\Y$ spaces, we require boundedness of $\mc K_{cc}$ in 
\emph{weighted} $L^p$ spaces.
However, this problem reduces to boundedness on ordinary $L^p$ by simply
attaching the weights to the kernel.
The (weighted) $L^2$ boundedness of $\mc{K}_0$ is already established in \cite{KST09}, Proposition 5.2.
Unfortunately, the result there does not exactly apply to our situation (at least not for small
frequencies) since it only considers
weights of the form $\langle \cdot \rangle^{2\alpha}$ but we have
to deal with more general expressions, cf.~
Definition \ref{def:spaces}.
Thus, we have to make sure that despite this slight modification the reasoning in \cite{KST09} still goes through.
Furthermore, we have to take care of the $L^p$ component in $\X$ which is not present
in \cite{KST09}.
Finally, we need to exploit a certain additional smoothing property of $\mc K_{cc}$ 
at small frequencies
which was irrelevant
for the construction in \cite{KST09}.

In order to prove boundedness in weighted spaces of the type occurring in the definition of
$X^{p,\alpha}_\delta$,
we define kernels $K^{(a,b)}_0$ for $a,b \in \mathbb{R}$ by
\begin{equation}
\label{eq:defKab}
K^{(a,b)}_0(\xi, \eta):=\xi^{-\frac12}\langle \xi \rangle^\frac12
\xi^{-a}\langle \xi \rangle^{-b}K_0(\xi,\eta)\eta^a \langle \eta\rangle^b
\end{equation}
and denote by $\mc{K}_0^{(a,b)}$ the corresponding operators.
Observe carefully the additional weight $\xi^{-\frac12}\langle \xi\rangle^\frac12$ on
the ``output'' variable which encodes the aforementioned smoothing effect.
Our aim is to prove $L^p$ boundedness of $\mc{K}_0^{(a,b)}$ for any $a,b$ and $1<p<\infty$ which then implies
the desired boundedness properties of $\mc{K}_{cc}$ in $X^{p,\alpha}_\delta$ and $\Y$.

Due to the singular behavior of the spectral measure at zero it is advantageous to 
separate the diagonal from the off-diagonal behavior.
This is most effectively done by introducing a dyadic covering of the diagonal $\Delta=\{(\xi,\eta)\in\R^2: \xi=\eta\}$,
\[ \Delta \subset \bigcup_{j \in \mathbb{Z}} I_j \times I_j, \]
where $I_j:=[2^{j-1},2^{j+1}]$.
Furthermore, let $\chi: \mathbb{R} \to [0,1]$ be a smooth bump function satisfying
\[ \chi(\xi):=\left \{ \begin{array}{l}
1 \mbox{ if } \xi \in [\frac34, \frac74] \\
0 \mbox{ if } \xi \leq \frac12 \mbox{ or } \xi \geq 2 \end{array} \right .
\]
and set $\chi_j(\xi):=\chi(2^{-j}\xi)$.
Then we have $\mathrm{supp}(\chi_j)\subset I_j$ and the sets 
$\{(\xi,\eta) \in \mathbb{R}^2: \chi_j(\xi)\chi_j(\eta)=1\}$ still cover the diagonal.
We smoothly restrict the kernel $K_0^{(a,b)}$ to 
$I_j \times I_j$ by using $\chi_j$ and write
$$ \mc{K}_{0,j}^{(a,b)}f:=\chi_j\mc{K}_0^{(a,b)}(\chi_j f) $$
for the corresponding truncated operator.

\begin{lemma}
\label{lem:Kjabsm}
Fix $(a,b) \in \mathbb{R}^2$ and let $1<p<\infty$. Then $\mc{K}_{0,j}^{(a,b)}$ extends to a bounded operator
on $L^p(\R)$ and 
$$ \|\mc{K}_{0,j}^{(a,b)}f\|_{L^p(\R)}\lesssim \|f\|_{L^p(\R)}
$$
for all $f \in L^p(\R)$ and $j \in \mathbb{Z}$, $j\leq 2$.
\end{lemma}

\begin{proof}
The point of the dyadic decomposition is of course that $\xi,\eta \in I_j$ implies 
$\eta \leq 4\xi$. In other words, for any two elements $\xi,\eta \in I_j$
we have $\xi \simeq \eta \simeq 2^j$ \emph{with implicit constants independent of $j$}!
This allows us to control the singular behavior of $\rho(\eta)$ uniformly in $j$.
We write $F(\xi,\eta)=F(\xi,\xi)+(\xi-\eta)\tilde{F}(\xi,\eta)$
where
$$ \tilde{F}(\xi,\eta)=-\int_0^1 \partial_2 F(\xi,\xi+s(\eta-\xi))ds $$
and by Theorem \ref{thm:K0} we have the bound $|\tilde{F}(\xi,\eta)|\lesssim 1$ for all $\xi,\eta$.
Thus, the kernel decomposes as $K^{(a,b)}_{j,0}(\xi,\eta)=A_j(\xi,\eta)+B_j(\xi,\eta)$
where
\begin{align*}
A_j(\xi,\eta)&=\chi_j(\xi)\xi^{-\frac12}\langle \xi \rangle^\frac12 \xi^{-a}\langle \xi \rangle^{-b}
\frac{F(\xi,\xi)}{\xi-\eta}\chi_j(\eta)\rho(\eta)\eta^a \langle \eta \rangle^b
=:\frac{\psi_{0,j}(\xi)\psi_{1,j}(\eta)}{\xi-\eta} \\
B_j(\xi,\eta)&=\chi_j(\xi)\xi^{-\frac12}\langle \xi \rangle^\frac12 \xi^{-a}\langle \xi \rangle^{-b}
\tilde{F}(\xi,\eta)\chi_j(\eta)\rho(\eta)\eta^a \langle \eta \rangle^b
\end{align*}
and we call the respective operators $\mc{A}_j$ and $\mc{B}_j$.
In other words, we have $\mc A_j f=\pi \psi_{0,j}H(\psi_{1,j}f)$ where $H$ is the Hilbert transform. 
By the $L^p$ boundedness of $H$ for $p \in (1,\infty)$ we immediately obtain 
\[ \|\mc A_j f\|_{L^p(\R)}\lesssim \|\psi_{0,j}\|_{L^\infty(\R)}\|\psi_{1,j}f\|_{L^p(\R)}\leq 
\|\psi_{0,j}\|_{L^\infty(\R)}\|\psi_{1,j}\|_{L^\infty(\R)}\|f\|_{L^p(\R)} \]
and from Theorem \ref{thm:K0} and Lemma \ref{lem:smxi} we infer the bounds
\[ |\psi_{0,j}(\xi)|\lesssim 2^{(-\frac12-a+1)j}, \quad |\psi_{1,j}(\eta)|\lesssim 2^{(-\frac12+a)j} \]
for all $\xi,\eta \in \R$ which yield $\|\mc A_j f\|_{L^p(\R)}\lesssim \|f\|_{L^p(\R)}$. 
Furthermore, the kernel $B_j$ satisfies
\[ |B_j(\xi,\eta)|\lesssim 2^{-j}\chi_j(\xi)\chi_j(\eta) \]
and thus, 
\[ \|\mc B_j f\|_{L^p(\R)}\lesssim 2^{-j}\|\chi_j\|_{L^p(\R)}\|\chi_j\|_{L^{p'}(\R)}\|f\|_{L^p(\R)}
\lesssim 2^{(-1+\frac{1}{p}+\frac{1}{p'})j}\|f\|_{L^p(\R)}=\|f\|_{L^p(\R)}.\]
\end{proof}

An analogous result holds for $j\geq 1$ as well.

\begin{lemma}
\label{lem:Kjablg}
Fix $(a,b) \in \mathbb{R}^2$ and let $1<p<\infty$. Then $\mc{K}_{0,j}^{(a,b)}$ extends to a bounded operator
on $L^p(\R)$ and 
$$ \|\mc{K}_{0,j}^{(a,b)}f\|_{L^p(\R)}\lesssim \|f\|_{L^p(\R)}
$$
for all $f \in L^p(\R)$ and $j \in \mathbb{Z}$, $j\geq 1$.
\end{lemma}

\begin{proof}
We perform the same decomposition $\mc K_{0,j}^{(a,b)}=\mc A_j+\mc B_j$ as in the proof of Lemma 
\ref{lem:Kjabsm} and this time we have
\[ |\psi_{0,j}(\xi)|\lesssim 2^{-(a+b+1)j},\quad |\psi_{1,j}(\eta)|\lesssim 2^{(\frac12+a+b)j} \]
for all $\xi,\eta \in \R$ by Theorem \ref{thm:K0} and Lemma \ref{lem:smxi}.
Consequently, as before, the $L^p$ boundedness of the Hilbert transform yields 
$\|\mc A_j f\|_{L^p(\R)}\lesssim \|f\|_{L^p(\R)}$ for $p\in (1,\infty)$.
For the operator $\mc B_j$ we note that $|\tilde{F}(\xi,\eta)|\lesssim |\xi+\eta|^{-\frac32}$ 
by Theorem \ref{thm:K0} and thus, 
$|B_j(\xi,\eta)|\lesssim 2^{-j}\chi_j(\xi)\chi_j(\eta)$ which yields
$\|\mc B_j f\|_{L^p(\R)}\lesssim \|f\|_{L^p(\R)}$. 
\end{proof}

The bounds obtained in Lemmas \ref{lem:Kjabsm} and \ref{lem:Kjablg} can be summed.

\begin{corollary}
\label{cor:Kjabdiag}
Fix $(a,b) \in \R^2$ and $1<p<\infty$. Then the operator 
$$ \mc{K}_{0,\Delta}^{(a,b)}:=\sum_{j\in\mathbb{Z}}\mc{K}_{0,j}^{(a,b)} $$
is bounded on $L^p(\R)$.
\end{corollary}

\begin{proof}
It suffices to note that
\begin{align*}
\|\mc{K}_{0,\Delta}^{(a,b)}f\|_{L^p(\R)}^p \lesssim 
\sum_{j\in\mathbb{Z}}\|\mc{K}_{0,j}^{(a,b)}(1_{I_j}f)\|_{L^p(\R)}^p
\lesssim \sum_{j\in \mathbb{Z}}\|1_{I_j}f\|_{L^p(\R)}^p\lesssim \|f\|_{L^p(\R)}^p
\end{align*}
by Lemmas \ref{lem:Kjabsm} and \ref{lem:Kjablg} where $1_{I_j}$ denotes the characteristic 
function of the set $I_j$.
\end{proof}

Now we can conclude the desired boundedness properties of $\mc{K}_{cc}$.

\begin{proposition}
\label{prop:KccX}
Fix $\alpha\geq 0$ and $\delta> 0$ small. Furthermore, 
let $1<p<\infty$ be so large that $p'(1-\delta)<1$.
Then we have the bounds
\begin{align*} 
\|\mc{K}_{cc}f\|_{\X}&\lesssim \|f\|_\X \\
\|\mc{K}_{cc}f\|_{\Y}&\lesssim \|f\|_\X \\
\|\mc{K}_{cc}f\|_{\Y}&\lesssim \|f\|_\Y.
\end{align*}
\end{proposition}

\begin{proof}
We write $\mc{K}_\Delta f:=\sum_{j\in \mathbb{Z}}\chi_j \mc{K}_{cc}(\chi_j f)$ for the diagonal part of $\mc{K}_{cc}$.
By Corollary \ref{cor:Kjabdiag} it is evident that
$\|\mc K_\Delta f\|_\X \lesssim \|f\|_\X$ and $\|\mc K_\Delta f\|_\Y \lesssim \|f\|_\Y$.
In order to obtain the mixed estimate we exploit the smoothing property, i.e., we note that
\[ \|\mc K_\Delta f\|_{L^p(0,\infty)}\lesssim \||\cdot|^\frac12 \langle \cdot \rangle^{-\frac12}f
\|_{L^p(0,\infty)}\lesssim \||\cdot|^{\frac12-\delta}\langle \cdot \rangle^{-\frac12+\delta}f\|_{L^p(0,\infty)}
\lesssim \|f\|_\X \]
as well as
\[ \|\mc K_\Delta (f)\langle \cdot \rangle^\alpha \rho^\frac12 \|_{L^2(0,\infty)}
\lesssim \|f|\cdot|^\frac12 \langle \cdot \rangle^{\alpha-\frac12} \rho^\frac12 \|_{L^2(0,\infty)}
\lesssim \|f\|_\X. \]

It remains to study the off-diagonal contributions and to this end we set
\[ \varphi_A(\xi,\eta):=1_{A}(\xi,\eta)\left [1-\sum_{j\in\mathbb{Z}}\chi_j(\xi)
\chi_j(\eta) \right ] \]
for $A \subset [0,\infty)^2$.
We distinguish between large $\xi,\eta$ and small $\xi,\eta$ and consider the two truncated kernels 
\[ K_-:=\varphi_{[0,1]^2}K_0,\quad K_+:=\varphi_{[0,\infty)^2\backslash [0,1]^2}K_0 \]
and denote by $\mc{K}_-$, $\mc{K}_+$ the respective operators.
Note that $(\xi,\eta)\in \mathrm{supp}(\varphi_{[0,1]^2})$ implies that $\eta \leq c \xi$ or
$\eta \geq \frac{1}{c}\xi$ for a suitable $c \in (0,1)$.
This implies $|\xi-\eta|\gtrsim \xi+\eta$
and from Theorem \ref{thm:K0} and Lemma \ref{lem:smxi} we obtain the estimate 
$|K_-(\xi,\eta)|\lesssim \varphi_{[0,1]^2}(\xi,\eta)\eta^{-\frac12}$.
We infer
\begin{align*}
|\mc{K}_-f(\xi)|&\lesssim 1_{[0,1]}(\xi)\int_0^1 \eta^{-\frac12}|f(\eta)|d\eta
=1_{[0,1]}(\xi)\int_0^1 \eta^{-1+\delta}|\eta^{\frac12-\delta}f(\eta)|d\eta \\
&\lesssim 1_{[0,1]}(\xi)\||\cdot|^{\frac12-\delta}f\|_{L^p(0,1)}
\end{align*}
by H\"older's inequality and the condition $p'(1-\delta)<1$.
Note that this estimate implies $|\mc K_- f(\xi)|\lesssim 1_{[0,1]}(\xi)\|f\|_\X$ and also
$|\mc K_- f(\xi)|\lesssim 1_{[0,1]}(\xi)\|f\|_\Y$ for all $\xi\geq 0$.
Thus, we immediately obtain the bound
$\|\mc K_- f\|_\X \lesssim \|\mc K_- f\|_{L^\infty(0,1)}\lesssim \|f\|_\X$.
For the remaining two estimates it is crucial that the spectral measure $\rho$ be integrable
near $0$, i.e., we have
\[ \|\mc K_- f\|_\Y \lesssim \|\mc K_- f\|_{L^\infty(0,1)}\|\rho\|_{L^1(0,1)}\lesssim \|f\|_\X \]
and analogously, $\|\mc K_- f\|_\Y \lesssim \|f\|_\Y$.

In order to bound the operator $\mc{K}_+$ we note that, as before, we have 
$|\xi-\eta|\gtrsim \xi+\eta$ on $\mathrm{supp}(\varphi_{[0,\infty)^2
\backslash [0,1]^2})$ and thus, by Theorem \ref{thm:K0} and
Lemma \ref{lem:lgxi} we infer $|K_+(\xi,\eta)|\leq C_N \langle \xi\rangle^{-N}\langle \eta \rangle^{-N}$
for any $N$.
This yields the three stated estimates for $\mc K_+$ as well.
\end{proof}

\subsection{Estimates for the operators $\mc{K}_{dc}$ and $\mc{K}_{cd}$}

It is now an easy exercise to conclude the respective boundedness for the remaining operators
$\mc{K}_{dc}$ and $\mc{K}_{cd}$.

\begin{lemma}
\label{lem:Kcd}
Let $\alpha\geq 0$, $\delta>0$ and $1<p<\infty$ be as in Proposition \ref{prop:KccX}.
Then we have the bounds
\[ \|\mc{K}_{cd}a\|_{X^{p,\frac18}_\delta}\lesssim |a|,\quad |\mc{K}_{dc}f|\lesssim \|f\|_\X \]
as well as
\[ \|\mc{K}_{cd}a\|_{Y^{p,\frac18}}\lesssim |a|,\quad |\mc{K}_{dc}g|\lesssim \|g\|_\Y \]
for all $a \in \C$ and $f \in \X$, $g \in \Y$.
\end{lemma}

\begin{proof}
Recall that
\begin{align*}\mc{K}_{cd}a(\xi)&=\frac{a}{\|\phi(\cdot,\xi_d)\|_{L^2(0,\infty)}^2} \int_0^\infty
\phi(R,\xi)[R\partial_R-1] \phi(R,\xi_d)dR \\
\mc{K}_{dc}f&=\int_0^\infty \int_0^\infty \phi(R,\xi_d)[R\partial_R\phi(R,\eta)-
2\eta \partial_\eta \phi(R,\eta)]f(\eta)\rho(\eta)d\eta dR. \\
\end{align*}
According to Lemma \ref{lem:Jostlgxi} and Corollary \ref{cor:phi} we obtain $|\mc{K}_{cd}a(\xi)|\lesssim
|a|\langle \xi \rangle^{-\frac32}$ by performing two integrations by parts.
This already shows $\|\mc{K}_{cd}a\|_{X^{p,\frac18}_\delta}\lesssim |a|$
and $\|\mc{K}_{cd}a\|_{Y^{p,\frac18}}\lesssim |a|$ as claimed.
Furthermore, the operator $\mc{K}_{dc}$ has a similar (in fact, better behaved) kernel as $\mc{K}_{cc}$ and
we conclude $|\mc{K}_{dc}f|\lesssim \|f\|_\X$ as well as $|\mc K_{dc}g|\lesssim \|g\|_\Y$.
\end{proof}

We summarize our results.

\begin{corollary}
\label{cor:K}
Let $\delta>0$ be small and assume $p \in (1,\infty)$ be so large that $p'(1-\delta)<1$.
Then the operator $\mc K$ satisfies the estimates
\begin{align*}
\|\mc K (a,f)\|_{\C \times X^{p,\frac18}_\delta}&\lesssim \|(a,f)\|_{\C \times X^{p,\frac18}_\delta} \\
\|\mc K (a,f)\|_{\C \times Y^{p,\frac18}}&\lesssim \|(a,f)\|_{\C \times X^{p,\frac18}_\delta} \\
\|\mc K (a,g)\|_{\C \times Y^{p,\frac18}}&\lesssim \|(a,g)\|_{\C \times Y^{p,\frac18}}
\end{align*}
for all $a \in \C$, $f \in X^{p,\frac18}_\delta$ and $g\in Y^{p,\frac18}$.
\end{corollary}

\begin{proof}
This is the content of Proposition \ref{prop:KccX} and Lemma \ref{lem:Kcd}.
\end{proof}

\subsection{Estimates for $[\mc{A},\mc{K}]$}
It remains to estimate the commutator $[\mc A,\mc K]$.
To this end recall that
\[ \mc A=\left (\begin{array}{cc} 0 & 0 \\ 0 & \mc A_c \end{array} \right ) \]
with $\mc A_cf(\xi)=-[2\xi\partial_\xi +\frac52 +\tfrac{\xi \rho'(\xi)}{\rho(\xi)}]f(\xi)$.
Thus, the commutator reads
\[ [\mc A,\mc K]=\left ( \begin{array}{cc}
0 & \mc{K}_{dc}\mc A_c \\
-\mc A_c\mc K_{cd} & [\mc A_c, \mc K_{cc}] \end{array} \right ). \]
Obviously, the most complicated contribution comes from $[\mc A_c, \mc K_{cc}]$.
Recall that $\mc K_{cc}$ is a continuous map from $C^\infty_c(0,\infty)$ to the space of distributions
$\mc D'(0,\infty)$.
Consequently, $[\mc A_c, \mc K_{cc}]: C^\infty_c(0,\infty)\to \mc{D}'(0,\infty)$ is well-defined 
and continuous.

\begin{proposition}
\label{prop:KccAcc}
Let $\alpha,\delta,p$ be as in Proposition \ref{prop:KccX}. 
Then $[\mc A_c,\mc K_{cc}]$ satisfies the bounds
\begin{align*} 
\|[\mc A_c, \mc K_{cc}]f\|_\X &\lesssim \|f\|_\X \\
\|[\mc A_c, \mc K_{cc}]f\|_\Y &\lesssim \|f\|_\X \\
\|[\mc A_c, \mc K_{cc}]f\|_\Y &\lesssim \|f\|_\Y. 
\end{align*}
\end{proposition}

\begin{proof}
In \cite{KST09}, Proposition 5.2 it is shown that $[\mc K_{cc},\mc A_c]$ has a similar
kernel as $\mc K_{cc}$. Thus, the verification of the stated bound consists of a repetition of 
the above arguments that led to Proposition \ref{prop:KccX}.
\end{proof}

Finally, we bound the operators $\mc A_c\mc K_{cd}$ and $\mc K_{dc} \mc A_c$.

\begin{lemma}
\label{lem:AccKcd}
Let $\alpha,\delta,p$ be as in Proposition \ref{prop:KccX}.
Then we have the bounds
\[ \|\mc A_c \mc K_{cd}a\|_{X^{p,\frac18}_\delta}\lesssim |a|,\quad
|\mc K_{dc} \mc A_cf|\lesssim \|f\|_\X \]
as well as
\[ \|\mc A_c \mc K_{cd}a\|_{Y^{p,\frac18}}\lesssim |a|,\quad
|\mc K_{dc} \mc A_cg|\lesssim \|g\|_\Y \]
for all $a\in \C$ and $f\in \X$, $g\in \Y$.
\end{lemma}

\begin{proof}
We start with $\mc A_c \mc K_{cd}$.
If $0<\xi\ll 1$ we have $|\xi\partial_\xi \phi(R,\xi)|\lesssim \langle R \rangle$
for all $R>0$ by Lemmas \ref{lem:Phi}, \ref{lem:Jost} and Corollary \ref{cor:phi}.
This implies $|\mc A_c\mc K_{cd}a(\xi)|\lesssim 1$ since $\phi(R,\xi_d)$ decays
exponentially as $R \to \infty$.
On the other hand, if $\xi\gtrsim 1$, we obtain by integration by parts the estimate 
$|\mc A_c \mc K_{cd}a(\xi)|\lesssim \langle \xi \rangle^{-\frac32}$ as in the proof of Lemma
\ref{lem:Kcd}. This shows $\|\mc A_c \mc K_{cd}a\|_{X^{p,\frac18}_\delta}\lesssim |a|$
and $\|\mc A_c \mc K_{cd}a\|_{Y^{p,\frac18}}\lesssim |a|$.

Furthermore, from \cite{KST09}, Theorem 5.1 we have the representation
\[ \mc K_{dc} \mc A_cf=\int_0^\infty \tilde{K}_d(\eta)f(\eta)d\eta \]
with $\tilde{K}_d$ bounded and rapidly decreasing.
This yields
\begin{align*} |\mc K_{dc}\mc A_cf|&\lesssim 
\||\cdot|^{-\frac12+\delta}|\tilde{K}_d|^\frac12\|_{L^{p'}(0,\infty)}\||\cdot|^{\frac12-\delta}
|\tilde{K}_d|^\frac12 f\|_{L^p(0,\infty)}
\lesssim \||\cdot|^{\frac12-\delta}
|\tilde{K}_d|^\frac12 f\|_{L^p(0,\infty)}
\end{align*}
since $p'(-\frac12+\delta)>p'(-1+\delta)>-1$ by assumption and this yields
$ |\mc K_{dc}\mc A_cf|\lesssim \|f\|_\X$ as well as $|\mc K_{dc}\mc A_cg|\lesssim \|g\|_\Y$.
\end{proof}

By putting together Proposition \ref{prop:KccAcc} and Lemma \ref{lem:AccKcd} we arrive at
the desired result.

\begin{corollary}
\label{cor:AK}
Let $\delta>0$ be small and assume $p \in (1,\infty)$ be so large that $p'(1-\delta)<1$.
Then the operator $[\mc A, \mc K]$ satisfies the estimates
\begin{align*}
\|[\mc A,\mc K] (a,f)\|_{\C \times X^{p,\frac18}_\delta}&\lesssim \|(a,f)\|_{\C \times X^{p,\frac18}_\delta} \\
\|[\mc A, \mc K] (a,f)\|_{\C \times Y^{p,\frac18}}&\lesssim \|(a,f)\|_{\C \times X^{p,\frac18}_\delta} \\
\|[\mc A, \mc K] (a,g)\|_{\C \times Y^{p,\frac18}}&\lesssim \|(a,g)\|_{\C \times Y^{p,\frac18}}
\end{align*}
for all $a \in \C$, $f \in X^{p,\frac18}_\delta$ and $g\in Y^{p,\frac18}$.
\end{corollary}

\subsection{Construction of an exact solution in the forward lightcone}

We solve the main equation \eqref{eq:sysFourier} by a contraction argument.
To this end, we write $\mc{H}:=\mathrm{diag}(\mc{H}_d,\mc{H}_c)$ for the solution operator of the 
transport equation \eqref{eq:xd}, \eqref{eq:x}.
Furthermore, we set
\begin{equation}
\label{eq:Contract}
\Phi(x_d,x):=\mc{H}\left [\mc{N}(x_d,x)+\mc{R}_\nu\mc B_\nu(x_d,x)+\mc{T}_\nu(x_d,x)+\hat{E}_2\right ]
\end{equation}
where
\begin{align*}
\mc{R}_\nu \mc B_\nu \left (\begin{array}{c}x_d \\ x \end{array} \right )(\tau,\xi)&:=-\beta_\nu(\tau)(2\mc K
+1)(\partial_\tau+\beta_\nu(\tau)\mc A)\left (\begin{array}{c}x_d(\tau) \\ x(\tau,\cdot) \end{array} \right )(\xi) \\
\mc T_\nu \left (\begin{array}{c}x_d \\ x \end{array} \right )(\tau,\xi)&:=-\beta_\nu(\tau)^2
\left (\mc K^2+[\mc A,\mc K]+\mc K+\tfrac{\beta_\nu'(\tau)}{\beta_\nu(\tau)^2}\mc K\right )
\left (\begin{array}{c}x_d(\tau) \\ x(\tau,\cdot) \end{array} \right )(\xi) \\
\hat{E}_2(\tau,\xi)&:=\left (\begin{array}{c}\hat{e}_2(\tau,\xi_d)\\ \hat{e}_2(\tau,\xi) \end{array} \right )
\end{align*}
and 
\[ \mc B_\nu \left (\begin{array}{c}x_d \\ x \end{array} \right )(\tau,\xi):=
(\partial_\tau+\beta_\nu(\tau)\mc A)\left (\begin{array}{c}x_d(\tau) \\ x(\tau,\cdot) \end{array} \right )(\xi) \]
Thus, solutions of Eq.~\eqref{eq:sysFourier} correspond to fixed points of $\Phi$.
In order to simplify notation, we define the following Banach space where we run our fixed 
point argument.

\begin{definition}
\label{eq:defX}
Fix a $\delta$ with $2|\frac{1}{\nu}-1|<\delta<\frac18$ and a (large) 
$p \in(2,\infty)$ such that $p'(1-\delta+2|\frac{1}{\nu}-1|)<1$. Then, for 
$x_d: [\tau_0,\infty)\to\R$ and $x: [\tau_0,\infty)\times [0,\infty)\to \R$, where $\tau_0>0$, 
we set
\[ \|(x_d,x)\|_{\mc X^{\tau_0,\nu}}:=\|(x_d,x)\|_{L^{\infty,3-2\delta}_{\tau_0}\times L^{\infty,2-4\delta}_{\tau_0}
X^{p,\frac18}_\delta}+\|\mc B_\nu (x_d,x)\|_{L^{\infty,3-2\delta}_{\tau_0}\times L^{\infty,2-2\delta}_{\tau_0}
Y^{p,\frac18}}. \]
The Banach space of the respective functions is denoted by $\mc X^{\tau_0,\nu}$.
\end{definition}

\begin{theorem}
\label{thm:Contract}
Let $\tau_0 \geq 1$ be sufficiently large and $\nu$ be sufficiently close to $1$.
Then $\Phi$ (as defined in Eq.~\eqref{eq:Contract}) maps the unit ball in $\mc X^{\tau_0,\nu}$ to itself 
and is contractive, i.e.,
\[ \|\Phi(x_d,x)-\Phi(y_d,y)\|_{\mc X^{\tau_0,\nu}}\leq \tfrac12 \|(x_d,x)-(y_d,y)\|_{\mc X^{\tau_0,\nu}} \]
for all $(x_d,x), (y_d,y) \in \mc X^{\tau_0,\nu}$ with $\|(x_d,x)\|_{\mc X^{\tau_0,\nu}}\leq 1,
\|(y_d,y)\|_{\mc X^{\tau_0,\nu}}\leq 1$.
As a consequence, there exists a unique fixed point of $\Phi$ (and hence, a solution of 
Eq.~\eqref{eq:sysFourier}) which belongs to $\mc{X}^{\tau_0,\nu}$ (in fact, to the unit ball
in $\mc X^{\tau_0,\nu}$).
\end{theorem}

\begin{proof}
We consider each term on the right-hand side of Eq.~\eqref{eq:Contract} separately.
\begin{enumerate}
\item According to Lemma \ref{lem:Fe2} we have $\hat{E}_2 \in L^{\infty,3-\delta}_{\tau_0}
\times L^{\infty,3-\delta}_{\tau_0}Y^{p,\frac18}$ and thus, we find 
\[ 
\|\mc H \hat{E}_2\|_{L^{\infty,3-2\delta}_{\tau_0}\times
L^{\infty,2-4\delta}_{\tau_0}X^{p,\frac18}_\delta}\lesssim 
\tau_0^{-\delta}\|\hat{E}_2\|_{L^{\infty,3-\delta}_{\tau_0}
\times L^{\infty,3-\delta}_{\tau_0}Y^{p,\frac18}} \]
and also
\[ \|\mc B_\nu \mc H \hat{E}_2\|_{L^{\infty,3-2\delta}_{\tau_0}\times
L^{\infty,2-2\delta}_{\tau_0}Y^{p,\frac18}}\lesssim 
\tau_0^{-\delta}\|\hat{E}_2\|_{L^{\infty,3-\delta}_{\tau_0}
\times L^{\infty,3-\delta}_{\tau_0}Y^{p,\frac18}} \]
by Proposition \ref{prop:H} and Lemma \ref{lem:Hd}.
Hence, since $\tau_0\geq 1$ is assumed large, we infer
$\|\mc H \hat E_2\|_{\mc X^{\tau_0,\nu}}\leq \frac18$.

\item By Corollaries \ref{cor:K} and \ref{cor:AK} we infer 
\[ \|\mc T_\nu (x_d,x)\|_{L^{\infty,4-4\delta}_{\tau_0}\times L^{\infty,4-4\delta}_{\tau_0}Y^{p,\frac18}}
\lesssim \|(x_d,x)\|_{L^{\infty,3-2\delta}_{\tau_0}\times L^{\infty,2-4\delta}_{\tau_0}X^{p,\frac18}_\delta} \]
and by Proposition \ref{prop:H} and Lemma \ref{lem:Hd} this implies
\[ \|\mc H \mc T_\nu (x_d,x)\|_{L^{\infty,3-2\delta}_{\tau_0}\times L^{\infty,2-4\delta}_{\tau_0}X^{p,\frac18}}
\lesssim \tau_0^{-1+2\delta}\|(x_d,x)\|_{L^{\infty,3-2\delta}_{\tau_0}\times L^{\infty,2-4\delta}_{\tau_0}X^{p,\frac18}_\delta} \]
as well as
\[\|\mc B_\nu \mc H \mc T_\nu (x_d,x)\|_{L^{\infty,3-2\delta}_{\tau_0}\times L^{\infty,2-2\delta}_{\tau_0}Y^{p,\frac18}}
\lesssim \tau_0^{-1+2\delta}\|(x_d,x)\|_{L^{\infty,3-2\delta}_{\tau_0}\times L^{\infty,2-4\delta}_{\tau_0}X^{p,\frac18}_\delta}. \]
We obtain $\|\mc{H}\mc T_\nu (x_d,x)\|_{\mc{X}^{\tau_0,\nu}}\leq \frac18 \|(x_d,x)\|_{\mc{X}^{\tau_0,\nu}}$
provided that $\tau_0\geq 1$ is sufficiently large.

\item From Corollary \ref{cor:K} we obtain
\[ \|\mc R_\nu \mc B_\nu (x_d,x)\|_{L^{\infty,3-2\delta}_{\tau_0}\times L^{\infty,3-2\delta}_{\tau_0}Y^{p,\frac18}}
\lesssim |\tfrac{1}{\nu}-1| \|\mc B_\nu (x_d,x)\|_{L^{\infty,3-2\delta}_{\tau_0}\times L^{\infty,2-2\delta}_{\tau_0}Y^{p,\frac18}} \]
by recalling that $\beta_\nu(\tau)=-(\frac{1}{\nu}-1)\tau^{-1}$.
By Proposition \ref{prop:H} and Lemma \ref{lem:Hd} this yields
\[ \|\mc H \mc R_\nu \mc B_\nu (x_d,x)\|_{L^{\infty,3-2\delta}_{\tau_0}\times L^{\infty,2-4\delta}_{\tau_0}X^{p,\frac18}_\delta}
\lesssim |\tfrac{1}{\nu}-1| \|\mc B_\nu (x_d,x)\|_{L^{\infty,3-2\delta}_{\tau_0}\times L^{\infty,2-2\delta}_{\tau_0}Y^{p,\frac18}}
\] 
as well as
\[ \|\mc B_\nu \mc H \mc R_\nu \mc B_\nu (x_d,x)\|_{L^{\infty,3-2\delta}_{\tau_0}\times L^{\infty,2-2\delta}_{\tau_0}Y^{p,\frac18}}
\lesssim |\tfrac{1}{\nu}-1| \|\mc B_\nu (x_d,x)\|_{L^{\infty,3-2\delta}_{\tau_0}\times L^{\infty,2-2\delta}_{\tau_0}Y^{p,\frac18}}.
\]
Consequently, if $\nu$ is sufficiently close to $1$, we infer 
$\|\mc H \mc R_\nu \mc B_\nu (x_d,x)\|_{\mc X^{\tau_0,\nu}}
\leq \frac18 \|(x_d,x)\|_{\mc X^{\tau_0,\nu}}$.

\item Finally, for the nonlinearity we infer from Lemma \ref{lem:R} that
\begin{align*} 
\| \mc N(x_d,x)&-\mc N(y_d,y)\|_{L^{\infty,4}_{\tau_0}\times L^{\infty,3-2\delta}_{\tau_0}Y^{p,\frac18}} \\
&\lesssim \tau_0^{-\frac14+2\delta}\|(x_d,x)-(y_d,y)\|_{
L^{\infty,3-2\delta}_{\tau_0}\times L^{\infty,2-4\delta}_{\tau_0}X^{p,\frac18}_\delta}. 
\end{align*}
for all $(x_d,x)$, $(y_d,y)$ in the unit ball in $\mc X^{\tau_0,\nu}$.
As before, since $\tau_0$ is large, Proposition \ref{prop:H} and Lemma \ref{lem:Hd}
yield $\| \mc H \mc N(x_d,x)-\mc H\mc N(y_d,y)\|_{\mc X^{\tau_0,\nu}}
\leq \frac18 \|(x_d,x)-(y_d,y)\|_{\mc{X}^{\tau_0,\nu}}$ and by recalling that $\mc N(0)=0$ this also
implies
$\|\mc H \mc N(x_d,x)\|_{\mc X^{\tau_0,\nu}}\leq \frac18 \|(x_d,x)\|_{\mc X^{\tau_0,\nu}}$.
\end{enumerate}
The claim now follows by the contraction mapping principle.
\end{proof}

\subsection{Transformation to the physical space}

Recall that our original intention was to solve Eq.~\eqref{eq:eps}, given by
\begin{equation}
\label{eq:eps2}
\Box \varepsilon+5u_0^4 \varepsilon=5(u_0^4-u_2^4)\varepsilon-N(u_2,\varepsilon)-e_2,
\end{equation}
where $u_0(t,r)=\lambda(t)^\frac12 W(\lambda(t)r)$ and the functions $u_2$ and $e_2$ are the
approximate solution and the error 
constructed in Section \ref{sec:approx}. However, $u_2$ and $e_2$ are only defined in the
forward lightcone 
\[ K_{t_0,c}^\infty=\{(t,x) \in \R \times \R^3: t\geq t_0, 0\leq |x| \leq t-c \},\quad t_0>c>0 \]
and as a consequence, we had to introduce the smooth cut-off
$\chi_c(t,r):=\chi(\frac{t-r}{c})$ where $\chi(s)=0$ for $s\leq \frac12$ and $\chi(s)=1$ for
$s\geq 1$. Then we considered the truncated equation
\begin{equation}
\label{eq:eps3}
\Box \varepsilon+5u_0^4 \varepsilon=\chi_c \left [5(u_0^4-u_2^4)\varepsilon-N(u_2,\varepsilon)-e_2\right ]
\end{equation}
instead of Eq.~\eqref{eq:eps2}.
As a consequence of Theorem \ref{thm:Contract} we now have the following result which concludes
the construction of the solution inside the forward lightcone.

\begin{lemma}
\label{lem:eps}
Let $t_0>0$ be sufficiently large and $\nu$ be sufficiently close to $1$.
Then there exists a solution $\varepsilon$ of Eq.~\eqref{eq:eps3}
with
\[ \|(\varepsilon(t,\cdot),\varepsilon_t(t,\cdot))\|_{H^1(B_{2t})\times L^2(B_{2t})}\lesssim t^{-1} \]
for all $t\geq t_0$ where $B_{2t}:=\{x \in \R^3: |x|< 2t\}$.
\end{lemma}

\begin{proof}
Recall the transformations that led from Eq.~\eqref{eq:eps3} to the main system 
Eq.~\eqref{eq:sysFourier} which
we have finally solved in Theorem \ref{thm:Contract}. Schematically, we had
\[ \varepsilon \xrightarrow{(\tau,R)} \tilde{v} \xrightarrow{\cdot R} v \xrightarrow{\mc{F}}
\left (\begin{array}{c}x_d \\ x \end{array} \right ) \]
and in order to prove the claim we simply have to undo these transformations.
The coordinates $(t,r)$ and $(\tau,R)$ are related by $\tau=\tfrac{1}{\nu}\lambda(t)t$ and 
$R=\lambda(t)r$.
Thus, let $(x_d,x)\in \mc X^{\tau_0,\nu}$ be the solution of Eq.~\eqref{eq:sysFourier} constructed
in Theorem \ref{thm:Contract} (which exists since we assume $t_0$ and hence 
$\tau_0=\tfrac{1}{\nu}\lambda(t_0)t_0$ to be
large) and set 
\begin{align*} 
\tilde{v}_-(\tau,R)&:=R^{-1}\mc{F}^{-1}\left (\begin{array}{c}x_d(\tau) \\(1-\chi) x(\tau,\cdot) \end{array}
\right)(R) \\
\tilde{v}_+(\tau,R)&:=R^{-1}\mc{F}^{-1}\left (\begin{array}{c}0 \\ \chi x(\tau,\cdot) \end{array}
\right)(R)
\end{align*}
where $\chi$ is again a smooth cut-off with $\chi(\xi)=0$ for, say, $\xi \leq \frac12$ and $\chi(\xi)=1$
for $\xi \geq 1$.
According to Lemma \ref{lem:F-1} we have $\tilde{v}_+(\tau,\cdot)\in H^\frac54(\R^3)$ (recall
that we have set $\alpha=\frac18$) and therefore, $\|\tilde{v}_+(\tau,\cdot)\|_{H^1(\R^3)}
\lesssim \|\tilde{v}_+(\tau,\cdot)\|_{H^\frac54(\R^3)}\lesssim \tau^{-2+4\delta}$ 
for some small $\delta>0$ (see Definition \ref{eq:defX}).
Furthermore, Lemma \ref{lem:F-1} yields $\|\tilde{v}_-(\tau,\cdot)\|_{\dot{H}^2(\R^3)}\lesssim
\tau^{-2+4\delta}$ and
also, $\||\cdot|\tilde{v}_-(\tau,\cdot)\|_{L^\infty(\R^3)}\lesssim \tau^{-2+4\delta}$.
Thus, we obtain
\[ \|\tilde{v}_-(\tau,\cdot)\|_{L^2(B_{2\nu\tau})}\lesssim \||\cdot|\tilde{v}_-(\tau,\cdot)\|_{L^\infty(\R^3)}
\|1\|_{L^2(B_{2\nu\tau})}\lesssim \tau^{-\frac32+4\delta} \]
which implies $\|\tilde{v}_-(\tau,\cdot)\|_{H^2(B_{2\nu\tau})}\lesssim \tau^{-\frac32+4\delta}$.
By setting $\tilde{v}=\tilde{v}_-+\tilde{v}_+$ we infer 
$\|\tilde{v}(\tau,\cdot)\|_{H^1(B_{2\nu\tau})}\lesssim \tau^{-\frac32+4\delta}$ and thus,
with $\varepsilon(t,r)=\tilde{v}(\tfrac{1}{\nu}\lambda(t)t,\lambda(t)r)$, we obtain
$\|\varepsilon(t,\cdot)\|_{H^1(B_{2t})}\lesssim t^{-1}$ since $\nu$ is assumed to be close to $1$.

For the time derivative recall that $\partial_t=\tilde{\lambda}(\tau)[\partial_\tau+\beta_\nu(\tau)R\partial_R]$ 
and thus,
\[ \partial_t \varepsilon(t,r)=\tilde{\lambda}(\tau)R^{-1}[\partial_\tau+\beta_\nu(\tau)
(R\partial_R-1)] v(\tau,R). \]
By the transference identity Eq.~\eqref{eq:K} we have
\[ R^{-1}[\partial_\tau+\beta_\nu(\tau)
(R\partial_R-1)] v(\tau,R)=R^{-1}\mc F^{-1}\left (
[\partial_\tau+\beta_\nu(\tau)(\mc A+\mc K)]\left (\begin{array}{c} x_d(\tau) \\
x(\tau,\cdot) \end{array} \right) \right ). \]
By Theorem \ref{thm:Contract}, Lemma \ref{lem:H2alpha}, Corollary \ref{cor:K} and the definition
of the space $Y^{p,\frac18}$ we therefore obtain the desired 
$\|\partial_t \varepsilon(t,\cdot)\|_{L^2(B_{2t})}\lesssim t^{-1}$.
\end{proof}

%%%%%%%%%%%%%%%%%%%%%
\section{Extraction of initial data}

Let $\chi_c(t,r)$ be the smooth localizer to the truncated cone which
is defined by $\chi_c(t,r)=\chi(\frac{t-r}{c})$ where $\chi$ is a fixed smooth cut-off with
$\chi(s)=0$ for $s\leq \frac12$ and $\chi(s)=1$ for $s\geq 1$.
Furthermore, we set $\theta(t,r):=1-\chi(\frac{r}{2t})$, i.e., $\theta(t,r)=1$ if $r\leq t$ and
$\theta(t,r)=0$ if $r\geq 2t$.
As a consequence of Lemma \ref{lem:eps} there exists a function $u_c$ of the form
\begin{equation}
\label{eq:solu}
u_c(t,r)=\lambda(t)^\frac12 W(\lambda(t)r)+\theta(t,r)[v_0(t,r)+v_1^g(t,r)+\varepsilon(t,r)]
+\chi_c(t,r)v_1^b(t,r),
\end{equation}
where $v_0$ and $v_1^g:=v_{11}^g+v_{12}^g$, $v_1^b:=v_{11}^b+v_{12}^b$ are from Lemmas \ref{lem:e1} and \ref{lem:v1}, and
$\Box u_c(t,r)+u_c(t,r)^5=0$ provided that $(t,r) \in K_{t_0,c}^\infty$ with $t_0>0$ sufficiently
large.

\subsection{Energy estimates}
As a first step we establish energy bounds for the solution $u_c$.

\begin{lemma}
\label{lem:boundE}
Let $t_0\geq 1$ be sufficiently large and suppose $\nu$ is sufficiently close to $1$. 
Then we have the bounds
\begin{align*} 
\|\chi_c(t,\cdot) \partial_t u_c(t,\cdot)\|_{L^2(\R^3)}&\lesssim 
[\lambda(t)t]^{-\frac12}+c^{-\frac{\nu}{2}} \\
\|\chi_c(t,\cdot) \partial_t W_{\lambda(t)}\|_{L^2(\R^3)}&\lesssim [\lambda(t)t]^{-\frac12} \\
\|\nabla [u_c(t,\cdot)-W_{\lambda(t)}]\|_{L^2(\R^3)}
&\lesssim [\lambda(t)t]^{-\frac12} +c^{-\frac{\nu}{2}}
\end{align*}
for all $t\geq t_0> 2c\geq 1$
where, as before, $W_{\lambda(t)}(r)=\lambda(t)^\frac12 W(\lambda(t)r)$.
\end{lemma}

\begin{proof}
We consider each constituent of $u_c$ separately.
\begin{enumerate}
\item 
For the time derivative of $W_{\lambda(t)}$ we have
\[ \partial_t [\lambda(t)^\frac12 W(\lambda(t)r)]
=\tfrac12 \lambda(t)^{-\frac12}\lambda'(t)W(\lambda(t)r)+\lambda(t)^\frac12 \lambda'(t)rW'(\lambda(t)r)\]
and, since
\[ \int_0^{t}|W(\lambda(t)r)|^2 r^2 dr=\lambda(t)^{-3}\int_0^{\lambda(t)t}|W(r)|^2 r^2 dr
\lesssim \lambda(t)^{-2}t, \]
we infer 
\[ \|\chi_c(t,\cdot)\partial_t W_{\lambda(t)}\|_{L^2(\R^3)}
\lesssim \lambda(t)^{-\frac32}\lambda'(t)t^\frac12
\simeq  t^{-\frac12 \nu}=[\lambda(t)t]^{-\frac12}. \]

\item According to Lemma \ref{lem:e1} we have the estimates 
\begin{align*}
|v_0(t,r)|&\lesssim \lambda(t)^{-\frac12}t^{-2}r \\
|\partial_t v_0(t,r)|+|\partial_r v_0(t,r)|&\lesssim \lambda(t)^{-\frac12}[t^{-3}r+t^{-2}]\lesssim
\lambda(t)^{-\frac12}t^{-2}
\end{align*}
for all $t\geq t_0$ and
$0\leq r\leq 2t$. Furthermore, note that $|\partial_t \theta(t,r)|\simeq rt^{-2}|\chi'(\frac{r}{2t})|$ 
and $|\nabla \theta(t,r)|\simeq t^{-1}|\chi'(\frac{r}{2t})|$. 
Thus, we obtain
\[ \|\chi_c(t,\cdot) \partial_t \theta(t,\cdot) v_0(t,\cdot)\|_{L^2(\R^3)}^2\lesssim
\lambda(t)^{-1}t^{-8}\int_{t}^{2t}r^6 dr\lesssim [\lambda(t)t]^{-1} \]
and analogously, $\|\nabla \theta(t,\cdot)v_0(t,\cdot)\|_{L^2(\R^3)}\lesssim [\lambda(t)t]^{-\frac12}$.
Similarly, we infer 
\begin{align*}
\|\chi_c(t,\cdot)\theta(t,\cdot)\partial_t v_0(t,\cdot)\|_{L^2(\R^3)}^2
+\|\theta(t,\cdot)\nabla v_0(t,\cdot)\|_{L^2(\R^3)}^2&\lesssim 
\lambda(t)^{-1}t^{-4}\int_0^{2t}r^2 dr \\
&\lesssim [\lambda(t)t]^{-1}.
\end{align*}

\item \label{item:3}
Note that $\theta(t,r)v_1^g(t,r)+\chi_c(t,r)v_1^b(t,r)=v_1(t,r)$ provided $r\leq \frac12 t$.
Thus, in the case $r\leq \frac12 t$ we put $v_1^g$ and $v_1^b$ together and use the bounds
\begin{align*}
|v_1(t,r)|&\lesssim \lambda(t)^{-\frac12}t^{-1} \\
|\partial_t v_1(t,r)|+|\partial_r v_1^g(t,r)|&
\lesssim \lambda(t)^{-\frac12}t^{-2}
\end{align*}
from Lemma \ref{lem:v1}.
If $\frac12 t\leq r\leq 2t$ we similarly have
\begin{align*}
|v_1^g(t,r)|&\lesssim \lambda(t)^{-\frac12}t^{-1} \\
|\partial_t v_1^g(t,r)|+|\partial_r v_1^g(t,r)|&
\lesssim \lambda(t)^{-\frac12}t^{-2}
\end{align*}
by Lemma \ref{lem:v1}
and thus, these terms can be treated in the exact same fashion as $v_0$.

\item For $\varepsilon$ it suffices to invoke the bound from Lemma \ref{lem:eps} which
immediately yields
\[ \|\chi_c(t,\cdot)\partial_t [\theta(t,\cdot)\varepsilon(t,\cdot)]\|_{L^2(\R^3)}+
\|\nabla [\theta(t,\cdot)\varepsilon(t,\cdot)]\|_{L^2(\R^3)}\lesssim 
t^{-1}\lesssim [\lambda(t)t]^{-\frac12}. \]

\item Finally, we come to the most interesting contribution, the term $\chi_c v_1^b=\chi_c(v_{11}^b+v_{12}^b)$. We emphasize that the following estimates
are key to the whole construction. 
We may restrict ourselves to $r\geq \frac12 t$ since the case $r\leq \frac12 t$ is already included
in point $\eqref{item:3}$ above.
We start with $\chi_c v_{11}^b$.
Lemma \ref{lem:v1} yields
\begin{align*}
|v_{11}^b(t,r)|&\lesssim
\lambda(t)^{-\frac12}t^{-1}(1-\tfrac{r}{t})^{\frac12(1-\nu)}\\
|\partial_t v_{11}^b(t,r)|+|\partial_r v_{11}^b(t,r)|
&\lesssim \lambda(t)^{-\frac12}t^{-2}(1-\tfrac{r}{t})^{\frac12(1-\nu)-1}
\end{align*}
for all $t\geq t_0$ and $\frac12 t\leq r< t$.
Furthermore, note that 
\[ |\partial_t \chi_c(t,r)|+|\partial_r \chi_c(t,r)|\simeq \tfrac{1}{c}|\chi'(\tfrac{t-r}{c})|. \]
We infer
\begin{align*}
\|\partial_t \chi_c(t,\cdot)v_{11}^b(t,\cdot)\|_{L^2(\R^3\backslash B_{t/2})}^2&\lesssim
\tfrac{1}{c^2}\lambda(t)^{-1}t^{-2}\int_{t-c}^{t-\frac{c}{2}}(1-\tfrac{r}{t})^{1-\nu}r^2 dr \\
&=\tfrac{1}{c^2}\lambda(t)^{-1}t\int_{1-\frac{c}{t}}^{1-\frac{c}{2t}}(1-s)^{1-\nu}s^2 ds \\
&\lesssim \tfrac{1}{(2-\nu)c^2}\underbrace{\lambda(t)^{-1}t}_{t^{2-\nu}}\left (\tfrac{c}{t}\right)^{2-\nu}\lesssim c^{-\nu}
\end{align*}
and by the very same calculation we also obtain $\|\nabla \chi_c(t,\cdot)v_{11}^b(t,\cdot)\|_{L^2(\R^3)}\lesssim 
c^{-\frac{\nu}{2}}$.
If the derivative hits $v_{11}^b$ we have
\begin{align*}
\|\chi_c(t,\cdot)\partial_t v_{11}^b(t,\cdot)\|_{L^2(\R^3 \backslash B_{t/2})}^2
&\lesssim
\lambda(t)^{-1}t^{-4}\int_{\frac12 t}^{t-\frac{c}{2}}(1-\tfrac{r}{t})^{-\nu-1}r^2 dr \\
&=[\lambda(t)t]^{-1}\int_{\frac12}^{1-\frac{c}{2t}}(1-s)^{-\nu-1}s^2 ds \\
&\lesssim \tfrac{1}{\nu}t^{-\nu}\left(\tfrac{c}{2t}\right)^{-\nu}\lesssim c^{-\nu}
\end{align*}
and analogously we get $\|\chi_c(t,\cdot)\nabla v_{11}^b(t,\cdot)\|_{L^2(\R^3)}\lesssim c^{-\frac{\nu}{2}}$
as well. The term $v_{12}^b$ is handled by the exact same computations upon replacing $\nu$ by $3\nu$.
\end{enumerate}
\end{proof}

Next, we estimate the residual energy outside the (truncated) lightcone $K_{t_0,c}^\infty$.

\begin{lemma}
\label{lem:boundEout}
Under the assumptions of Lemma \ref{lem:boundE} we have the bounds
\begin{align*} 
\|\chi_c(t,\cdot)\partial_t W_{\lambda(t)}\|_{L^2(\R^3\backslash B_{t-c})}&\lesssim
[\lambda(t)t]^{-\frac12} \\
\|\nabla W_{\lambda(t)}\|_{L^2(\R^3\backslash B_{t-c})}
&\lesssim [\lambda(t)t]^{-\frac12}
\end{align*}
for all $t\geq t_0 > 2c\geq 2$ where $B_{t-c}=\{x\in \R^3: |x|<t-c\}$.
\end{lemma}

\begin{proof}
The first bound follows from Lemma \ref{lem:boundE}. 
Thus, it suffices to note that
\begin{align*}
\|\nabla W_{\lambda(t)}\|_{L^2(\R^3\backslash B_{t-c})}^2&\simeq 
\lambda(t)\int_{t-c}^\infty |\partial_r W(\lambda(t)r)|^2r^2 dr \\
&\lesssim \lambda(t)^3\int_{t-c}^\infty
\lambda(t)^{-4}r^{-2} dr 
\lesssim \lambda(t)^{-1}(t-c)^{-1} \\
&\lesssim [\lambda(t)t]^{-1}.
\end{align*}
\end{proof}

To conclude the energy bounds, we provide estimates for the $L^6$ norms.

\begin{corollary}
\label{cor:L6}
Under the assumptions of Lemma \ref{lem:boundE} we have
\begin{align*} 
\|u_c(t,\cdot)-W_{\lambda(t)}\|_{L^6(\R^3)}&\lesssim [\lambda(t)t]^{-\frac12}+c^{-\frac{\nu}{2}} \\
\|W_{\lambda(t)}\|_{L^6(\R^3\backslash B_{t-c})}&\lesssim [\lambda(t)t]^{-\frac12}
\end{align*}
for all $t\geq t_0>2c\geq 2$.
\end{corollary}

\begin{proof}
The first assertion is an immediate consequence of Lemma \ref{lem:boundE}
and the Sobolev embedding $\dot{H}^1(\R^3)\hookrightarrow L^6(\R^3)$.
For the second bound we calculate explicitly
\begin{align*}
\|W_{\lambda(t)}\|_{L^6(\R^3\backslash B_{t-c})}^6\lesssim \lambda(t)^3 \int_{t-c}^\infty
\lambda(t)^{-6}r^{-6}r^2dr\lesssim \lambda(t)^{-3}(t-c)^{-3}\lesssim [\lambda(t)t]^{-3}.
\end{align*}
\end{proof}

\subsection{Extension of the solution to the whole space}
Lemmas \ref{lem:boundE} and \ref{lem:boundEout} show that the bulk of the energy of $u_c$ is concentrated
on the soliton inside the truncated lightcone.
Now we pick a sequence of times $(T_n)$ with $T_n \geq t_0$ and $T_n \to \infty$ as $n\to \infty$.
Then we consider the sequence of Cauchy data $(f_c^n,g_c^n)$ given by
\begin{equation}
\label{eq:fg}
\begin{aligned}
f_c^n(r)&=u_c(T_n,r) \\
g_c^n(r)&=\chi_c(T_n,r)\partial_t u_c(t,r)|_{t=T_n}.
\end{aligned}
\end{equation}
Since $\chi_c(T_n,r)\equiv 1$ for $r\leq T_n-c$, we have 
\footnote{Here and in the following we employ the convenient abbreviation $u[t]=(u(t,\cdot),\partial_t u(t,\cdot))$.}
$(f_c^n,g_c^n)=u_c[T_n]$ on $B_{T_n-c}$.
As a consequence of Lemma \ref{lem:boundE}, the sequence $(f_c^n,g_c^n)$ is uniformly bounded
in $\dot{H}^1\times L^2(\R^3)$ for all $n \in \N$.
Now we solve the equation backwards in time with data $(f_c^n,g_c^n)$ at $t=T_n$.
For the following it is useful to introduce the notation
\[ K_{t_0,c}^{T_n}:=\{(t,x) \in \R\times \R^3: t_0\leq t\leq T_n, |x|\leq t-c\}. \]
Furthermore, for $U\subset \R^3$ open we set
\begin{align*} 
\mc{E}_U(f,g)&:=\frac12 \int_U (|\nabla f(x)|^2+|g(x)|^2)dx-\frac16 \int_U |f(x)|^6 dx \\
&=\tfrac12 \|(f,g)\|_{\dot{H^1}\times L^2(U)}^2-\tfrac16 \|f\|_{L^6(U)}^6.
\end{align*}
Thus, $\mc{E}_{\R^3}$ is the energy functional associated with the focusing quintic wave equation.

\begin{lemma}
\label{lem:extend}
Let $t_0\geq 1$ and $c\geq 1$ be sufficiently large and assume $t_0\geq 2c$. 
Then, for any $n\in\N$, there exists an energy class solution $u^{(T_n)}$ of
\[ \left \{ \begin{array}{l}
\Box u^{(T_n)}(t,x)+u^{(T_n)}(t,x)^5=0,\quad (t,x) \in [t_0,T_n] \times \R^3 \\
u^{(T_n)}[T_n]=(f_c^n,g_c^n) \end{array} \right . \]
which satisfies $u^{(T_n)}=u_c$ on $K_{t_0,c}^{T_n}$. 
Furthermore, 
\[ \|u^{(T_n)}[t_0]\|_{\dot{H}^1\times L^2(\R^3\backslash B_{t_0-c})}\to 0 \]
as $t_0,c \to \infty$, uniformly in $n$.
\end{lemma}

\begin{proof}
Recall that, given data in $\dot{H}^1\times L^2(\R^3)$, 
the Cauchy problem for the quintic wave equation can be solved locally in time.  
Furthermore, if the data are \emph{small} in $\dot{H}^1\times L^2(\R^3)$, the corresponding
solution exists globally in time, see \cite{Pec84} or \cite{Sog08}, p.~142, Theorem 3.1.
Given any $\delta>0$, we have $\|(f_c^n,g_c^n)\|_{\dot{H}^1(\R^3)\times L^2(\R^3)}\leq 
\|W\|_{\dot{H}^1(\R^3)}+\delta$ for all $n\in \N$ if we assume $t_0$ and $c$ to be sufficiently 
large (see Lemma \ref{lem:boundE}).
Thus, we infer the existence of $u^{(T_n)}$ on $(T_n^*,T_n]\times \R^3$
with some $T_n^*<T_n$ and we assume $T_n^*$ to be minimal with this property.
Furthermore, the map 
\[ t \mapsto \|u^{(T_n)}[t]\|_{\dot{H}^1\times L^2(\R^3)}: (T_n^*,T_n]\to \R \]
is continuous (\cite{Sog08}, p.~142, Theorem 3.1).
If $T_n^*\leq t_0$ we are done.
Thus, assume that $T_n^*>t_0$.
By causality it is clear that $u^{(T_n)}=u_c$ on $({K}_{T_n^*,c}^{T_n})^\circ$, the interior
of the truncated lightcone.
Furthermore, by Lemma \ref{lem:boundE} and Corollary \ref{cor:L6}, the data satisfy
\begin{align*} 
\|(f_c^n,g_c^n)-(W_{\lambda(T_n)},0)\|_{\dot{H}^1 \times L^2(\R^3)}&\lesssim \delta \\
\|f_c^n-W_{\lambda(t)}\|_{L^6(\R^3)}&\lesssim \delta
\end{align*}
and, by using that $\|\nabla W_{\lambda(t)}\|_{L^2(\R^3)}=\|\nabla W\|_{L^2(\R^3)}$ and analogously
for $\|W_{\lambda(t)}\|_{L^6(\R^3)}$, 
this implies $|\mc E_{\R^3}(f_c^n,g_c^n)-\mc E_{\R^3} (W,0)|\lesssim \delta^2+\delta^6\lesssim \delta^2$ for the total energy.
In other words, the bulk of the total energy is concentrated on the soliton.
By conservation of energy we infer that this has to hold for all times, i.e., 
$|\mc E_{\R^3} (u^{(T_n)}[t])-\mc E_{\R^3} (W,0)|\lesssim \delta^2$
for all $t \in (T_n^*, T_n]$.
Since $u^{(T_n)}=u_c$ on $(K_{T^*_n,c}^{T_n})^\circ$, we infer from Lemma \ref{lem:boundE} 
and Corollary \ref{cor:L6} the bounds
\begin{align*} 
\|u^{(T_n)}[t]-(W_{\lambda(t)},0)\|_{\dot{H}^1\times L^2(B_{t-c})}&\lesssim \delta \\
\|u^{(T_n)}(t,\cdot)-W_{\lambda(t)}\|_{L^6(B_{t-c})}&\lesssim \delta
\end{align*}
which imply $|\mc E_{B_{t-c}}(u^{(T_n)}[t])-\mc E_{B_{t-c}}(W_{\lambda(t)},0)|\lesssim \delta^2$.
Moreover, by Lemma \ref{lem:boundEout} and Corollary \ref{cor:L6} we have
$|\mc E_{\R^3\backslash B_{t-c}}(W,0)|\lesssim \delta^2$ and thus,
$|\mc E_{\R^3\backslash B_{t-c}}(u^{(T_n)}[t])|\lesssim \delta^2$ which shows that the
total energy of the solution $u^{(T_n)}$ outside the truncated lightcone stays small for
all times.
By a (slight modification of) the Sobolev inequality we infer 
$\|u^{(T_n)}(t,\cdot)\|_{L^6(\R^3\backslash B_{t-c})}\lesssim \|\nabla u^{(T_n)}(t,\cdot)\|_{L^2(\R^3\backslash
B_{t-c})}$ with an implicit constant that is independent of $t$.
This estimate implies
\begin{align*} \delta^2&\gtrsim \mc E_{\R^3\backslash B_{t-c}}(u^{(T_n)}[t])
=\tfrac12 \|u^{(T_n)}[t]\|_{\dot{H}^1\times L^2(\R^3\backslash B_{t-c})}^2-\tfrac16
\|u^{(T_n)}(t,\cdot)\|_{L^6(\R^3\backslash B_{t-c})}^6 \\
&\geq \tfrac12 \|u^{(T_n)}[t]\|_{\dot{H}^1\times L^2(\R^3\backslash B_{t-c})}^2
\left [1-C\|u^{(T_n)}[t]\|_{\dot{H}^1\times L^2(\R^3\backslash B_{t-c})}^4 \right ]
\end{align*}
for all $t\in (T_n^*,T_n]$.
Initially, at $t=T_n$, we have 
\[ \|u^{(T_n)}[T_n]\|_{\dot{H}^1\times L^2(\R^3\backslash B_{t-c})}
=\|(f_c^n,g_c^n)\|_{\dot{H}^1\times L^2(\R^3\backslash B_{t-c})}\leq \delta \]
by Lemma \ref{lem:boundEout} and therefore, we must have
\[ \|u^{(T_n)}[t]\|_{\dot{H}^1\times L^2(\R^3\backslash B_{t-c})}^2\lesssim \delta^2 \]
for all $t \in (T_n^*,T_n]$ (provided $\delta>0$ is sufficiently small) 
since the map $t \mapsto u^{(T_n)}[t]$ is continuous from $(T_n^*,T_n]$
to $\dot{H}^1\times L^2(\R^3\backslash B_{t-c})$ by a classical $\frac{\varepsilon}{2}$ argument.
We conclude that not only the total energy but also the \emph{kinetic} energy of $u^{(T_n)}$
stays small outside the truncated cone.
Consequently, the small data global existence result
allows us to extend the solution beyond time $T_n^*$ which contradicts the minimality of $T_n^*$.
Thus, we must have $T_n^*\leq t_0$ and the Lemma is proved.
\end{proof}

\subsection{The Bahouri-G\'erard decomposition}

Our idea now is to consider the sequence of data $u^{(T_n)}[t_0]$ and attempt to extract
a limit as $n\to \infty$.
In effect, we shall not be able to do so, but we shall nonetheless be able to construct
new initial data $u_*[t_0]$ resulting in an energy class solution $u_*(t,x)$ defined on all of
$[t_0,\infty) \times \R^3$ with the property that
\[ u_*=u_c \mbox{ on }K_{t_0,c}^\infty. \]
In order to achieve this, we apply the celebrated Bahouri-G\'erard decomposition \cite{BG99}.
From now on we always assume $t_0$ and $c$ to be sufficiently large with $t_0\geq 2c$.
Note also that no space translations are necessary in the following lemma since all functions
are in fact radial.
Furthermore, it is convenient to introduce the notation 
\[ u^{\lambda,t_0}(t,x):=\lambda^{-\frac12}u\left (\frac{t-t_0}{\lambda}, \frac{x}{\lambda}\right ). \]

\begin{lemma}[Linear profile decomposition]
\label{lem:BG}
Consider the sequence $(u^{(T_n)}[t_0])_{n\in\N}$ of Cauchy data for the 
solutions constructed in Lemma \ref{lem:extend}.
Then, upon passing to a subsequence,
there exists a sequence $(V_i)_{i\in\N}$ of (fixed) radial free waves 
\footnote{A ``free wave'' is a function $v:\R\times\R^3\to\R$ such that the map 
$t \mapsto \|v[t]\|_{\dot{H}^1\times L^2(\R^3)}$ is continuous (in particular, 
$\|v[t]\|_{\dot{H}^1\times L^2(\R^3)}\lesssim 1$ for any $t\in\R$) and
$v(t,\cdot)=\cos(t|\nabla|)v(0,\cdot)+|\nabla|^{-1}\sin(t|\nabla|)\partial_1 v(0,\cdot)$, i.e., 
$\Box v=0$ in the weak sense.},
such that
\begin{equation}
\label{eq:BG}
u^{(T_n)}[t_0]=\sum_{i=1}^A V_i^{\lambda_n^i,t_n^i}[t_0]
+W^{nA}[t_0]
\end{equation}
for all $n, A\in \N$ where $(t_n^i)_{n\in\N}$ and $(\lambda_n^i)_{n\in\N}$ 
are suitable sequences of times and positive 
scaling factors, respectively, that satisfy
\begin{equation}
\label{eq:orth}
\left |\log \left ( \frac{\lambda_n^i}{\lambda_n^j}\right )\right |+\frac{|t_n^i-t_n^j|}{\lambda_n^i}
\to \infty,\quad i\not= j 
\end{equation}
as $n \to \infty$, and $W^{nA}$ is a free wave with the property
\[ \lim_{A\to \infty} \limsup_{n\to \infty}\|W^{nA}\|_{L^\infty(\R) L^6(\R^3)}=0. \]
Furthermore, for any $A\in\N$, we have asymptotic orthogonality in the sense that
\begin{align*} 
\left ( \left . V_i^{\lambda_n^i,t_n^i}[t_0]\right | V_j^{\lambda_n^j,t_n^j}[t_0] \right )_{\dot{H}^1
\times L^2(\R^3)} &\to 0, \quad 1\leq i,j\leq A,\quad i\not= j \\
\left ( \left . V_i^{\lambda_n^i,t_n^i}[t_0]\right | W^{nA}[t_0] \right )_{\dot{H}^1
\times L^2(\R^3)} &\to 0,\quad 1\leq i\leq A \\
\end{align*}
as $n\to\infty$.
\end{lemma}

\begin{proof}
From Lemma \ref{lem:boundE} we deduce 
\[ \|u^{(T_n)}[t_0]\|_{\dot{H}^1\times L^2(B_{t_0-c})}
=\|u_c[t_0]\|_{\dot{H}^1 \times L^2(B_{t_0-c})}\lesssim 1 \]
and in the proof of Lemma \ref{lem:extend} above we had
\[ \|u^{(T_n)}[t_0]\|_{\dot{H}^1\times L^2(\R^3\backslash B_{t_0-c})}\lesssim \delta \]
for all $n\in\N$.
Thus, we infer
\[\sup_{n\in\N}\|u^{(T_n)}[t_0]\|_{\dot{H}^1\times L^2(\R^3)}\lesssim 1\]
and
everything follows by the main theorem in \cite{BG99} and the remark on p.~159.
\end{proof}

\begin{remark}
\label{rem:locorth}
We will also use a kind of ``localized'' orthogonality which can be derived as follows. Suppose
$(v_n[t_0]|w_n[t_0])_{\dot H^1\times L^2(\R^3)}=o(1)$ as $n\to\infty$ and the energy of 
$v_n[t_0]$ concentrates in $U\subset \R^3$ as $n\to \infty$, i.e.,
\[ \|v_n[t_0]\|_{\dot H^1\times L^2(\R^3)}=
\|v_n[t_0]\|_{\dot H^1\times L^2(U)}+o(1) \]
as $n\to\infty$, or, in other words,
\[ \|v_n[t_0]\|_{\dot H^1\times L^2(\R^3 \backslash U)} \to 0 \quad (n\to\infty).\]
Of course, we also assume that $\|v_n[t_0]\|_{\dot H^1\times L^2(\R^3)}, \|w_n[t_0]\|_{\dot H^1\times L^2(\R^3)}\lesssim 1$
for all $n$.
Then
we have the localized orthogonality
\begin{align*} 
(v_n[t_0]|w_n[t_0])_{\dot H^1\times L^2(U)}
&=(v_n[t_0]|w_n[t_0])_{\dot H^1\times L^2(\R^3)}-(v_n[t_0]|w_n[t_0])_{\dot H^1\times L^2(\R^3\backslash U)}=o(1)
\end{align*}
since
\[ \big |(v_n[t_0]|w_n[t_0])_{\dot H^1\times L^2(\R^3\backslash U)}\big |\leq \|v_n[t_0]\|_{\dot H^1\times L^2(\R^3\backslash U)}
\|w_n[t_0]\|_{\dot H^1\times L^2(\R^3\backslash U)}\to 0 \]
as $n\to\infty$.
\end{remark}

A triple $(V_i,(\lambda_n^i),(t_n^i))$ like in Lemma \ref{lem:BG} is called a (concentration) profile.
As a first step we now show that certain profiles with scaling factors tending to zero
cannot exist. 
Heuristically speaking, such profiles are excluded by the fact that they would concentrate at the origin
as $n\to \infty$ but near $x=0$, $u^{(T_n)}[t_0]$ equals $u_c[t_0]$ and is thus independent
of $n$.
In order to make the argument rigorous, one has to show that the concentration effect cannot be
``cancelled'' by the error term $W^{nA}$.
This is a consequence of the asymptotic orthogonality of the profiles.
Before coming to that, however, we introduce another notion.
We call a profile $(V_i,(\lambda_n^i),(t_n^i))$ \emph{bounded}, if 
$\lambda_n^i\simeq 1$ and $|t_0-t_n^i|\lesssim 1$ for all $n\in\N$.
Otherwise, the profile is called \emph{unbounded}.
Note that by condition \eqref{eq:orth} there exists at most one (nonzero) bounded profile
in the decomposition Eq.~\eqref{eq:BG}.

\begin{lemma}
\label{lem:ZA}
Consider the decomposition given in Lemma \ref{lem:BG} and suppose there exists a profile
$(V_i, (\lambda_n^i), (t_n^i))$ with
$\lambda_n^i \to 0$ as $n\to \infty$ and $\frac{|t_0-t_n^i|}{\lambda_n^i}\lesssim 1$ for all $n\in\N$. 
Then $V_i=0$.
\end{lemma}

\begin{proof}
Fix $A\in\N$ and
let $V_b$ be the unique bounded profile ($V_b$ might be zero in which
case we set $\lambda_n^b=1$ and $t_n^b=0$ for all $n\in \N$).
First, we claim that for any given $\varepsilon>0$ we can find a $\delta>0$ such that
\begin{equation}
\label{eq:epsdelta}
\left \|u^{(T_n)}[t_0]-V_b^{\lambda_n^b,t_n^b}[t_0] 
\right \|_{\dot{H}^1\times L^2(B_{\delta})}<\varepsilon
\end{equation}
for all $n\in\N$.
Indeed, $u^{(T_n)}[t_0]|_{B_\delta}$ is independent of $n$ for small enough $\delta$ since
by construction we have
$u^{(T_n)}[t_0]=u_c[t_0]$ on $B_{t_0-c}$, see Lemma \ref{lem:extend}.
This shows that $\|u^{(T_n)}[t_0]\|_{\dot{H}^1\times L^2(B_\delta)}\to 0$ as $\delta \to 0$,
uniformly in $n$.
Furthermore,
by scaling and energy conservation we have
\[ \|V_i^{\lambda_n^i,t_n^i}[t_0]\|_{\dot{H}^1\times L^2(\R^3)}=\|V_i[t_0]\|_{\dot{H}^1\times L^2(\R^3)}\]
for any profile $V_i$ (bounded or unbounded).
Since $\lambda_n^b\simeq 1$ and $|t_n^b|\lesssim 1$ for all $n\in\N$, we obtain 
by the continuity of $t\mapsto \|V_b[t]\|_{\dot{H}^1\times L^2(\R^3)}$ the bound
$\|V_b^{\lambda_n^b,t_n^b}[t_0]\|_{\dot{H}^1\times L^2(B_\delta)}
<\tfrac{\varepsilon}{2}$
for all $n\in \N$ provided $\delta>0$ is sufficiently small.
Consequently, the triangle inequality yields the claim \eqref{eq:epsdelta}.

Note that by Eq.~\eqref{eq:BG}, Eq.~\eqref{eq:epsdelta} is equivalent to
\[ \left \|\sum_{i=1,i\not=b}^A V_i^{\lambda_n^i,t_n^i}[t_0]+W^{nA}[t_0] \right \|_{\dot{H}^1\times
L^2(B_\delta)}<\varepsilon. \]
We write $i\in Z_A$ iff $\lambda_n^i\to 0$ as $n\to\infty$
and $\frac{|t_0-t_n^i|}{\lambda_n^i}\lesssim 1$ for all $n\in\N$.
Now observe that, for $i\in Z_A$,
\begin{align*} 
\|V_i^{\lambda_n^i,t_n^i}[t_0]\|_{\dot{H}^1\times L^2(\R^3\backslash B_\delta)}
&\simeq \left \|\partial_1 V_i\left (\frac{t_0-t_n^i}{\lambda_n^i},\cdot\right )\right 
\|_{L^2(\R^3\backslash B_{\delta/\lambda_n^i})} \\
&\quad+\left \|\nabla V_i\left (\frac{t_0-t_n^i}{\lambda_n^i},\cdot\right )\right 
\|_{L^2(\R^3\backslash B_{\delta/\lambda_n^i})} \to 0
\end{align*}
as $n\to\infty$ by the continuity of $t\mapsto \|V_i[t]\|_{\dot{H}^1\times L^2(\R^3)}$.
For brevity we write
\begin{align*}
v_n[t_0]:=\sum_{i \in Z_A} V_i^{\lambda_n^i,t_n^i}[t_0],\quad \quad w_n[t_0]:=\sum_{i \notin Z_A,i\not= b}
V_i^{\lambda_n^i,t_n^i}[t_0]+W^{nA}[t_0].
\end{align*}
By the pairwise orthogonality of the profiles and the triangle inequality we obtain
\begin{align*}
\varepsilon&>\left \|\sum_{i=1,i\not=b}^A V_i^{\lambda_n^i,t_n^i}[t_0]+W^{nA}[t_0]
\right \|_{\dot{H}^1\times
L^2(B_\delta)}^2=\|v_n[t_0]+w_n[t_0]\|_{\dot H^1\times L^2(B_\delta)}^2 \\
&=\|v_n[t_0]+w_n[t_0]\|_{\dot H^1\times L^2(\R^3)}^2-\|v_n[t_0]+w_n[t_0]\|_{\dot H^1\times L^2(\R^3\backslash B_\delta)}^2 \\
&\geq\|v_n[t_0]\|_{\dot H^1\times L^2(\R^3)}^2+\|w_n[t_0]\|_{\dot H^1\times L^2(\R^3)}^2
-\|v_n[t_0]\|_{\dot H^1\times L^2(\R^3\backslash B_\delta)}^2-\|w_n[t_0]\|_{\dot H^1\times L^2(\R^3\backslash B_\delta)}^2 \\
&\quad -2\|v_n[t_0]\|_{\dot H^1\times L^2(\R^3\backslash B_\delta)}
\|w_n[t_0]\|_{\dot H^1\times L^2(\R^3\backslash B_\delta)}+o(1) \\
&= \sum_{i\in Z_A}\|V_i[t_0]\|_{\dot{H}^1\times L^2(\R^3)}^2 
+\left \|\sum_{i\notin Z_A,i\not=b} V_i^{\lambda_n^i,t_n^i}[t_0]+W^{nA}[t_0]
\right \|_{\dot{H}^1\times
L^2(B_\delta)}^2+o(1)
\end{align*}
as $n\to\infty$.
Consequently, $\|V_i[t_0]\|_{\dot{H}^1\times L^2(\R^3)}<\varepsilon$
for all $i\in Z_A$ provided $n$ is sufficiently large and,
since $\varepsilon>0$ and $A\in\N$ were arbitrary, this yields the claim.
\end{proof}

Our next goal is to prove that the energy of \emph{all} unbounded profiles is small.
As a preparation for this we need the following elementary observation which is just an instance
of the strong Huygens' principle. As always, it suffices to consider the radial
case for our purposes.

\begin{lemma}[Strong Huygens' principle]
\label{lem:Huygens}
Let $v:\R\times \R^3\to\R$ be a (radial) free wave and set 
\[ A(R_1,R_2):=\{x\in\R^3: R_1<|x|<R_2\}. \]
Then we have the estimate
\[ \|v[t]\|_{\dot{H}^1\times L^2(B_R)}\lesssim \|v[0]\|_{\dot{H}^1\times L^2(A(|t|-R,|t|+R))}
+\|v(0,\cdot)\|_{L^6(A(|t|-R,|t|+R))}\]
for all $t\in \R$ and all $R>0$ provided $|t|\geq R$.
\end{lemma}

\begin{proof}
Note first that by the Sobolev embedding $\dot{H}^1(\R^3)\hookrightarrow L^6(\R^3)$ we may assume
$v(t,\cdot)\in L^6(\R^3)$ for all $t\in\R$.
Furthermore, by the time reflection symmetry it suffices to consider the case $t\geq R$. Since 
$v$ is radial, it is given explicitly by
d'Alembert's formula
\[ v(t,r)=\frac{1}{2r}\left [(t+r)f(t+r)-(t-r)f(t-r)+\int_{t-r}^{t+r}sg(s)ds \right ] \]
for $0\leq r\leq t$ where $(f,g)=v[0]$.
Based on this formula it is straightforward to prove the claimed estimate by using Hardy's and 
H\"older's inequalities.
\end{proof}

We obtain a simple corollary which applies to certain concentration profiles in the
Bahouri-G\'erard decomposition.

\begin{corollary}
\label{cor:leave}
Suppose $v:\R\times \R^3\to\R$ is a (radial) free wave and let $(\lambda_n)_{n\in\N}$, $(t_n)_{n\in\N}$ be sequences of
(positive) scaling factors and times, respectively.
If 
\begin{itemize}
\item $\lambda_n\to\infty$ 
\item or $\frac{|t_0-t_n|-R}{\lambda_n}\to\infty$ as $n\to \infty$
\end{itemize}
then
\[ \|v^{\lambda_n,t_n}[t_0]\|_{\dot{H}^1\times L^2(B_R)}\to 0 \]
as $n\to\infty$ for any (fixed) $R>0$.
\end{corollary}

\begin{proof}
Note that by scaling we have
\begin{equation}
\label{eq:vt0}
 \|v^{\lambda_n,t_n}[t_0]\|_{\dot{H}^1\times L^2(B_R)}^2=\left \|\partial_1 
v \left (\frac{t_0-t_n}{\lambda_n},\cdot\right )\right \|_{L^2(B_{R/\lambda_n})}^2
+\left \|\nabla 
v \left (\frac{t_0-t_n}{\lambda_n},\cdot\right )\right \|_{L^2(B_{R/\lambda_n})}^2.
\end{equation}
First, we consider the case $\lambda_n\to \infty$.
If $\frac{|t_0-t_n|}{\lambda_n}\lesssim 1$ for all $n\in\N$ then
the continuity of $t\mapsto \|v[t]\|_{\dot{H}^1\times L^2(\R^3)}$ and Eq.~\eqref{eq:vt0}
show that $\|v^{\lambda_n,t_n}[t_0]\|_{\dot{H}^1\times L^2(B_R)}\to 0$ as $n\to \infty$.
On the other hand, if $\frac{|t_0-t_n|}{\lambda_n}\to\infty$, we must have $|t_0-t_n|\to\infty$ and
thus, $|t_0-t_n|/\lambda_n\geq R/\lambda_n$ for large $n$. Consequently, 
Lemma \ref{lem:Huygens} yields
the claim.
The second case is a direct consequence of
Lemma \ref{lem:Huygens}.
\end{proof}

Now we can show the aforementioned smallness of the unbounded profiles.

\begin{lemma}
\label{lem:ubsmall}
Let $\varepsilon>0$.
For the energy of any unbounded profile $(V_i, (\lambda_n^i), (t_n^i))$
in the decomposition Eq.~\eqref{eq:BG} we have the bound
\[ \|V_i[t_0]\|_{\dot{H}^1\times L^2(\R^3)}\leq 
\|u^{(T_n)}[t_0]\|_{\dot{H}^1\times L^2(\R^3\backslash B_{t_0-c})}+\varepsilon \]
as $n\to \infty$.
\end{lemma}

\begin{remark}
By Lemma \ref{lem:extend} we see that the energy of the unbounded profiles can be assumed to be
arbitrarily small provided $t_0$ and $c$ are chosen large enough.
\end{remark}

\begin{proof}[Proof of Lemma \ref{lem:ubsmall}]
By Lemma \ref{lem:ZA} the claim holds trivially for those unbounded profiles $V_i$ where
$\lambda_n^i\to 0$ as $n\to\infty$ and $\frac{|t_0-t_n^i|}{\lambda_n^i}\lesssim 1$ for all $n\in\N$.
Furthermore, if the sequences $(\lambda_n^i)$ and $(t_n^i)$ satisfy the hypothesis of 
Corollary \ref{cor:leave}, we infer
\[ 
\|V_i[t_0]\|_{\dot{H}^1\times L^2(\R^3)}=\|V_i[t_0]\|_{\dot{H}^1\times L^2(\R^3\backslash 
B_{t_0-c})}+o(1) \]
as $n\to\infty$ and the claim follows by the orthogonality of the profiles stated in 
Lemma \ref{lem:BG} and Remark \ref{rem:locorth}.

It remains to study those profiles where $\frac{|t_0-t_n^i|-(t_0-c)}{\lambda_n^i}\lesssim 1$ and $\lambda_n^i\lesssim 1$.
Since the profiles in question are unbounded, we must have $\lambda_n^i\to 0$.
Consequently, $|t_0-t_n^i|$ is bounded and after selecting a subsequence, we may assume that $t_n^i\to t^i$. 
Let $\varepsilon>0$ be given. Then we can find an $R>0$ such that
\[ \|V_i[0]\|_{\dot H^1\times L^2(\R^3)}=\|V_i[0]\|_{\dot H^1\times L^2(B_R)}+\varepsilon, \]
i.e., the energy of $V_i[0]$ is essentially contained in $B_R$.
Taking into account the fact that $V$ is a free wave, we infer by the strong Huygens' principle that the 
energy of $V$ at time $\frac{t_0-t_n^i}{\lambda_n^i}$ (for large enough $n$) is essentially contained in 
\[ \Omega_n:=\left \{x\in\R^3: \tfrac{|t_0-t_n^i|}{\lambda_n^i}-R\leq |x|\leq \tfrac{|t_0-t_n^i|}{\lambda_n^i}+R\right \}, \] 
cf.~Lemma \ref{lem:Huygens}.
More precisely, we have
\[ \left \|V_i\left [\tfrac{t_0-t_n^i}{\lambda_n^i}\right ]\right \|_{\dot H^1\times L^2(\R^3)}
=\left \|V_i\left [\tfrac{t_0-t_n^i}{\lambda_n^i}\right ]\right \|_{\dot H^1\times L^2(\Omega_n)}+\varepsilon \]
as $n \to \infty$.
By rescaling this is equivalent to
\[ \|V_i^{\lambda_n^i,t_n^i}[t_0]\|_{\dot H^1\times L^2(\R^3)}=
\|V_i^{\lambda_n^i,t_n^i}[t_0]\|_{\dot H^1\times L^2(\tilde \Omega_n)}+\varepsilon \]
where $\tilde \Omega_n=\{x\in\R^3: |t_0-t_n|-\lambda_n^i R\leq |x|\leq |t_0-t_n|+\lambda_n^i R\}$.
Thus, as $n\to\infty$, $V_i^{\lambda_n^i,t_n^i}[t_0]$ concentrates at $r^i=|t_0-t^i|$.
If $r^i<t_0-c$, we apply the argument from Lemma \ref{lem:ZA} to conclude that $V_i=0$ (the logic being
that $V_i^{\lambda_n^i,t_n^i}[t_0]$ cannot concentrate inside of $B_{t_0-c}$ because there, $u^{(T_n)}[t_0]$
equals $u_c[t_0]$).
If $r^i\geq t_0-c$, it follows from Lemma \ref{lem:extend} and the orthogonality of Lemma \ref{lem:BG}, see also
Remark \ref{rem:locorth}, that the energy of the profile $V_i$ is small.
In any case, we arrive at the desired conclusion.
\end{proof}

An immediate consequence of Lemma \ref{lem:ubsmall} is the fact that the bulk of the energy
of $u^{(T_n)}[t_0]$ is concentrated on the bounded profile $V_b$ (in particular, $V_b$ is nonzero).

\subsection{The nonlinear profile decomposition}

After extracting a subsequence we have $\lambda_n^b\to \lambda^b$, $t_n^b\to t^b$ as $n\to\infty$
where $\lambda^b$ and $t^b$ are some finite numbers.
Now we define new initial data by
\[ (f,g):=\lim_{n\to\infty}V_b^{\lambda_n^b,t_n^b}[t_0]=V_b^{\lambda^b,t^b}[t_0] \]
with convergence in $\dot{H}^1\times L^2(\R^3)$.
By the local existence theory \cite{Sog08} we obtain a time $t_*>t_0$ and an energy class solution $u_*$ satisfying
\begin{equation}
\label{eq:u*}
\left \{ \begin{array}{l}\Box u_*(t,x)+u_*(t,x)^5=0,\quad (t,x)\in (t_0,t_*)\times \R^3 \\
u_*[t_0]=(f,g) \end{array} \right .
\end{equation}
where we assume $t_*$ to be maximal.

To each profile $(V_i, (\lambda_n^i), (t_n^i))$ 
in the decomposition of Lemma \ref{lem:BG} there is associated a \emph{nonlinear profile}
$(U_i, (\lambda_n^i), (t_n^i))$ which is characterized by $\Box U_i+U_i^5=0$ and either
\[
\|V_i[t]-U_i[t]\|_{\dot{H}^1\times L^2(\R^3)}\to 0 \mbox{ as }t\to \pm \infty \]
if $\frac{t_0-t_n^i}{\lambda_n^i}\to \pm \infty$ in the limit $n\to\infty$ or
\[ U_i[t_0]=V_i[t_0] \]
in the case $\frac{|t_0-t_n^i|}{\lambda_n^i}\lesssim 1$ for all $n\in \N$.
The existence of the nonlinear profiles is a consequence of the small data 
scattering theory 
\cite{Pec84}, cf.~also \cite{KM08}.
Moreover, since the energy of all unbounded profiles $V_i$ is small, it follows that the $U_i$
exist globally (and scatter) provided $i\not= b$ and by a continuity argument as in the proof of Lemma \ref{lem:extend}
we may assume $\|U_i[t]\|_{\dot{H}^1\times L^2(\R^3)}$ to be small for all $t\in\R$.
In the case of $U_b$ we have at least existence for small times.
By the symmetries of the equation we see that, for all $n\in\N$, 
$U_b^{\lambda_n^b,t_n^b}$ is a solution
with data $U_b^{\lambda_n^b,t_n^b}[t_0]=V_b^{\lambda_n^b,t_n^b}[t_0]$ and thus, by the local
well-posedness we infer
\begin{equation}
\label{eq:Ubu*}
\|U_b^{\lambda_n^b,t_n^b}[t]-u_*[t]\|_{\dot{H}^1\times L^2(\R^3)}\to 0\quad (n\to\infty) 
\end{equation}
for any $t \in [t_0,t_*)$.
Now we want to compare the solution $u_*$ with $u^{(T_n)}$.
The following lemma yields a representation of the \emph{nonlinear} evolution of the decomposition
Eq.~\eqref{eq:BG}.

\begin{lemma}
\label{lem:BGnl}
Let $t_1 \in [t_0,t_*)$.
Then there exists an $n_0\in \N$ such that, for all $n\geq n_0$, 
the nonlinear profiles $U_i^{\lambda_n^i, t_n^i}$
associated to the decomposition in Lemma \ref{lem:BG} exist on $[t_0,t_1]\times \R^3$ and
\[ u^{(T_n)}(t,x)=\sum_{i=1}^A U_i^{\lambda_n^i,t_n^i}(t,x)+W^{nA}(t,x)+R^{nA}(t,x) \]
for all $A\in \N$ and $t\in [t_0,t_1]$ where $W^{nA}$ is the free wave from Lemma \ref{lem:BG}.
Furthermore, the error $R^{nA}$ satisfies 
\[ \lim_{A\to\infty}\limsup_{n\to\infty}\|R^{nA}[\cdot]\|_{L^\infty([t_0,t_1]) \dot{H}^1\times L^2(\R^3)}=0. \]
\end{lemma}

\begin{proof}
This is (the second part of) the main theorem in \cite{BG99}.
Note that in \cite{BG99} the result is actually proved for the defocusing critical wave equation.
However, once the existence of the nonlinear profiles $U_i$ is established, one checks that the argument
in Section IV of
\cite{BG99} is in fact insensitive to the sign of the nonlinearity.
As already mentioned, the $U_i$ for $i\not=b$ exist globally and scatter since the energy of the corresponding
$V_i$ is small (in fact, arbitrarily small if we choose $t_0$ and $c$ large enough).
In the case of $U_b$ 
it follows from Eq.~\eqref{eq:Ubu*} that we may assume the existence
of $U_b^{\lambda_n^b,t_n^b}$ on $[t_0,t_1]\times \R^3$ provided $n$ is sufficiently large.
\end{proof}

The final step in the construction consists of showing that $u_*=u_c$ on $(K_{t_0,c}^{t_*})^\circ$.
In particular, we thereby obtain that $u_*$ can be extended beyond time $t_*$ and thus extends
globally.

\begin{lemma}
The solution $u_*$ extends to all of $[t_0,\infty)\times \R^3$ and satisfies
\[ u_*=u_c \mbox{ on }K_{t_0,c}^\infty. \]
\end{lemma}

\begin{proof}
Let $t \in [t_0,t_*)$ and choose $n$ so large that Lemma \ref{lem:BGnl} applies.
For $N\in\N$ denote by $P_{<N}$ the Littlewood-Paley projector to 
frequencies $\{\xi \in \R^3: |\xi|\leq N\}$.
Given $\varepsilon>0$ we choose $N$ so large that
\[ \|u_*(t,\cdot)-P_{<N}u_*(t,\cdot)\|_{L^6(B_{t-c})}<\varepsilon. \]
Consider the decomposition in Lemma \ref{lem:BGnl}.
For an unbounded profile $(U_i,(\lambda_n^i),(t_n^i))$ with $\lambda_n^i \to 0$ as $n\to\infty$ we
clearly have $\|P_{<N}U_i^{\lambda_n^i,t_n^i}(t,\cdot)\|_{L^6(B_{t-c})} \to 0$ as $n \to \infty$.
Furthermore, for a profile $(U_i, (\lambda_n^i),(t_n^i))$ with $\lambda_n^i \to \infty$ we obtain
$\|U_i^{\lambda_n^i,t_n^i}(t,\cdot)\|_{L^6(B_{t-c})} \to 0$ as $n\to\infty$. 
Finally, if $\lambda_n^i \simeq 1$ and $|t_n^i| \to \infty$ we infer
$\|U_i^{\lambda_n^i,t_n^i}(t,\cdot)\|_{L^6(B_{t-c})} \to 0$ as $n\to\infty$
by the small data scattering theory and Huygens' principle (recall that the $\dot{H}^1\times L^2(\R^3)$
norm of all unbounded profiles is small).
By choosing $A$ sufficiently large we can also achieve
\[ \|W^{nA}(t,\cdot)\|_{L^6(\R^3)}+\|R^{nA}(t,\cdot)\|_{L^6(\R^3)}<\varepsilon \]
by Lemma \ref{lem:BGnl}.
These estimates and the decomposition from Lemma \ref{lem:BGnl} imply 
\[ \|u_*(t,\cdot)-u^{(T_n)}(t,\cdot)\|_{L^6(B_{t-c})}\lesssim \varepsilon \]
provided $n$ is chosen large enough.
Since $\varepsilon>0$ was arbitrary and $u^{(T_n)}=u_c$ on $K_{t_0,c}^t$ by Lemma \ref{lem:extend},
we obtain $u_*=u_c$ on $K_{t_0,c}^t$.
Furthermore, the solution can be continued since outside of $K_{t_0,c}^t$ the kinetic energy
stays small for all times by a continuity argument as in the proof of Lemma \ref{lem:extend} and
inside of $K_{t_0,c}^t$ the solution $u_*$ equals $u_c$ which exists on $K_{t_0,c}^\infty$.
\end{proof}

\bibliography{nsc}
\bibliographystyle{plain}

\end{document}